\documentclass[a4paper,reqno,10pt]{amsart}
\usepackage[utf8]{inputenc}
\usepackage{hyperref}
\usepackage{relsize}
\usepackage{amssymb,amsmath,dsfont,tikz-cd}

\usepackage{dsfont, mathtools}
\usepackage[all]{xy}
\usepackage{etoolbox}

\usepackage{mathrsfs}

\newtheorem{theorem}{Theorem}[section]
\newtheorem{thmx}{Theorem}

\newtheorem{proposition}[theorem]{Proposition}
\newtheorem{lemma}[theorem]{Lemma}
\newtheorem{conjecture}[theorem]{Conjecture}
\newtheorem{corollary}[theorem]{Corollary}

\theoremstyle{remark}
\newtheorem{remark}[theorem]{Remark}
\newtheorem{warning}[theorem]{Warning}
\newtheorem{example}[theorem]{Example}

\theoremstyle{definition}
\newtheorem{definition}[theorem]{Definition}
\newtheorem{convention}[theorem]{Convention}

\usepackage{bbm}
\renewcommand{\AA}{\mathbb{A}}
\newcommand{\Boxtimes}{\mathlarger{\mathlarger{\boxtimes}}}

\newcommand{\SSN}{\mathcal{SSN}}
\newcommand{\N}{\mathcal{N}}
\newcommand{\SN}{\mathcal{SN}}
\newcommand{\Sp}{\mathcal{S}}

\newcommand{\DD}{\mathbb{D}}
\newcommand{\kac}{\mathtt{a}}
\newcommand{\dbl}{(\!(}
\newcommand{\dbr}{)\!)}
\DeclareMathOperator{\B}{B\!}
\DeclareMathOperator{\Sh}{Sh}
\DeclareMathOperator{\UEA}{U}
\DeclareMathOperator{\rank}{rank}
\DeclareMathOperator{\BM}{BM}
\DeclareMathOperator{\TS}{\mathtt{TS}}
\newcommand{\dd}{\mathbf{d}}
\newcommand{\cdott}{\!\cdot\!}
\newcommand{\ee}{\mathbf{e}}
\newcommand{\p}{\mathbf{\mathtt{JH}}}
\newcommand{\EE}{\mathrm{E}}

\newcommand{\ff}{\mathbf{f}}

\DeclareMathOperator{\HN}{HN}
\DeclareMathOperator{\hp}{\mathtt{h}}

\newcommand{\LL}{\mathbb{L}}
\newcommand{\g}{\mathfrak{gl}}
\newcommand{\NN}{\mathbb{N}}

\newcommand{\QQ}{\mathbb{Q}}

\newcommand{\WW}{\mathcal{T}r(W)}
\newcommand{\WWW}{\mathfrak{Tr}(W)}

\newcommand{\Mst}{\mathfrak{M}}

\newcommand{\dvst}{\Lambda_\theta^{\zeta}}
\newcommand{\Msp}{\mathcal{M}}

\newcommand{\ICS}{\mathcal{IC}}

\newcommand{\DTS}{\mathcal{BPS}}

\newcommand{\phim}[1]{\phi^{\mon}_{#1}}
\newcommand{\phin}[1]{\phi_{#1}}
\newcommand{\BIG}[1]{\mathlarger{\mathlarger{\mathlarger{\mathlarger{#1}}}}}
\newcommand{\phip}[1]{\phi^{\mathfrak{p}}_{#1}}
\newcommand{\psip}[1]{\psi^{\mathfrak{p}}_{#1}}

\DeclareMathOperator{\sst}{-ss}
\DeclareMathOperator{\Log}{Log}

\DeclareMathOperator{\poi}{P}
\DeclareMathOperator{\nnilp}{\mathrm{-nilp}}
\DeclareMathOperator{\Dim}{Dim}
\DeclareMathOperator{\surj}{surj}
\DeclareMathOperator{\Hilb}{Hilb}
\DeclareMathOperator{\st}{-st}

\DeclareMathOperator{\tw}{tw}
\DeclareMathOperator{\reg}{reg}
\DeclareMathOperator{\crit}{crit}
\DeclareMathOperator{\Kk}{K_0}

\newcommand{\slope}{\varrho}

\DeclareMathOperator{\Hom}{Hom}
\DeclareMathOperator{\mon}{mon}
\DeclareMathOperator{\End}{End}
\DeclareMathOperator{\wt}{\chi_{wt}}

\DeclareMathOperator{\Exp}{Exp}
\DeclareMathOperator{\G}{GL}
\DeclareMathOperator{\MMHM}{MMHM}
\DeclareMathOperator{\MHS}{MHS}

\DeclareMathOperator{\Gr}{Gr}
\DeclareMathOperator{\Vect}{Vect}
\DeclareMathOperator{\Jor}{Jor}

\DeclareMathOperator{\Id}{Id}
\DeclareMathOperator{\rat}{rat}
\DeclareMathOperator{\lmod}{-mod}

\DeclareMathOperator{\Nak}{\mathbf{M}}
\DeclareMathOperator{\perv}{Perv}
\DeclareMathOperator{\Rep}{Rep}
\DeclareMathOperator{\supp}{supp}

\DeclareMathOperator{\Aut}{Aut}

\DeclareMathOperator{\codim}{codim}

\DeclareMathOperator{\MHM}{MHM}

\DeclareMathOperator{\DT}{BPS}
\DeclareMathOperator{\Sym}{Sym}

\DeclareMathOperator{\op}{op}
\DeclareMathOperator{\Spec}{Spec}
\DeclareMathOperator{\Gl}{GL}
\DeclareMathOperator{\Bl}{P}
\DeclareMathOperator{\PGL}{PGL}
\DeclareMathOperator{\id}{id}
\DeclareMathOperator{\Jac}{Jac}

\DeclareMathOperator{\Tr}{Tr}

\DeclareMathOperator{\pt}{pt}
\DeclareMathOperator{\Tot}{Tot}

\DeclareMathOperator{\cone}{cone}

\DeclareMathOperator{\vir}{vir}
\DeclareMathOperator{\Ob}{Ob}
\DeclareMathOperator{\Ho}{\mathcal{H}}
\DeclareMathOperator{\HO}{H}

\DeclareMathOperator{\Coha}{\mathcal{A}}

\newcommand{\Du}{\mathcal{D}^{\geq}}
\newcommand{\Dl}{\mathcal{D}^{\leq}}

\newcommand{\Db}{\mathcal{D}^{b}}
\newcommand{\Dub}{\mathcal{D}}

\newcommand{\sbrk}{\smallbreak}

\title[Integrality conjecture and cohomology of preprojective stacks]{The integrality conjecture and the cohomology of preprojective stacks}
\author{Ben Davison}

\begin{document}

\begin{abstract}
We study the Borel--Moore homology of stacks of representations of preprojective algebras $\Pi_Q$, via the study of the DT theory of the undeformed 3-Calabi--Yau completion $\Pi_Q[x]$.  Via a result on the supports of the BPS sheaves for $\Pi_Q[x]\lmod$, we prove purity of the BPS cohomology for the stack of $\Pi_Q[x]$-modules, and define BPS sheaves for stacks of $\Pi_Q$-modules.  These are mixed Hodge modules on the coarse moduli space of $\Pi_Q$-modules that control the Borel--Moore homology and geometric representation theory associated to these stacks.  We show that the hypercohomology of these objects is pure, and thus that the Borel--Moore homology of stacks of $\Pi_Q$-modules is also pure.  
\sbrk
We transport the cohomological wall-crossing and integrality theorems from DT theory to the category of $\Pi_Q$-modules.  Among these and other applications, we use our results to prove positivity of a number of ``restricted'' Kac polynomials, determine the critical cohomology of $\Hilb_n(\AA^3)$, and the Borel--Moore homology of genus one character stacks, as well as various applications to the cohomological Hall algebras associated to Borel--Moore homology of stacks of preprojective algebras, including the PBW theorem, and torsion-freeness.
\end{abstract}

\setcounter{tocdepth}{1}% to get subsubsections in toc

\let\oldtocsection=\tocsection

\let\oldtocsubsection=\tocsubsection

\let\oldtocsubsubsection=\tocsubsubsection

\renewcommand{\tocsection}[2]{\hspace{0em}\oldtocsection{#1}{#2}}
\renewcommand{\tocsubsection}[2]{\hspace{3em}\oldtocsubsection{#1}{#2}}
\renewcommand{\tocsubsubsection}[2]{\hspace{2em}\oldtocsubsubsection{#1}{#2}}

\makeatletter
\patchcmd{\@tocline}
  {\hfil}
  {\leaders\hbox{\,.\,}\hfil}
  {}{}
\makeatother

\maketitle

\tableofcontents

\section{Introduction}
\subsection{Background}
This paper concerns the Borel--Moore homology of stacks of representations of preprojective algebras $\Pi_Q$, which play a prominent role in many branches of mathematics, and which we study through the prism of cohomological Donaldson--Thomas theory and BPS cohomology.  The Borel--Moore homology of stacks of preprojective algebras occurs as the underlying vector space of the cohomological Hall algebra containing all raising operators for the cohomology of Nakajima quiver varieties \cite{Nak94,Nak98}, which themselves can be presented as certain stacks of semistable representations of preprojective algebras.  More generally, stacks of representations of preprojective algebras model the local geometry of complex 2-Calabi--Yau categories possessing good moduli spaces \cite{2CYpurity}, for example coherent sheaves on K3 and Abelian surfaces, Higgs bundles on smooth projective curves, local systems on Riemann surfaces, and moduli of semistable objects in Kuznetsov components.
\smallbreak
Via dimensional reduction we study the Borel--Moore homology of the stack of $\Pi_Q$-modules by relating it to the BPS sheaves for the stack of objects in the 3-Calabi--Yau completion $\mathscr{C}_{\Pi_Q}$ (as defined by Keller in \cite{Kel09}) of the category of $\Pi_Q$-modules.  This paper is devoted to understanding the BPS sheaves (as defined in \cite{DaMe15b}) of the 3CY categories $\mathscr{C}_{\Pi_Q}$ formed this way.  By studying these BPS sheaves and the associated BPS cohomology, we prove a number of theorems regarding the Borel--Moore homology of stacks of $\Pi_Q$-representations, Nakajima quiver varieties, stacks of coherent sheaves on surfaces, as well as vanishing cycle cohomology of $\mathrm{Hilb}_n(\mathbb{A}^3)$, and vanishing cycle cohomology of stacks of objects in $\mathscr{C}_{\Pi_Q}$.

\subsection{The role of purity}
\label{TPT}
In Donaldson--Thomas theory, as well as many of the other subjects this paper touches on, we are typically interested in \textit{motivic} invariants.  See e.g. \cite{JS08,KS} and references therein for extensive background on motivic DT theory.  This just means that we are interested in invariants $\tilde{\chi}$ of objects in a triangulated category $\mathscr{D}$ that factor through the Grothendieck group of $\mathscr{D}$, i.e. if $V'\rightarrow V\rightarrow V''$ is a distinguished triangle in $\mathscr{D}$ then we require
\begin{equation}
\label{motInv}
\tilde{\chi}(V)=\tilde{\chi}(V')+\tilde{\chi}(V'').  
\end{equation}
Alternatively, by ``motivic'', people mean invariants of varieties $X$ such that if $U\subset X$ is open, with complement $Z$, then $\tilde{\chi}(X)=\tilde{\chi}(U)+\tilde{\chi}(Z)$ (the ``cut-and-paste relation'').  The link between the two meanings is provided by the distinguished triangle $\HO_c(U,\mathbb{Q})\rightarrow \HO_c(X,\mathbb{Q})\rightarrow \HO_c(Z,\mathbb{Q})$, so that a motivic invariant in the first sense induces one in the second sense.
\smallbreak
A very basic example of a motivic invariant is the \textit{Euler characteristic} of a complex of vector spaces $\chi(V)=\sum_{i\in\mathbb{Z}}(-1)^i\dim(V^i)$.  A basic example of a \textit{non}-motivic invariant is the Poincar\'e polynomial $P(V,q)=\sum_{i\in \mathbb{Z}}\dim(V^i)q^i$; since the connecting morphisms in a long exact sequence of vector spaces may be nonzero, the Poincar\'e polynomial may not satisfy \eqref{motInv}, for example
\[
P(\HO_c(\mathbb{A}^1,\mathbb{Q}),q)=q^2\neq P(\HO_c(\mathbb{C}^*,\mathbb{Q}),q)+P(\HO_c(\pt,\mathbb{Q}),q).
\]
\subsubsection{}
In \textit{refined} Donaldson--Thomas theory the invariant $\tilde{\chi}$ that we consider is a motivic invariant that is defined via Hodge theory.  Recall that a Hodge structure on a rational vector space $V$ is the data of an ascending weight filtration $W_{\bullet}V$, along with a descending Hodge filtration $F^{\bullet}V_{\mathbb{C}}$ of the complexification, such that the Hodge filtration induces a weight $n$ Hodge structure on the $n$th piece $\Gr^W_n(V)$ of the associated graded object with respect to the weight filtration.  
\sbrk
Given $\mathcal{L}$, a cohomologically graded mixed Hodge structure, one defines its Hodge series, E series, and weight series respectively:
\begin{align*}
\hp(\mathcal{L},x,y,z)=&\sum_{a,b,c\in\mathbb{Z}}\dim(\Gr_F^{b}(\Gr^W_{b+c}(\HO^a(\mathcal{L}))))x^by^cz^a\\
\EE(\mathcal{L},x,y)=&\hp(\mathcal{L},x,y,-1)\\
\wt(\mathcal{L},q^{1/2})=&\EE(\mathcal{L},q^{1/2},q^{1/2}).
\end{align*}
Since both the E series and weight series involve an alternating sum over cohomological degrees, it is easy to see that they are motivic invariants.
\smallbreak
Enumerative theories like DT theory, and point-counting over finite fields, provide powerful tools for determining motivic invariants.  However, from a naive point of view they may produce the ``wrong kind'' of invariants, since we are often most interested in the Poincar\'e polynomial of cohomologically graded vector spaces, or might even hope to describe the cohomology itself (rather than its class in some Grothendieck group).  We say that a cohomologically graded mixed Hodge structure $\mathcal{L}$ is \textit{pure} if its $a$th cohomologically graded piece is pure of weight $a$, i.e. if $\Gr^W_b\!\HO^a(\mathcal{L})=0$ for $b\neq a$.  Our interest in pure mixed Hodge structures comes from the fact that if $\mathcal{L}$ is pure then $P(\mathcal{L},q)=\wt(\mathcal{L},q)$.  Moreover, when the Borel--Moore homology of a stack is pure we have a much better chance of being able to actually calculate it, as we will demonstrate in this paper.

\subsection{The purity theorem}
\label{pthmsec}
Let $Q$ be a quiver with vertices $Q_0$ and arrows $Q_1$.  The quiver $\overline{Q}$, which is the double of $Q$, is obtained by adding an arrow $a^*$ for every arrow of $a$, with the reverse orientation.  Then the preprojective algebra is defined as the quotient of the free path algebra of $\overline{Q}$:
\[
\Pi_Q:=\mathbb{C}\overline{Q}/\left\langle\sum_{a \in Q_1} [a,a^*]\right\rangle.
\]
%The primary purpose of this paper is to describe the mixed Hodge structure on the compactly supported cohomology of various stacks of representations of the preprojective algebra $\Pi_Q$.  
\sbrk
We define $\mathbb{N}:=\mathbb{Z}_{\geq 0}$.  Let $\dd\in\mathbb{N}^{Q_0}$ be a dimension vector for $\overline{Q}$.  Define
\[
X(\overline{Q})_{\dd}=\prod_{a\textrm{ an arrow of }\overline{Q}}\Hom(\mathbb{C}^{\dd_{\textrm{source}(a)}},\mathbb{C}^{\dd_{\textrm{target}(a)}}).
\]
This space is symplectic, via the natural isomorphism $X(\overline{Q})_{\dd}\cong \mathrm{T}^* (X(Q)_{\dd})$.  This symplectic manifold carries an action of the gauge group 
\[
\Gl_{\dd}:=\prod_{i\in Q_0}\Gl_{\dd_i}(\mathbb{C}), 
\]
with moment map
\begin{align*}
\mu_{Q,\dd}\colon &X(\overline{Q})_{\dd}\rightarrow\g_{\dd}:=\prod_{i\in Q_0} \mathfrak{gl}_{\dd_i}(\mathbb{C})\\
&\rho\mapsto \sum_{a \in Q_1} [\rho(a),\rho(a^*)].
\end{align*}
Here we identify $\mathfrak{gl}_{\dd_i}(\mathbb{C})$ with the dual vector space $\mathfrak{gl}_{\dd_i}(\mathbb{C})^{\vee}$ via the trace pairing.  The stack $\Mst(\Pi_Q)_{\dd}$ of $\Pi_Q$-representations with dimension vector $\dd$ is isomorphic to the stack-theoretic quotient $\mu_{Q,\dd}^{-1}(0)/\G_{\dd}$.  

\subsubsection{}
Our first main result is the following.

\begin{thmx}\label{purityThm}\cite[Conj.3.1]{Chicago3}
Fix a quiver $Q$, and a dimension vector $\dd\in\mathbb{N}^{Q_0}$.  Then the mixed Hodge structure on
\begin{equation}
\label{mmCo}
\HO^{\BM}(\Mst(\Pi_Q)_{\dd},\mathbb{Q})\coloneqq\HO_c(\Mst(\Pi_Q)_{\dd},\mathbb{Q})^{\vee}
\end{equation}
is pure, of Tate type.
\end{thmx}
We prove a more general version of Theorem \ref{purityThm}, concerning Borel--Moore homology of stacks of semistable $\Pi_Q$-modules: see \S \ref{conseq}, and Theorem \ref{2dstab}.
\sbrk
In Theorem \ref{purityThm} the symbol $\vee$ denotes the dual in the category of cohomologically graded mixed Hodge structures.  Purity means that Deligne's mixed Hodge structure on each cohomologically graded piece $\HO^n_c(\Mst(\Pi_Q)_{\dd},\mathbb{Q})$ is pure of weight $n$, and the statement that a cohomologically graded mixed Hodge structure $\mathcal{L}$ is of Tate type is the statement that we can write
\[
\mathcal{L}=\bigoplus_{m,n\in\mathbb{Z}} \left(\LL^{m}[n]\right)^{\oplus a_{m,n}},
\]
for some set of numbers $a_{m,n}\in\mathbb{N}$, with 
\[
\LL:=\HO_c(\mathbb{A}^1,\mathbb{Q})
\]
given the usual weight 2 pure Hodge structure, concentrated in cohomological degree 2.  Purity is the further statement that $a_{m,n}=0$ for $n\neq 0$.
\subsubsection{} 
Theorem \ref{purityThm} concerns compactly supported cohomology.  Since $\mu_{Q,\dd}^{-1}(0)$ is a cone, and hence homotopic to a point, there is an isomorphism
\begin{equation}
\label{dudPure}
\HO(\Mst(\Pi_Q)_{\dd},\mathbb{Q})\cong\HO(\B\G_{\dd},\mathbb{Q})
\end{equation}
in usual singular cohomology, and it is known that the right hand side of (\ref{dudPure}) is pure \cite{Deligne74}.  On the other hand, compactly supported cohomology is not preserved by homotopy equivalence, and the highly singular nature of $\mu^{-1}_{Q,\dd}(0)/\G_{\dd}$ means that its compactly supported cohomology is a great deal more complicated than its cohomology.  In fact, purity requires an essentially new type of argument, requiring the full force of cohomological Donaldson--Thomas theory.  In particular, outside of finite type $Q$, there is no way known (to date) of proving this purity statement without invoking the cohomological integrality theorem for the Donaldson--Thomas theory of quivers with potential, along with dimensional reduction.
\subsubsection{Categorification}
If $\mathcal{L}$ is a cohomologically graded mixed Hodge structure that is pure, of Tate type, then its isomorphism class is determined by its \\Poincar\'e/Hodge/weight/E series.  The weight series of the mixed Hodge structures (\ref{mmCo}) are described in terms of the Kac polynomials \cite{kac83} of $Q$, via the results of \cite{Hua00} and \cite{Moz11}, and so Theorem \ref{purityThm} enables us to calculate the compactly supported cohomology $\HO_c(\Mst(\Pi_Q)_{\dd},\mathbb{Q})$ itself.  Purity enables us to go beyond motivic invariants and actually determine the compactly supported cohomology of the stack $\Mst(\Pi_Q)_{\dd}$, not just its class in some Grothendieck group.
\subsubsection{Okounkov's conjecture}
Theorem \ref{purityThm} is a singular stack-theoretic cousin of the result that the cohomology of Nakajima quiver varieties is pure, with Hodge polynomial expressible as a polynomial in $xyz^2$ (this can be obtained by combining the proof of \cite[Thm.1]{Hau10} with \cite[Thm.6.1.2(3)]{HRV08}).  In fact we recover this result (Corollary \ref{oriPure}).  The purity of Nakajima quiver varieties provides one of the main motivations for the purity statement in Theorem \ref{purityThm}. 
\sbrk
In a little more detail, it is conjectured that the cohomological Hall algebra $\mathcal{A}_{\Pi_Q}$ obtained by taking the direct sum of $\HO^{\BM}(\Mst(\Pi_Q)_{\dd},\mathbb{Q})$ across all dimension vectors $\dd$ is isomorphic to the positive half of the Yangian $\mathcal{Y}_{\mathtt{MO},Q}$ constructed by Maulik and Okounkov in \cite{MaOk12}.  This in turn would imply that the graded dimensions of $\mathfrak{g}_{\mathtt{MO},Q}$ are given by Kac polynomials, as conjectured by Okounkov.  Since the algebra $\mathcal{Y}_{\mathtt{MO},Q}$ is defined as a subalgebra of the endomorphism algebra of the cohomology of Nakajima quiver varieties, the purity of $\mathcal{Y}_{\mathtt{MO},Q}$ follows from purity for these quiver varieties.  Our purity theorem provides evidence towards the conjecture that $\mathcal{A}_{\Pi_Q}\cong \mathcal{Y}_{\mathtt{MO},Q}^+$.

\subsection{From Donaldson--Thomas theory to symplectic geometry}
%We relate the study of $\HO^{\BM}(\Mst(\Pi_Q)_{\dd},\mathbb{Q})$, via dimensional reduction, to cohomological noncommutative Donaldson--Thomas theory, by which we mean the study of the cohomological Hall algebra associated to a quiver $Q$ with potential $W$ by Kontsevich and Soibelman in \cite{COHA}.  The underlying mixed Hodge structure of this object is the mixed Hodge structure on the cohomology of the vanishing cycle complex on the stack of representations for the Jacobi algebra associated to $(Q,W)$; we recall the details in Sections \ref{NandC} and \ref{DTtheory}.  
%\smallbreak
%There are numerous features of the theory that will possibly be foreign to symplectic geometers, in this subsection we explain their application to the study of the stack $\Mst(\Pi_Q)_{\dd}$.  

%\smallbreak

Consider the following general setup, of which our situation with $X(\overline{Q})_{\dd}$ being acted on by $\Gl_{\dd}$ is a special case.  Let $X$ be a complex symplectic manifold, with the affine algebraic group $G$ acting on it via a Hamiltonian action, with ($G$-equivariant) moment map $\mu\colon X\rightarrow \mathfrak{g}^*$.  Then define the function
\begin{align}\label{gdef}
\overline{g}\colon&X\times\mathfrak{g}\rightarrow \mathbb{C}\\
&(x,\zeta)\mapsto \mu(x)(\zeta).\nonumber
\end{align}
This function is $G$-invariant, and so defines a function on the stack-theoretic quotient
\[
g\colon (X\times\mathfrak{g})/G\rightarrow \mathbb{C}.
\]
Via dimensional reduction \cite[Thm.A.1]{Chicago2} there is a natural isomorphism in compactly supported cohomology
\begin{equation}
\label{taste}
\HO_c(\mu^{-1}(0)/G,\mathbb{Q})\otimes\LL^{\dim(\mathfrak{g})}\cong \HO_c((X\times\mathfrak{g})/G,\phi_g\QQ)
\end{equation}
where $\phi_g\QQ$ is the mixed Hodge module complex of vanishing cycles for $g$.  This explains the appearance of vanishing cycles in what follows.

\smallbreak
 
Note that $\phi_g\QQ$ is supported on the critical locus of $g$.  A guiding principle for Donaldson--Thomas theory (e.g. as expressed in \cite{Casson}) is that a given moduli stack $\mathfrak{N}$ of coherent sheaves on a Calabi--Yau 3-fold can be locally expressed as the critical locus of a function $g$ on some smooth ambient stack $\Mst$.  Donaldson--Thomas invariants are then defined by taking invariants, factoring through the Grothendieck group of mixed Hodge structures, of 
\[
\HO_c(\Mst,\phi_g\QQ)=\HO_c(\crit(g),\phi_g\QQ)=\HO_c(\mathfrak{N},\phi_g\QQ).  
\]
The link between Donaldson--Thomas theory and symplectic geometry is completed by the observation of \cite[Sec.4.2]{ginz} (see also \cite{Moz11}) that associated to any quiver $Q$ there is a tripled quiver with potential $(\tilde{Q},\tilde{W})$ such that $(X(\overline{Q})_{\dd}\times\g_{\dd})/\Gl_{\dd}$ is identified with the smooth stack of $\dd$-dimensional representations of $\mathbb{C}\tilde{Q}$, and the critical locus of the function $\mathfrak{Tr}(\tilde{W})$ (which is the function $g$ from (\ref{gdef})) is exactly the substack of representations belonging to the category of representations of the Jacobi algebra\footnote{The definition of $\Jac(\tilde{Q},\tilde{W})$ is recalled in Section \ref{QandP}.} $\Jac(\tilde{Q},\tilde{W})$ associated to the pair $(\tilde{Q},\tilde{W})$.  %This the category of $\Jac(\tilde{Q},\tilde{W})$-modules is noncommutative Donaldson--Thomas theory's analogue of the category of coherent sheaves on a Calabi--Yau 3-fold.  

\smallbreak

Putting all of this together, the cohomological Donaldson--Thomas theory of $\Jac(\tilde{Q},\tilde{W})$ gives us a tool for understanding the compactly supported cohomology of $\Mst(\Pi_Q)$, i.e. there is an isomorphism of cohomologically graded mixed Hodge structures
\begin{equation}
\label{Ataste}
\HO_c\!\left(\Mst(\Pi_Q)_{\dd},\mathbb{Q}\right)\otimes \LL^{\dim(\Gl_{\dd})}\cong\HO_c\!\left(\Mst(\Jac(\tilde{Q},\tilde{W}))_{\dd},\phi_{\mathfrak{Tr}(\tilde{W})}\QQ\right).
\end{equation}
Cohomological DT theory enables us to prove powerful theorems regarding the right hand side of (\ref{Ataste}), which we use to deduce results regarding the left hand side.
\medbreak
\subsection{BPS sheaves and their supports}
We prove Theorem \ref{purityThm} via an analysis of BPS sheaves.  These were introduced in \cite{DaMe15b}, in the course of the proof of the relative cohomological integrality/PBW theorem for the critical cohomological Hall algebras introduced by Konstevich and Soibelman \cite{COHA}.  This theorem states that for a symmetric quiver $Q'$ with potential $W'$, and stability condition $\zeta$ the direct image of the mixed Hodge module of vanishing cycles for the function $\Tr(W')$ along the morphism $\p$ from the moduli stack of $\zeta$-semistable $\mathbb{C}Q'$-modules to the coarse moduli space is obtained by taking the free symmetric algebra generated by an explicitly defined mixed Hodge module $\DTS^{\zeta}_{Q',W'}$, called the \textit{BPS sheaf}, tensored with a half Tate twist of $\HO(\B\mathbb{C}^*,\mathbb{Q})$.  The \textit{BPS cohomology} $\DT^{\zeta}_{Q',W'}$ is defined to be the hypercohomology of this sheaf.
\sbrk
Although the direct image of the mixed Hodge module of vanishing cycles along $\p$ is concentrated in infinitely many cohomological degrees, this BPS sheaf is a genuine mixed Hodge module, i.e. its underlying complex of constructible sheaves is a perverse sheaf.  This is what we mean by ``integrality''.  As a consequence, for any dimension vector $\dd\in\mathbb{N}^{Q_0}$, the associated BPS cohomology $\DT^{\zeta}_{Q',W',\dd}$ lives in bounded degrees.

\subsubsection{} Unless the pair $Q',W'$ is quite special, it is difficult to actually determine $\DTS^{\zeta}_{Q',W'}$.  In particular, the hypercohomology of this sheaf can fail to be pure.  However, in this paper, we show that for the quiver $\tilde{Q}$ with potential $\tilde{W}$ appearing in the previous section, the situation is much more promising.  A key role is played by a support lemma, Lemma \ref{lem1}, which imposes strong restrictions on the support of the BPS sheaf in the case $(Q',W')=(\tilde{Q},\tilde{W})$ for $Q$ any quiver.
\begin{lemma}[Lemma \ref{lem1}]
\label{slemma}
Let $x$ be a point in $\mathcal{M}(\tilde{Q})^{\zeta\sst}_{\dd}$ corresponding to a $\mathbb{C}\tilde{Q}$-module $\rho$, and let $x$ lie in the support of $\DTS^{\zeta}_{\tilde{Q},\tilde{W},\dd}$.  Let $\Lambda$ be the set of generalised eigenvalues of the operators $\rho(\omega_i)$, with $i$ the vertices of $Q$.  Then $\Lambda$ contains only one element.
\end{lemma}
%In particular, if we set $d=\sum_i\dd_i$ and take the direct image of $\DTS^{\zeta}_{\tilde{Q},\tilde{W},\dd}$ along the morphism 
%\[
%\lambda\colon \mathcal{M}(\tilde{Q})^{\zeta\sst}_{\dd}\rightarrow \Sym^{d}(\mathbb{A}^1)
%\]
%recording the eigenvalues of the operators $\rho(\omega_i)$, the complex $\lambda_*\DTS^{\zeta}_{\tilde{Q},\tilde{W},\dd}$ is supported on the small diagonal.
%\sbrk
This is a crucial lemma on the way to proving purity of BPS cohomology for $\Jac(\tilde{Q},\tilde{W})$.  In combination with this purity result, the lemma also enables us to provide some of the first nontrivial calculations of BPS sheaves; see in particular \S \ref{Jordan}, lifting the work of Behrend, Bryan and Szendr\H{o}i on motivic degree zero invariants to the level of BPS sheaves.  This lemma is also one of the crucial ingredients in proving the purity of the BPS sheaves $\DTS^{\zeta}_{\tilde{Q},\tilde{W}}$ themselves, and the definition of the ``less perverse filtration'': see \cite{preproj3,2CYpurity} for developments in this direction.

\subsection{2d BPS sheaves}
Aside from purity of BPS cohomology, one of the main applications of the support lemma is that it enables us to define \textit{2d BPS sheaves}:
\begin{thmx}
\label{2dbpsthm}
Let $m\colon \AA^1\times\Msp(\Pi_Q)_{\dd}^{\zeta\sst}\rightarrow \Msp(\tilde{Q})_{\dd}^{\zeta\sst}$ be the morphism extending a $\Pi_Q$-module to a $\mathbb{C}\tilde{Q}$-module by letting each of the extra loops $\omega_i$ act via scalar multiplication by $z\in\AA^1$.  Then there is a Verdier self-dual mixed Hodge module
\[
\DTS^{\zeta}_{\Pi_Q,\dd}\in\Ob(\MHM(\Msp(\Pi_Q)_{\dd}^{\zeta\sst})),
\]
which we call the \textit{2d BPS sheaf}, such that 
\begin{equation}
\DTS^{\zeta}_{\tilde{Q},\tilde{W},\dd}\cong m_*\!\left(\ICS_{\AA^1}(\QQ)\boxtimes \DTS^{\zeta}_{\Pi_Q,\dd}\right).
\end{equation}
\end{thmx}
The pure intersection complex $\ICS_{\AA^1}(\QQ)$ is defined in \S \ref{iygtyf}.  The 2d BPS sheaves enjoy a number of properties: 
\begin{enumerate}
\item
They categorify the Kac polynomials; we elaborate upon this in \S \ref{AppKac}.  
\item
They are Verdier self-dual (see \S \ref{thbproof}), which we expect to have a role in producing geometric doubles of BPS Lie algebras.  
\item
Their hypercohomology carries a Lie algebra structure, the so-called \textit{BPS Lie algebra} $\mathfrak{g}_{\Pi_Q}$.  \item
They are \textit{pure} as mixed Hodge modules, enabling us to relate generators of $\mathfrak{g}_{\Pi_Q}$ to intersection cohomology.  
\end{enumerate}
These last two properties are explained and explored in the paper \cite{preproj3}, which is devoted to the further study of 2d BPS sheaves.
\subsection{Serre subcategories}
Working with the BPS sheaf $\DTS_{\Pi_Q}$, as opposed to its hypercohomology, enables us to calculate the compactly supported cohomology of substacks of the stack $\Mst(\Pi_Q)$ corresponding to Serre subcategories, even when these stacks are not pure, leading to e.g. applications for genus one character stacks.
\sbrk
A Serre subcategory $\mathcal{S}\subset \mathbb{C}\overline{Q}\lmod$ is a full subcategory such that for every short exact sequence 
\[
\xymatrix{
0\ar[r]&M'\ar[r]&M\ar[r]&M''\ar[r]&0
}
\]
of $\mathbb{C}\overline{Q}$-modules, $M$ is in $\mathcal{S}$ if and only if $M'$ and $M''$ are.  Note that a module $M$ is in $\mathcal{S}$ if and only if all of the subquotients in its Jordan-Holder filtration are in $\mathcal{S}$, or equivalently if its semisimplification is in $\mathcal{S}$.  So restricting attention to $\Mst(\mathbb{C}\overline{Q})^{\mathcal{S}}$, which is defined to be the substack of $\mathbb{C}\overline{Q}$-modules belonging to $\mathcal{S}$, is the same as restricting to the preimage of a particular subspace under the semisimplification map from the stack of $\mathbb{C}\overline{Q}$-modules to the coarse moduli space $\Msp(\overline{Q})$.  
\sbrk
Because many of our results can be stated in the category of mixed Hodge modules\footnote{In cohomological Donaldson--Thomas theory, this is what is meant by the ``relative'' in the relative integrality conjecture.} on $\Msp(\overline{Q})$, we can prove results on the Borel--Moore homology of $\Mst(\Pi_{Q})^{\mathcal{S}}$ via restriction functors and base change.

\smallbreak 

For example consider the quiver $\overline{Q_{\Jor}}$, with one vertex and two loops $X,Y$, and set $\mathcal{S}$ to be the category of representations for which the two loops $X$ and $Y$ are sent to invertible morphisms.  The resulting stack is the character stack for the genus one Riemann surface, and we use the above ideas to calculate its compactly supported cohomology, even though it is \textit{not} pure; see \S \ref{appsss} for details.
\subsection{Structural results}
We prove two general structural results regarding the compactly supported cohomology of stacks $\Msp(\Pi_Q)^{\mathcal{S}}$ for arbitrary finite quiver $Q$ and Serre subcategory $\mathcal{S}$, stated below as Theorems \ref{PoinThm} and \ref{2dint}.  The first is a kind of cohomological wall-crossing isomorphism:
\begin{thmx}
\label{PoinThm}
Let $Q$ be a quiver, let $\mathcal{S}\subset \mathbb{C}\overline{Q}\lmod$ be a Serre subcategory, let $\zeta\in\mathbb{H}_+^{Q_0}$ be a stability condition, and let $\slope $ be the slope function defined with respect to $\zeta$.  Then there is an isomorphism of $\mathbb{N}^{Q_0}$-graded mixed Hodge structures
\begin{align}\label{caf}
&\bigoplus_{\dd\in\mathbb{N}^{Q_0}}\HO_c\!\left(\Mst(\Pi_Q)^{\mathcal{S}}_{\dd},\mathbb{Q}\right)\otimes\LL^{(\dd,\dd)}\cong \\&\bigotimes_{\theta\in({}-\infty,\infty)}\left(\bigoplus_{\dd\in\mathbb{N}^{Q_0}|\substack{\dd=0\textrm{ or}\\\slope(\dd)=\theta}}\HO_c\!\left(\Mst(\Pi_Q)^{\mathcal{S},\zeta\sst}_{\dd},\mathbb{Q}\right)\otimes\LL^{ (\dd,\dd)}\right),\nonumber
\end{align}
where for $\dd',\dd''\in\mathbb{N}^{Q_0}$,
\[
(\dd',\dd''):=\sum_{i\mathrm{ \:a\: vertex\: of\: }Q}\dd'_i\dd''_i-\sum_{a\mathrm{\: an\: arrow\: of\: }Q}\dd'_{\mathrm{source}(a)}\dd''_{\mathrm{target}(a)}
\]
and $\Mst(\Pi_Q)^{\mathcal{S},\zeta\sst}_{\dd}$ is the stack of $\dd$-dimensional $\zeta$-semistable $\Pi_Q$-modules in $\mathcal{S}$.  
\end{thmx}

\subsubsection{}
Taking the Hodge series of both sides of (\ref{caf}), there is an equality of generating series
\begin{align}\label{decaf}
&\sum_{\dd\in\mathbb{N}^{Q_0}}\hp\!\left(\HO_c\!\left(\Mst(\Pi_Q)^{\mathcal{S}}_{\dd},\QQ\right),x,y,z\right)(xyz^2)^{(\dd,\dd)} t^{\dd}\\=&\prod_{\theta\in({}-\infty,\infty)}\!\left(1+\sum_{\slope ({\dd})=\theta}  \hp\!\left(\HO_c\!\left(\Mst(\Pi_Q)^{\mathcal{S},\zeta\sst}_{\dd},\QQ\right),x,y,z\right)(xyz^2)^{(\dd,\dd)}t^{\dd}\right).\nonumber
\end{align}
The compactly supported cohomology of $\Mst(\Pi_Q)_{\dd}^{\mathcal{S},\zeta\sst}$ can fail to be pure, and fail to be of Tate type, but we show that the isomorphism (\ref{caf}) exists nonetheless, and hence, taking the Hodge series of both sides, equation (\ref{decaf}) holds.   We explain how a specialisation of a special case of equation (\ref{decaf}) yields Hausel's formula for the Betti polynomials of Nakajima quiver varieties \cite{Hau10} in Section \ref{AppBC}.  

\subsubsection{PBW/integrality isomorphism}
Our next result is an analogue of the PBW isomorphism from Donaldson--Thomas theory.  We fix a quiver $Q$, a stability condition $\zeta\in \mathbb{H}_+^{Q_0}$, a slope $\theta\in(-\infty,\infty)$, and a Serre subcategory $\Sp$ of the category of $\mathbb{C}\overline{Q}$-modules.  We write
\[
\Coha_{\Pi_Q,\theta}^{\Sp,\zeta}\coloneqq \bigoplus_{\dd\in\mathbb{N}^{Q_0}|\substack{\dd=0\textrm{ or}\\\slope(\dd)=\theta}}\HO^{\BM}(\Mst(\Pi_Q)^{\mathcal{S},\zeta\sst}_{\dd},\mathbb{Q})\otimes\LL^{ -(\dd,\dd)}.
\]
This graded mixed Hodge module carries a Hall algebra structure, see \S \ref{KSCSec} for a generalisation of the construction.
%Theorem \ref{PoinThm} is an analogue of the tensor product decomposition appearing in the statement of the cohomological wall crossing isomorphism of \cite{DaMe15b}.  In fact it is a consequence of it: all the work that goes into proof is demonstrating that we can bridge the symplectic geometry moduli spaces of objects in $\Pi_Q\lmod$ with the categorically 3-dimensional world of $\Jac(\tilde{Q},\tilde{W})\lmod$ and thus import the results of \cite{DaMe15b} into our study of the cohomology of moduli stacks of representations of $\Pi_Q$.  With the bridge built, the same arguments are reemployed to prove our final main result regarding moduli stacks of $\Pi_Q$-representations.
\begin{thmx}
\label{2dint}
Let $\slope$ be the slope function defined with respect to a stability condition $\zeta\in\mathbb{H}_+^{Q_0}$, let $\theta\in(-\infty,\infty)$ be a slope.  
Define the 2d BPS sheaf $\DTS^{\zeta}_{\Pi_Q,\theta}$ as in Theorem \ref{2dbpsthm} and the \textit{BPS cohomology} to be the mixed Hodge structure
\[
\DT_{\Pi_Q,\theta}^{\mathcal{S},\zeta}\coloneqq\bigoplus_{\substack{0\neq \dd\in\mathbb{N}^{Q_0}\\\slope(\dd)=\theta}}\HO_c\!\left(\Msp(\overline{Q})_{\dd}^{\mathcal{S},\zeta\sst},\DTS_{\Pi_Q,\dd}^{\zeta}\right)^{\vee}.
\]
Then there is an isomorphism
\begin{equation}
\label{rel2dI}
\p^{\zeta}_{\theta,!}\left(\bigoplus_{\dd\in\dvst}\QQ_{\Mst(\Pi_Q)_{\dd}}\otimes\LL^{(\dd,\dd)}\right)\cong \Sym_{\boxtimes_{\oplus}}\!\left(\DTS_{\Pi_Q,\theta}^{\zeta}\otimes\HO(\B \mathbb{C}^*,\QQ)^{\vee}\right).
\end{equation}
Moreover, there is a PBW isomorphism
\begin{equation}
\label{bmirror}
\Sym\!\left(\DT_{\Pi_Q,\theta}^{\Sp,\zeta}\otimes\HO(\B\mathbb{C}^*,\mathbb{Q})\right)\xrightarrow{\cong} \Coha_{\Pi_Q,\theta}^{\Sp,\zeta}.
\end{equation}
\end{thmx}
Since $\DTS_{\Pi_Q,\dd}^{\zeta}$ is Verdier self-dual by Theorem \ref{2dbpsthm}, $\DT_{\Pi_Q,\dd}^{\mathcal{S},\zeta}$ is the hypercohomology of the $!$-restriction of the BPS sheaf on the coarse moduli space of $\zeta$-semistable $\dd$-dimensional $\mathbb{C}\tilde{Q}$-modules to the subspace of points representing modules in $\mathcal{S}$.
\subsection{Positivity of restricted Kac polynomials}
\label{KacIntro}
For an arbitrary quiver $Q$, it was proven by Kac in \cite{kac83} that for each dimension vector $\dd\in\mathbb{N}^{Q_0}$ there is a polynomial $\kac_{Q,\dd}(q)\in\mathbb{Z}[q]$ which is equal to the number of absolutely indecomposable $\dd$-dimensional representations of $Q$ over the field of order $q$, whenever $q$ is equal to a prime power.  
\sbrk
In the case of the degenerate stability condition, for which \textit{all} modules are semistable of the same slope, and so the superscript $\zeta$ and the subscript $\theta$ can be dropped, (\ref{bmirror}) gives
\begin{equation}
\label{cmirror}
\bigoplus_{\dd\in\mathbb{N}^{Q_0}}\HO^{\BM}\!\left(\Mst(\Pi_Q)^{\mathcal{S}}_{\dd},\mathbb{Q}\right)\otimes\LL^{ -(\dd,\dd)}\cong\Sym\!\left(\DT_{\Pi_Q}^{\mathcal{S}}\otimes\HO(\B\mathbb{C}^*,\mathbb{Q})\right).
\end{equation}

Taking weight series of both sides of (\ref{cmirror}) yields
\begin{align}
\label{expintro}
\sum_{\dd\in\mathbb{N}^{Q_0}}\wt\!\left(\HO^{\BM}\!\left(\Mst(\Pi_Q)^{\mathcal{S}}_{\dd},\QQ\right),q^{1/2}\right)q^{-(\dd,\dd)} t^{\dd}=&\Exp\left(\sum_{\dd\neq 0}\kac_{Q,\dd}^{\mathcal{S}}(q^{-1/2})(1-q)^{-1}t^{\dd}\right)
\end{align}
where 
\begin{equation}
\label{rKdef}
\kac_{Q,\dd}^{\mathcal{S}}(q^{1/2})=\wt\left(\DT_{\Pi_Q,\dd}^{\mathcal{S}},q^{1/2}\right)
\end{equation}
is the ``$\mathcal{S}$-restricted Kac polynomial''.  The right hand side of (\ref{expintro}) is defined in terms of the plethystic exponential:
\begin{align*}
\Exp\colon &\mathbb{Z}(\!(q^{1/2})\!)[\![t_i\lvert i\in Q_0]\!]_+\rightarrow \mathbb{Z}(\!(q^{1/2})\!)[\![t_i\lvert i\in Q_0]\!]\\
&\sum_{i\in \mathbb{Z},\dd\in\mathbb{N}^{Q_0}}b_{i,\dd}q^{i/2}t^{\dd}\mapsto \prod_{i\in \mathbb{Z},\dd\in\mathbb{N}^{Q_0}}(1-q^{i/2}t^{\dd})^{-b_{i,\dd}}
\end{align*}
where the $+$ subscript means that $b_{i,0}=0$ for all $i\in\mathbb{Z}$.  
\subsubsection{}
Calculating the right hand side of (\ref{rKdef}) looks daunting, but the mere existence of isomorphism (\ref{cmirror}) can tell us something highly nontrivial about $\kac_{Q,\dd}^{\mathcal{S}}(q^{1/2})$ without knowing how to do this calculation (or even knowing what the definition of $\DTS_{\Pi_Q,\dd}$ or $\DTS_{\tilde{Q},\tilde{W},\dd}$ is!).  Namely, if the left hand side of (\ref{cmirror}) is pure, then the BPS cohomology $\DT_{\Pi_Q,\dd}^{\mathcal{S}}$ must also be pure, and so $\kac_{\dd}^{\mathcal{S}}(q^{1/2})$ has positive coefficients (expressed as a polynomial in $-q^{1/2})$.  
\smallbreak
In particular, for the case $\mathcal{S}=\mathbb{C}\overline{Q}\lmod$, the $\mathcal{S}$-restricted Kac polynomial is the same as Kac's original polynomial, and our purity theorem (Theorem \ref{purityThm}) implies Kac's positivity conjecture, originally proved by Hausel, Letellier and Rodriguez-Villegas \cite{HLRV13}.  
\subsubsection{}
A rich source of proper Serre subcategories of $\mathbb{C}\overline{Q}$ is provided by demanding nilpotence of certain paths in $\mathbb{C}\overline{Q}$.  Via this construction Bozec, Schiffmann and Vasserot \cite{BSV17,SV17} define the subcategory of nilpotent, *-semi-nilpotent and *-strongly-semi-nilpotent $\mathbb{C}\overline{Q}$-representations; see \S \ref{PProof} for definitions.  By the above method, in \S \ref{AppKac} we prove positivity of \textit{all} of the resulting polynomials:
\begin{thmx}[Theorem \ref{mpos}, Remark \ref{kacfin}]
\label{rkacp}
Let $Q$ be an arbitrary finite quiver, and let $\dd\in\mathbb{N}^{Q_0}$ be a dimension vector.  Setting $\Sp$ to be any out of the the full subcategory of nilpotent, *-semi-nilpotent, or *-strongly-semi-nilpotent $\mathbb{C}\overline{Q}$-representations, the $\Sp$-restricted Kac polynomial $\kac^{\Sp}_{Q,\dd}(q)$ has positive coefficients.
\end{thmx}
\subsection{Structure of the paper}
In \S\ref{NandC} we give definitions and notation concerning quivers.  In Section \ref{DTtheory} we collect together all required definitions and background theorems from noncommutative/cohomological Donaldson--Thomas theory.  Then in \S\ref{mainProof} we prove Theorem \ref{purityThm}, regarding purity of the compactly supported cohomology of the stack $\Mst(\Pi_Q)$.  Along the way we prove the support lemma (Lemma \ref{slemma} above) and Theorem \ref{2dbpsthm}.
\smallbreak
In \S\ref{Jordan} we focus on the Jordan quiver, corresponding to ``degree zero'' DT theory.  The associated Jacobi algebra $\Jac(\widetilde{Q_{\Jor}},\tilde{W})$ is isomorphic to the \textit{commutative} algebra $\mathbb{C}[x,y,z]$, so that our work makes contact with classical algebraic geometry.  We determine the BPS sheaves for the pair $\widetilde{Q_{\Jor}},\tilde{W}$, and thus determine the compactly supported/vanishing cycle cohomology of various stacks in geometry and nonabelian Hodge theory.  In particular, in this section we calculate the vanishing cycle cohomology of $\Hilb_n(\AA^3)$, via a proof that it is pure, categorifying the results of Behrend, Bryan and Szendr\H{o}i \cite{BBS}.
%For general $Q$, the proof of the support lemma rests heavily on the existence of an element $\sum\omega_i\in\Jac(\tilde{Q},\tilde{W})$ that is central.  Since the (commutative) algebra $\mathbb{C}[x,y,z]$ obviously has a large centre, it turns out we can push the ideas from the proof of Theorem \ref{purityThm} a lot further in the Jordan quiver case.  Put briefly, three applications of the support lemma (one for each of $x,y,z$) imply that the support of the BPS sheaf on the coarse moduli space $X_n(\widetilde{Q_{\Jor}})/\!\!/\Gl_n$ governing the DT theory of $\Jac(\tilde{Q},\tilde{W})$ is the locus of semisimple modules given by 
%\begin{align*}
%\big(\mathbb{C}[x,y,z]/(x-\lambda_x,y-\lambda_y,z-\lambda_z)\big)^{\oplus n}&&\lambda_x,\lambda_y,\lambda_z\in\mathbb{C}.
%\end{align*}
%This observation is enough for us to work out precisely what the BPS sheaf is in this case, combining our purity result with the main result of \cite{BBS}.
\smallbreak
In \S\ref{conseq} we turn back to the geometry of representations of $\Pi_Q$ for a general $Q$.  Thanks to a second support lemma (Lemma \ref{supportProp}), essentially all moduli spaces and stacks of representations of $\Pi_Q$-representations have a (categorically) 3-dimensional analogue, by which we mean that their compactly supported cohomologies fit into isomorphisms like (\ref{Ataste}), and so are recovered from the DT theory of $\Jac(\tilde{Q},\tilde{W})$.  This enables us to prove a generalisation of Theorem \ref{purityThm} incorporating stability conditions.  
\smallbreak
In \S\ref{BC} we use this second support lemma to prove Theorems \ref{PoinThm} and \ref{2dint}.  Then in \S\ref{AppKac} we use the PBW isomorphism in order to relate restricted Kac polynomials to hypercohomology of restrictions of BPS sheaves, and prove Theorem \ref{rkacp}.
\smallbreak
In Sections \S \ref{CoHASec} and \ref{CoHASec2} we explore the implications of purity for the structure of the cohomological Hall algebra $\Coha_{\Pi_Q}$ built out of $\HO_c(\Mst(\Pi_Q),\mathbb{Q})$.  Firstly, in \S \ref{CoHASec} we use equivariant formality (a consequence of purity) to relate this Hall algebra to Hall algebras defined with extra equivariant parameters, introduced by letting a torus $T$ scale the arrows of $\tilde{Q}$.  Then in \S \ref{CoHASec2} we prove (a generalisation of) Conjecture 4.4 of \cite{ScVa12}, stating that the cohomological Hall algebra $\Coha_{\tau,\Pi_Q}$ for an arbitrary preprojective algebra $\Pi_Q$, deformed by incorporating extra equivariant parameters, is naturally a subalgebra of an elementary shuffle algebra.  With these results it becomes possible to perform concrete calculations in $\Coha_{\Pi_Q}$, and we finish with some example applications.  In particular, we prove that the cohomological Hall algebra $\Coha_{\mathfrak{C}oh(\mathbb{A}^2)}$ of finite length coherent sheaves on $\mathbb{A}^2$, without any extra equivariant parameters, is \textit{not} commutative, and that the cohomological Hall algebra $\Coha_{\tau,\Pi_Q}$ can fail to be torsion free if $T$ is chosen to be one of certain bad 1-dimensional tori.
\medbreak
\subsection{Conventions}
\begin{itemize}
\item
For $G$ a complex algebraic group, we set $\HO_G\coloneqq \HO(\B G,\mathbb{Q})$.
\item
All functors are assumed to be derived unless explicitly stated otherwise.
\item
All quivers are assumed to be finite.
\item
For $X$ a complex variety, or global quotient stack, we continue to denote
\[
\HO^{\BM}(X,\mathbb{Q})\coloneqq \HO_c(X,\mathbb{Q})^{\vee}.
\]
\item
We continue to write $\mathbb{N}=\mathbb{Z}_{\geq 0}$.
\item
Wherever an object appears with a subscript that is a bold Roman letter, that letter refers to a dimension vector, and $\bullet_{\dd}$ is the subobject corresponding to that dimension vector.  If any such object appears with a Greek letter such as $\theta$ as a subscript, then $\theta$ will refer to a slope, and $\bullet_{\theta}$ will refer to the subobject corresponding to dimension vectors in $\Lambda_{\theta}^{\zeta}$.  Finally, if an expected subscript is missing altogether, then the entire object is intended.
\item
We generally use capital Roman letters to refer to spaces of representations before taking any kind of quotient, calligraphic letters to refer to GIT moduli spaces, and fraktur letters to refer to moduli stacks.
\item
For $\mathscr{D}$ a triangulated category equipped with a t structure, we define the total cohomology functor $\Ho(\bullet)\coloneqq \bigoplus_{i\in\mathbb{Z}}\Ho^i(\bullet)[-i]$.
\end{itemize}
\begin{convention}\label{stabConv}
Where a space or object is defined with respect to a stability condition $\zeta$, that stability condition will appear as a superscript.  In the event that the superscript is missing, we assume that $\zeta$ is the degenerate King stability condition $(i,\ldots,i)\in \mathbb{H}_+^{Q_0}$.  With respect to this stability condition all representations have the same slope and are semistable, semisimple representations are the polystable representations, and the stable representations are exactly the simple ones.
\end{convention}

\subsection{Acknowledgements}
During the writing of this paper, I was a postdoctoral researcher at EPFL, supported by the Advanced Grant ``Arithmetic and physics of Higgs moduli spaces'' No. 320593 of the European Research Council.  During redrafting of the paper, I was supported by the starter grant ``Categorified Donaldson--Thomas theory'' No. 759967 of the European Research Council.  I was also supported by a Royal Society university research fellowship.
\sbrk
I would like to thank Sasha Minets, Tristan Bozec, Olivier Schiffmann, Eric Vasserot, Davesh Maulik and Victor Ginzburg for illuminating conversations that contributed greatly to the paper.  In particular, the idea for the proof of the crucial ``support lemma'' (Lemma \ref{lem1}) came from seeing Victor Ginzburg talk about the results of \cite{DGT16} at the Warwick EPSRC symposium ``Derived Algebraic Geometry, with a focus on derived symplectic techniques'', and the final section of the paper benefitted greatly from Sasha Minets' careful reading of an earlier draft.
\section{Quiver representations}
\label{NandC}
\subsection{Quivers and potentials}
\label{QandP}
Throughout the paper, $Q$ will be used to denote a finite quiver, i.e. a pair of finite sets $Q_0$ and $Q_1$ (the vertices and arrows, respectively), and a pair of maps $s,t\colon Q_1\rightarrow Q_0$ (taking an arrow to its source and target, respectively).  We denote by $\mathbb{C}Q$ the path algebra of $Q$, i.e. the algebra over $\mathbb{C}$ having as a basis the paths in $Q$, with structure constants for the multiplication given by concatenation of paths.  For each vertex $i\in Q_0$, there is a ``lazy'' path of length $0$ starting and ending at $i$, and we denote by $e_i$ the resulting element of $\mathbb{C}Q$.
\sbrk
A \textit{potential} on a quiver $Q$ is an element $W\in \mathbb{C} Q/[\mathbb{C} Q,\mathbb{C} Q]_{\textrm{vect}}$.  A potential is given by a linear combination of cyclic words in $Q$, where two cyclic words are considered to be the same if one can be cyclically permuted to the other.  If $W$ is a single cyclic word, and $a\in Q_1$, we define
\[
\partial W/\partial a=\sum_{\substack{W=cac'\\ c\textrm{ and }c'\textrm{ paths in }Q}}c'c
\]
and we extend this definition linearly to general $W$.  We define
\[
\Jac(Q,W):=\mathbb{C} Q/\langle \partial W/\partial a|a\in Q_1\rangle,
\]
the \textit{Jacobi algebra} associated to the quiver with potential $(Q,W)$.  We will often abbreviate ``quiver with potential'' to just ``QP''.
\sbrk
Given a quiver $Q$, we denote by $\overline{Q}$ the quiver obtained by doubling $Q$.  This is defined by setting $\overline{Q}_0:=Q_0$ and $\overline{Q}_1=\{a,a^*|a\in Q_1\}$, and extending $s$ and $t$ to maps $\overline{Q}_1\rightarrow \overline{Q}_0$ by setting
\begin{align*}
s(a^*)=&t(a)\\
t(a^*)=&s(a).
\end{align*}
We denote by $\Pi_Q$ the preprojectve algebra of $Q$, defined by
\[
\Pi_Q:=\mathbb{C}\overline{Q}/\left\langle \sum_{a\in Q_1}[a,a^*]\right\rangle.
\]
We denote by $\tilde{Q}$ the quiver obtained from $Q$ by setting 
\begin{align*}
\tilde{Q}_0:=&Q_0\\
\tilde{Q}_1:=&\overline{Q}_1\coprod \{\omega_i|i\in Q_0\}, 
\end{align*}
where each $\omega_i$ is an arrow satisfying $s(\omega_i)=t(\omega_i)=i$.  If a quiver $Q$ is fixed, we define the potential $\tilde{W}$ as in \cite[Sec.4.2]{ginz} and \cite{Moz11} by setting
\[
\tilde{W}=\sum_{a\in Q_1}[a,a^*]\sum_{i\in Q_0}\omega_i.
\]
\smallbreak 
If $A$ is an algebra, we denote by $A\lmod$ the category of finite-dimensional $A$-modules.
\begin{proposition}
\label{basicEq}
Define $\mathcal{C}_{\Pi_Q}$ to be the category whose objects are pairs $(M,f)$, where $M$ is a finite-dimensional $\Pi_Q$-module, and $f\in\End_{\Pi_Q\lmod}(M)$, and define $\Hom_{\mathcal{C}_{\Pi_Q}}\big((M,f),(M',f')\big)$ to be the subspace of morphisms $g\in\Hom_{\Pi_Q\lmod}(M,M')$ such that the diagram
\[
\xymatrix{
M\ar[d]^f\ar[r]^g&M'\ar[d]^{f'}\\
M\ar[r]^g&M'
}
\]
commutes.  Then there is an isomorphism of categories
\[
\mathcal{C}_{\Pi_Q}\cong \Jac(\tilde{Q},\tilde{W})\lmod.
\]
\end{proposition}
\begin{proof}
From the relations $\partial \tilde{W}/\partial \omega_i$, for $i\in Q_0$, we deduce that the natural inclusion $\mathbb{C}\overline{Q}\subset \mathbb{C}\tilde{Q}$ induces an inclusion $\Pi_Q\subset \Jac(\tilde{Q},\tilde{W})$.  So a $\Jac(\tilde{Q},\tilde{W})$-module is given by a $\Pi_Q$-module $M$, along with linear maps $M(\omega_i)\in\End_{\mathbb{C}}(e_i\cdott M)$ satisfying
\begin{align*}
M\big(\partial \tilde{W}/\partial a\big)=&M\big(a^*\big)M\big(\omega_{s(a^*)}\big)-M\big(\omega_{t(a^*)}\big)M\big(a^*\big)=0\\
M\big(\partial \tilde{W}/\partial a^*\big)=&M\big(\omega_{t(a)}\big)M\big(a\big)-M\big(a\big)M\big(\omega_{s(a)}\big)=0.
\end{align*}
These are precisely the conditions for the elements $\{M(\omega_i)\}_{i\in Q_0}$ to define an endomorphism of $M$, considered as a $\Pi_Q$-module.
\end{proof}

\subsection{Moduli spaces}
\label{modSpaces}
Given an algebra $A$, presented as a quotient 
\[
A=\mathbb{C}Q/I
\]
of a free path algebra by a two-sided ideal $I\subset \mathbb{C}Q_{\geq 1}$ generated by paths of length at least one, and a dimension vector $\dd\in \mathbb{N}^{Q_0}$, we denote by $\Mst(A)_{\dd}$ the stack of $\dd$-dimensional complex representations of $A$.  This is a finite type Artin stack.  In the case $A=\mathbb{C}Q$ we abbreviate $\Mst(\mathbb{C}Q)_{\dd}$ to $\Mst(Q)_{\dd}$.  This stack is naturally isomorphic to the quotient stack
\[
X(Q)_{\dd}/\Gl_{\dd},
\]
where
\[
X(Q)_{\dd}:=\prod_{a\in Q_1}\Hom\left(\mathbb{C}^{\dd_{s(a)}},\mathbb{C}^{\dd_{t(a)}}\right)
\]
and
\[
\Gl_{\dd}:=\prod_{i\in Q_0} \Gl_{\dd_i}(\mathbb{C}).
\]
We define $\g_{\dd}=\prod_{i\in Q_0}\mathfrak{gl}_{\dd_i}(\mathbb{C})$, and define
\begin{align*}
\mu_{Q,\dd}\colon&X(\overline{Q})_{\dd}\rightarrow \g_{\dd}\\
&\rho\mapsto \sum_{a\in Q_1}[\rho(a),\rho(a^*)].
\end{align*}
Then as substacks of $\Mst(\overline{Q})_{\dd}$, there is an equality $\mu_{Q,\dd}^{-1}(0)/\G_{\dd}=\Mst(\Pi_Q)_{\dd}$.  
\smallbreak
As in the introduction, we define the function
\begin{align*}
\Tr (\tilde{W})_{\dd}\colon&X(\tilde{Q})_{\dd}\rightarrow \mathbb{C}\\
&\rho\mapsto \Tr\left(\sum_{a\in Q_1}[\rho(a),\rho(a^*)]\sum_{i\in Q_0}\rho(\omega_i)\right),
\end{align*}
and denote by $\mathfrak{Tr}(\tilde{W})_{\dd}\colon\Mst(\tilde{Q})_{\dd}\rightarrow \mathbb{C}$ the induced function.  As substacks of $\Mst(\tilde{Q})_{\dd}$, there are equalities
\begin{equation}
\label{jacCrit}
\crit\left(\Tr(\tilde{W})_{\dd}\right)/\G_{\dd}=\Mst(\Jac(\tilde{Q},\tilde{W}))_{\dd}=\crit(\mathfrak{Tr}(\tilde{W})_{\dd}).
\end{equation}
We define $\Mst(\tilde{Q})^{\omega\nnilp}_{\dd}\subset\Mst(\tilde{Q})_{\dd}$ to be the reduced stack defined by the vanishing of the functions 
\[
\{\mathfrak{Tr}(\rho(\omega_i)^m)|i\in Q_0\hbox{ }1\leq m\leq {\dd}_{i}\}.  
\]
The geometric points of $\Mst(\tilde{Q})^{\omega\nnilp}_{\dd}$ over a field extension $K\supset \mathbb{C}$ correspond to $\dd$-dimensional $K\tilde{Q}$ representations $\rho$ such that for each $i\in Q_0$, the endomorphism $\rho(\omega_i)$ is a nilpotent $K$-linear endomorphism.

A \textit{stability condition} for $Q$ is defined to be an element of $\mathbb{H}_+^{Q_0}$, where 
\[
\mathbb{H}_+:=\{r\exp(i\pi\phi)\in \mathbb{C}\mid r>0, 0<\phi\le 1\}.
\]
\begin{definition}
For a fixed stability condition $\zeta\in \mathbb{H}_+^{Q_0}$, we define the \textit{central charge} 
\begin{align*}
Z\colon&\mathbb{N}^{Q_0}\setminus\{0\}\rightarrow \mathbb{H}_+\\
&\dd\mapsto \dd\cdot \zeta.
\end{align*}
\label{slopeDef}
We define the slope of a dimension vector $\dd\in\mathbb{N}^{Q_0}\setminus\{0\}$ by setting
\begin{equation}
\label{sl_def}
\slope(\dd):=\begin{cases} - \Re e (Z(\dd))/ \Im m (Z(\dd))&\textrm{if }\Im m(Z(\dd))\neq 0\\\infty&\textrm{otherwise.}\end{cases}
\end{equation}
\end{definition}
If $\rho$ is a representation of $Q$, we define $\slope(\rho):=\slope(\dim(\rho))$.  A representation $\rho$ is called $\zeta$\textit{-semistable} if for all proper subrepresentations $\rho'\subset \rho$ we have $\slope(\rho')\leq \slope(\rho)$, and $\rho$ is called $\zeta$-\textit{stable} if the inequality is strict.
\sbrk
We will always assume that our stability conditions are \textit{King stability conditions}, meaning that for each $1_i\in\mathbb{N}^{Q_0}$ in the natural generating set, $\Im m (Z(1_i))=1$ and $\Re e (Z(1_i))\in\mathbb{Q}$.
\sbrk
If $\zeta$ is a King stability condition, then for each $\dd\in\mathbb{N}^{Q_0}$ there is a geometric invariant theory (GIT) coarse moduli space of $\zeta$-semistable $Q$-representations of dimension $\dd$, constructed in \cite{King94}, which we denote $\Msp(Q)^{\zeta\sst}_{\dd}:=X(Q)^{\zeta\sst}_{\dd}/\!\!/_{\chi(\zeta)} \Gl_{\dd}$.  Here $X(Q)^{\zeta\sst}_{\dd}\subset X(Q)_{\dd}$ is the open subscheme whose geometric points correspond to $\zeta$-semistable $Q$-representations.
\sbrk
We denote by 
\begin{equation}
\label{pdef}
\p_{Q,\dd}^{\zeta}\colon\Mst(Q)^{\zeta\sst}_{\dd}\rightarrow \Msp(Q)^{\zeta\sst}_{\dd}
\end{equation}
the morphism from the stack to the coarse moduli space.  At the level of points, this map takes a semistable representation $\rho$ to the direct sum of the subquotients appearing in the Jordan--H\"older filtration of $\rho$, considered as an object in the category of $\zeta$-semistable representations of slope $\slope(\dd)$.  If there is no ambiguity, we omit the subscript $Q$ from the definition of $\p$.
\smallbreak
We denote by 
\[
q^{\zeta}_{Q,\dd}\colon \Msp(Q)^{\zeta\sst}_{\dd}\rightarrow \Msp(Q)_{\dd}
\]
the morphism from the GIT quotient to the affinization.  This morphism is proper, as can be seen from the construction of the domain via GIT.  At the level of points, $q^{\zeta}_{Q,\dd}$ takes a $\zeta$-semistable module to its semisimplification.  
\smallbreak
We define two pairings on $\mathbb{Z}^{Q_0}$
\begin{align*}
(\dd,\ee)_Q:=&\sum_{i\in Q_0} \dd_i \ee_i-\sum_{a\in Q_1}\dd_{s(a)}\ee_{t(a)}\\
\langle \dd,\ee\rangle_Q:=&(\dd,\ee)_Q-(\ee,\dd)_Q.
\end{align*}
Again, we will drop the subscript $Q$ when the choice of quiver is obvious from the context.  For $\theta\in (-\infty,\infty)$ a slope, we denote by 
\begin{equation}
\label{LambdaDef}
\Lambda_{\theta}^{\zeta}\subset\mathbb{N}^{Q_0}
\end{equation}
the submonoid of dimension vectors $\dd$ such that $\dd=0$ or $\slope(\dd)=\theta$.
\begin{definition}
A stability condition $\zeta\in\mathbb{H}_+^{Q_0}$ is $\theta$-\textit{generic} if for all $\dd,\ee\in\Lambda_{\theta}^{\zeta}$, $\langle \dd,\ee\rangle=0$, and we say that $\zeta$ is \textit{generic} if it is $\theta$-generic for all $\theta$.  
\end{definition}
\begin{definition}
A quiver $Q$ is a \textit{symmetric} if for any two vertices $i,j\in Q_0$ the number of arrows $a$ with $s(a)=i$ and $t(a)=j$ is equal to the number of arrows with $s(a)=j$ and $t(a)=i$.
\end{definition}
\begin{definition}\label{degstabcond}
For $Q$ a quiver, we define the \textit{degenerate} stability condition
\[
\zeta=(i,\ldots,i)\in\mathbb{H}_+^{Q_0}.
\]
\end{definition}
If $Q$ is symmetric, then all stability conditions $\zeta\in\mathbb{H}_+^{Q_0}$ are generic.  The degenerate stability condition is generic if and only if $Q$ is symmetric.  In particular, for all quivers $Q$, the degenerate stability condition is generic for $\overline{Q}$ and $\tilde{Q}$.

\begin{definition}
\label{dimDef}
We define by $\dim^{\zeta}\colon\Msp(Q)^{\zeta\sst}\rightarrow \mathbb{N}^{Q_0}$ the map taking a polystable quiver representation to its dimension vector, and define 
\[
\Dim^{\zeta}:=\dim^{\zeta}\circ \p_Q^{\zeta}.
\]
where $\p_Q^{\zeta}$ is as in (\ref{pdef}).
\end{definition}
If $\Sp$ is a Serre subcategory of the category of $\mathbb{C}Q\lmod$, we denote by $\iota'\colon \Msp(Q)_{\dd}^{\Sp,\zeta\sst}\hookrightarrow \Msp(Q)_{\dd}^{\Sp,\zeta\sst}$ the inclusion of the polystable $\mathbb{C}Q$ modules that are objects of $\Sp$.  We only consider choices of $\Sp$ for which this is a morphism of varieties.  We set
\[
\Mst(Q)_{\dd}^{\Sp,\zeta\sst}=\Msp(Q)_{\dd}^{\Sp,\zeta\sst}\times_{\Msp(Q)_{\dd}^{\zeta\sst}}\Mst(Q)_{\dd}^{\zeta\sst}
\]
and denote by
\[
\iota\colon \Mst(Q)_{\dd}^{\Sp,\zeta\sst}\hookrightarrow \Mst(Q)_{\dd}^{\zeta\sst}
\]
the inclusion.
\section{Cohomological Donaldson--Thomas theory}
\label{DTtheory}
\subsection{Vanishing cycles and mixed Hodge modules}
\label{iygtyf}
Let $X$ be a smooth complex variety, and let $f$ be a regular function on it.  Set 
\begin{align*}
X_0=f^{-1}(0)&&
X_{<0}=f^{-1}(\mathbb{R}_{<0}).
\end{align*}
We define the nearby cycle functor as the following composition of (derived) functors
\[
\psi_f:=(X_0\rightarrow X)_*(X_0\rightarrow X)^* (X_{<0}\rightarrow X)_*(X_{<0}\rightarrow X)^*,
\]
and we define the functor $\phip{f}=\cone\left((X_0\rightarrow X)_*(X_0\rightarrow X)^*\rightarrow\psi_f\right)[-1]$.  Alternatively, define $X_{\leq 0}=f^{-1}(\mathbb{R}_{\leq 0})$, and define the (underived) functor $\Gamma_{X_{\leq 0}}$ by setting 
\[
\Gamma_{X_{\leq 0}}\mathcal{F}(U)=\ker (\mathcal{F}(U)\rightarrow \mathcal{F}(U\setminus X_{\leq 0})).  
\]
Then we can define $\phip{f}\mathcal{F}=(R\Gamma_{X_{\leq 0}}\mathcal{F})_{X_0}$.  We define $\psip{f}\coloneqq\psi_f[-1]$.   

\smallbreak
If $X$ is a quasiprojective complex variety, and so there is a closed embedding $X\subset Y$ inside a smooth complex variety, and $f$ extends to a function $\overline{f}$ on $Y$, we define $\phip{f}=i^*\phip{\overline{f}}i_*$, where $i\colon X\rightarrow Y$ is the embedding.

For a complex variety $X$ we define as in \cite{Saito89,Saito90} the category $\MHM(X)$ of mixed Hodge modules on $X$.  See \cite{Saito1} for an overview of the theory.  There is an exact functor 
\[
\rat\colon\Dub(\MHM(X))\rightarrow\Dub(\perv(X))
\]
which takes a complex of mixed Hodge modules $\mathcal{F}$ to its underlying complex of perverse sheaves, and commutes with $f_*,f_!,f^*,f^!,\DD_X$ and tensor product.  In addition, the functors $\phip{f}$ and $\psip{f}$ lift to exact functors for the category of mixed Hodge modules.  We denote by $\phin{f}$ the lift of $\phip{f}$.
\begin{remark}
\label{critRem}
If  $f$ is a regular function on the smooth variety $X$, then $\supp(\phip{f}\mathbb{Q}_X)=\supp(\phin{f}\mathbb{Q}_X)=\crit(f)$.
\end{remark}
\subsubsection{}
In the context of Donaldson--Thomas theory of \textit{general} Jacobi algebras it is necessary to work in a larger category than $\MHM(X)$, called the category of monodromic mixed Hodge modules on $X$, denoted $\MMHM(X)$.  This category is equivalent to the full subcategory of mixed Hodge modules on $X\times\mathbb{C}^*$ such that along each fibre $\{x\}\times\mathbb{C}^*$ the total cohomology of the restriction is an admissible variation of mixed Hodge structure.  See \cite[Sec.7]{COHA} or \cite[Sec.2]{DaMe15b} for an introduction to this category, along with its slightly subtle monoidal product.  Shifted pullback along the inclusion $X\times\{1\}\rightarrow X\times\mathbb{C}^*$ gives a faithful functor $\MMHM(X)\rightarrow\MHM(X)$ --- one should think of this functor as ``forgetting monodromy.''  There is an embedding of monoidal categories $\tau\colon\MHM(X)\rightarrow\MMHM(X)$ defined by 
\begin{equation}
\label{taudef}
\tau=\bullet\boxtimes \QQ_{\mathbb{C}^*}[1].
\end{equation}
One should think of this functor as turning a mixed Hodge module into a monodromic mixed Hodge module by stipulating that the monodromy is trivial, and of the essential image of this functor as being ``monodromy-free'' monodromic mixed Hodge modules.  This is a symmetric monoidal functor.
\smallbreak
The functor $\phi_f\colon \MHM(X)\rightarrow \MHM(X)$ lifts to a functor 
\begin{align*}
\phim{f}&\colon \MHM(X)\rightarrow \MMHM(X)\\
\phim{f}&\coloneqq\phi_{f/u}(X\times\mathbb{C}^*\xrightarrow{\pi_X} X)^*[1]
\end{align*}
as in \cite[Def.27]{COHA}.  Here $u$ is the coordinate on $\mathbb{C}^*$.  
\begin{remark}
\label{canlift}
For $g$ a regular function on a complex variety $Y$, set 
\[
\nu_g\colon \phi_{g}\rightarrow (g^{-1}(0)\rightarrow Y)_*(g^{-1}(0)\rightarrow Y)^*
\]
to be the canonical natural transformation of functors of mixed Hodge modules.  Let $f$ be a regular function on $X$.  Then the natural transformation 
\[
\nu_{f/u}(X\times\mathbb{C}^*\xrightarrow{\pi_X} X)^*[1]
\]
provides a natural transformation $\nu^{\mon}_f\colon\phim{f}\rightarrow \tau \circ(\mathcal{F}\mapsto \mathcal{F}|_{X_0})$, where $\tau$
is as in (\ref{taudef}).  The natural transformation $\nu^{\mon}_f$ is a lift of the natural transformation $\nu_f$ to the category $\MMHM(X)$.  
\end{remark}
The reason for introducing monodromic mixed Hodge modules is that for a general pair $(Q,W)$, if one restates the cohomological integrality theorem (Theorem \ref{IntThm}) purely in terms of the ordinary tensor category of mixed Hodge modules, with $\phi$ instead of $\phim{}$, it is not true -- the subtlety here regards tensor products of monodromic mixed Hodge modules.  For our purposes though, this headache will not occur --- see Remark \ref{getOut}.

\begin{definition}
\label{DerDef}
We define $\Db(\MMHM(X))$ to be the bounded derived category of monodromic mixed Hodge modules on $X$.  If $X$ is connected, we define $\Du(\MMHM(X))$ to be the inverse limit of the diagram of categories
\[
\ldots\rightarrow\Db(\MMHM(X))\xrightarrow{\tau^{\leq n}}\Db(\MMHM(X))\xrightarrow{\tau^{\leq n-1}}\Db(\MMHM(X))\rightarrow\ldots 
\]
Explicitly, an object of $\Du(\MMHM(X))$ is given by a $\mathbb{Z}$-tuple of objects $\mathcal{F}_n$ in $\Du(\MMHM(X))$, along with isomorphisms $\tau^{\leq n-1}\mathcal{F}_n\cong \mathcal{F}_{n-1}$.  For $\mathcal{F}$ an object in $\Du(\MMHM(X))$  we write $\tau^{\leq n}\mathcal{F}=\mathcal{F}_n$ and $\Ho^n(\mathcal{F})=\Ho^n(\mathcal{F}_n)$.  For an object $\mathcal{F}$ of $\Du(\MMHM(X))$, the cohomological amplitude of the objects $\mathcal{F}_n$ are universally bounded below.  
\smallbreak
Similarly, we define $\Dl(\MMHM(X))$ to be the inverse limit of the diagram
\[
\ldots\rightarrow\Db(\MMHM(X))\xrightarrow{\tau^{\geq n}}\Db(\MMHM(X))\xrightarrow{\tau^{\geq n+1}}\Db(\MMHM(X))\rightarrow\ldots 
\]
For general $X$, we define
\[
\Du(\MMHM(X))\coloneqq\prod_{X'\in\pi_0(X)}\Du(\MMHM(X'))
\]
and $\Du(\MMHM(X))$ similarly.
\end{definition}

%\begin{remark}
%\label{noMonRem}
%For the reader that draws the line at learning what $\phi_f$ is, as opposed to what $\phim{f}$ is, this paper can be read without loss, since the monodromic mixed Hodge modules that we will be concerned with have trivial monodromy, i.e. they lie in the essential image of the fully faithful embedding $\tau$ defined in (\ref{taudef}) --- see Remark \ref{getOut}.  As a result, in the cases that will concern us the cohomological integrality theorem holds, even if stated in terms of ordinary mixed Hodge modules.
%\end{remark}
A monodromic mixed Hodge module $\mathcal{F}$ comes with a filtration 
\[
\ldots\subset W_{i}\mathcal{F}\subset W_{i+1}\mathcal{F}\subset\ldots, 
\]
the weight filtration, which is equal to the usual weight filtration if $\mathcal{F}$ is a genuine mixed Hodge module.  
\begin{definition}
We say that $\mathcal{F}\in\MMHM(X)$ is pure of weight $n$ if $W_{n-1}\mathcal{F}=0$ and $W_{n}\mathcal{F}=\mathcal{F}$.  Given $\mathcal{F}$ an object of $\Dub^{\heartsuit}(\MMHM(X))$, with $\heartsuit=b,\leq,\geq$, we say that $\mathcal{F}$ is pure of weight $n$ if $\Ho^i(\mathcal{F})$ is pure of weight $i+n$ for all $i$, or we just call $\mathcal{F}$ ``\textit{pure}'' if each $\Ho^i(\mathcal{F})$ is pure of weight $i$.
\end{definition}

We define $\mathbb{L}:=\HO_c(\mathbb{A}^1,\mathbb{Q})$, considered as a cohomologically graded mixed Hodge structure, i.e. as a pure cohomologically graded mixed Hodge structure concentrated in cohomological degree two.  Via the embedding (\ref{taudef}) we may consider this object alternatively as a cohomologically graded monodromic mixed Hodge structure, or a cohomologically graded monodromic mixed Hodge module on a point.  Working in the category $\MMHM(\pt)$, we define $\LL^{ 1/2}:=\HO_c(\mathbb{A}^1,\phim{x^2}\mathbb{Q}_{\mathbb{A}^1})$, to obtain a tensor square root of $\mathbb{L}$.  In other words we have $\mathbb{L}^{1/2}\otimes \mathbb{L}^{ 1/2}\cong\mathbb{L}$.  
\begin{warning}
Using the Thom--Sebastiani isomorphism and Theorem \ref{dimRedThm} below, one can show that there are \textit{two} equally natural choices for this isomorphism, depending on which ``dimensional reduction'' of $x_1^2+x_2^2=(x_1+ix_2)(x_1-ix_2)$ we consider.  These isomorphisms are not the same!  This issue can be largely ignored in this paper, but is the reason for the signs appearing in \eqref{signedIso}.
\end{warning}
\begin{convention}
Let $X$ be a complex variety, such that each connected component contains a connected dense smooth locus.  In this paper we will shift the definition of the intersection complex mixed Hodge module for $X$ so that it is pure of weight zero, while its underlying element in $\Db(\perv(X))$ is a perverse sheaf.  This we achieve by setting
\[
\ICS_{X}(\mathbb{Q}):=\sum_{Z\in \pi_0(X)}\mathrm{IC}_{Z_{\reg}}(\mathbb{Q})\otimes\mathbb{L}^{ -\dim(Z)/2}.
\]
If $X$ is a smooth connected variety, we set 
\[
\HO(X,\mathbb{Q})_{\vir}:=\HO(X,\ICS_{X}(\mathbb{Q}))\cong\HO(X,\mathbb{Q})\otimes\mathbb{L}^{ -\dim(X)/2}
\]
and
\[
\HO_c(X,\mathbb{Q})_{\vir}:=\HO_c(X,\ICS_{X}(\mathbb{Q}))\cong\HO_c(X,\mathbb{Q})\otimes\mathbb{L}^{-\dim(X)/2}.
\]
Since the smooth stack $\B\mathbb{C}^*$ has complex dimension -1, we extend this notation in the natural way by setting
\begin{align*}
\HO(\B\mathbb{C}^*,\mathbb{Q})_{\vir}:=&\HO(\B\mathbb{C}^*,\mathbb{Q})\otimes\mathbb{L}^{1/2}
\\
\HO_c(\B\mathbb{C}^*,\mathbb{Q})_{\vir}:=&\HO(\B\mathbb{C}^*,\mathbb{Q})^{\vee}\otimes\mathbb{L}^{-1/2}
\\\HO_c(\B\mathbb{C}^*,\mathbb{Q}):=&\HO(\B\mathbb{C}^*,\mathbb{Q})^{\vee}\otimes \mathbb{L}^{-1}.
\end{align*}
\end{convention}

\subsection{Pushforwards from stacks}
\label{pfs}
Assume that $X$ is a smooth complex variety, carrying the action of the algebraic group $G$, and let $f$ be a $G$-invariant regular function on $X$, and let $p\colon X/G\rightarrow Y$ be a morphism from the global quotient stack to a scheme $Y$.  Then we recall (following \cite[Sec.2]{DaMe15b}) how to define $p_*\phim{f}\mathbb{Q}_{X/G}\in\Ob(\Du(\MMHM(Y)))$.  We recall the definition for the case in which $X$ is connected --- the general definition is obtained by taking the direct sum over connected components.  The definition is a minor modification of Totaro's well-known construction \cite{To99}.  
\smallbreak
Firstly, let 
\[
V_0\subset V_1\subset\ldots
\]
be an ascending chain of $G$-representations, and let 
\[
U_0\subset U_1\subset\ldots
\]
be an ascending sequence of closed inclusions of $G$-equivariant varieties, with each $U_i\subset X\times V_i$ an open dense subvariety.  We assume furthermore that 
\[
\lim_{i\mapsto\infty}\left(\codim_{X\times V_i}\big( (X\times V_i)\setminus U_i\big)\right)=\infty,
\]
that $G$ acts freely on $U_i$ for all $i$, and that the principal bundle $U_i\rightarrow U_i/G$ exists in the category of complex varieties.  Then we define
\[
X_i:=U_i/G
\]
and denote by $p_i\colon X_i\rightarrow Y$ and $f_i\colon X_i\rightarrow\mathbb{C}$ the induced maps.  The closed embedding $\iota_{i,i+1}\colon X_i\rightarrow X_{i+1}$ induces maps 
\[
p_{i+1,*}\phim{f_{i+1}}\mathbb{Q}_{X_{i+1}}\rightarrow p_{i,*}\phim{f_i}\mathbb{Q}_{X_i}
\]
and 
\[
p_{i,!}\phim{f_i}\mathbb{Q}_{X_i}\rightarrow p_{i+1,!}\phim{f_{i+1}}\mathbb{Q}_{X_{i+1}}\otimes\mathbb{L}^{\dim(X_i)-\dim(X_{i+1})}.
\]
For fixed $n$ and sufficiently large $i$, the maps
\[
\Ho^n(p_{i+1,*}\phim{f_{i+1}}\mathbb{Q}_{X_{i+1}})\rightarrow \Ho^n(p_{i,*}\phim{f_i}\mathbb{Q}_{X_i})
\]
and
\[
\Ho^n(p_{i,!}\phim{f_i}\mathbb{Q}_{X_i}\otimes\mathbb{L}^{-\dim(U_i)})\rightarrow \Ho^n(p_{i+1,!}\phim{f_{i+1}}\mathbb{Q}_{X_{i+1}}\otimes\mathbb{L}^{-\dim(U_{i+1})})
\]
are isomorphisms (see e.g. \cite[Sec.3.4]{QCP}), stabilising to a monodromic mixed Hodge module that is independent of our choice of $\ldots\subset U_i\subset U_{i+1}\subset\ldots$.  Note that 
\[
\Ho^m(p_{i,*}\phim{f_{i}}\mathbb{Q}_{X_{i+1}})=0\quad\quad \Ho^n(p_{i,!}\phim{f_i}\mathbb{Q}_{X_i}\otimes\mathbb{L}^{-\dim(U_i)})=0
\]
for $m<\dim(X)$ and $n>\dim(X)$.  So for fixed $n$ and sufficiently large $i$ the morphisms
\[
\tau^{\leq n}(p_{i+1,*}\phim{f_{i+1}}\mathbb{Q}_{X_{i+1}})\rightarrow \tau^{\leq n}(p_{i,*}\phim{f_i}\mathbb{Q}_{X_i})
\]
and
\[
\tau^{\geq n}(p_{i,!}\phim{f_i}\mathbb{Q}_{X_i}\otimes\mathbb{L}^{-\dim(U_i)})\rightarrow \tau^{\geq n}(p_{i+1,!}\phim{f_{i+1}}\mathbb{Q}_{X_{i+1}}\otimes\mathbb{L}^{-\dim(U_{i+1})})
\]
are isomorphisms.  We define 
\begin{align*}
\tau^{\leq n}\!\left(p_*\phim{f}\ICS_{X/G}(\mathbb{Q})\right):=&\lim_{i\mapsto \infty}\tau^{\leq n}(p_{i,*}\phim{f_i}\mathbb{Q}_{X_i})\otimes\mathbb{L}^{(\dim(G)-\dim(X))/2}\\
\tau^{\geq n}\!\left(p_!\phim{f}\ICS_{X/G}(\mathbb{Q})\right):=&\lim_{i\mapsto \infty}\tau^{\geq n}(p_{i,!}\phim{f_i}\mathbb{Q}_{X_i}\otimes\mathbb{L}^{-\dim(U_i)})\otimes\mathbb{L}^{(\dim(G)-\dim(X))/2}.
\end{align*}
Similarly, we define
\begin{align*}
\tau^{\leq n}\!\left(p_*\ICS_{X/G}(\mathbb{Q})\right):=&\lim_{i\mapsto \infty}\tau^{\leq n}\left(p_{i,*}\mathbb{Q}_{X_i}\right)\otimes\mathbb{L}^{(\dim(G)-\dim(X))/2}\\
\tau^{\geq n}\!\left(p_!\ICS_{X/G}(\mathbb{Q})\right):=&\lim_{i\mapsto \infty}\tau^{\geq n}\left(p_{i,!}\mathbb{Q}_{X_i}\otimes\mathbb{L}^{-\dim(U_i)}\right)\otimes\mathbb{L}^{(\dim(G)-\dim(X))/2}.
\end{align*}
This can be seen as a special case of the previous definition, with $f=0$.
\smallbreak
Let $Z\subset X$ be a subvariety, preserved by the $G$-action, and denote by $\iota\colon Z/G\hookrightarrow X/G$ the inclusion of stacks.  We obtain inclusions
\begin{align*}
\iota_i \colon &Z_i:=\big( U_i\cap (Z\times V_i)\big)/G\rightarrow X_i
\end{align*}
and we define the restricted pushforward of vanishing cycle cohomology
\begin{align*}
\tau^{\leq n}\!\left(p_*\iota_*\iota^!\phim{f}\ICS_{X/G}(\mathbb{Q})\right):=&\lim_{i\mapsto \infty}\tau^{\leq n}\!\left(p_{i,*}\iota_{i,*}\iota_i^!\phim{f_i}\mathbb{Q}_{X_i}\right)\otimes\mathbb{L}^{(\dim(G)-\dim(X))/2}\\
\tau^{\geq n}\!\left(p_!\iota_!\iota^*\phim{f}\ICS_{X/G}(\mathbb{Q})\right):=&\lim_{i\mapsto \infty}\tau^{\geq n}\!\left(p_{i,!}\iota_{i,!}\iota_i^*\phim{f_i}\mathbb{Q}_{X_i}\otimes\mathbb{L}^{-\dim(U_i)}\right)\otimes\mathbb{L}^{(\dim(G)-\dim(X))/2}.
\end{align*}
As a particular case, setting $Y$ to be a point, we obtain
\[
\HO^n_c\!\left(Z/G,\phim{f}\ICS_{X/G}(\mathbb{Q})\right):=\lim_{i\mapsto \infty}\HO^n_c\!\left(Z_i,\iota_i^*\phim{f_i}\mathbb{Q}_{X_i}\otimes\mathbb{L}^{-\dim(U_i)}\right)\otimes\mathbb{L}^{(\dim(G)-\dim(X))/2}.
\]
\subsection{Dimensional reduction}
Assume that we are given a decomposition $X=X'\times\mathbb{A}^n$ of varieties, and that $\mathbb{C}^*$ acts on $X$ via the product of the trivial action on $X'$, and the scaling action on $\mathbb{A}^n$.  Assume that $f$ has weight one.  Denote by $\pi\colon X\rightarrow X'$ the natural projection.  Then we can write
\[
f=\sum_{1\leq i\leq n}\pi^*f_i\cdot x_i
\]
where $f_i$ are functions on $X'$, and $x_i$ are coordinates for $\mathbb{A}^n$.  Define 
\[
Z'=Z(f_1,\ldots,f_n)
\]
to be the shared vanishing locus of all the functions $f_1,\ldots,f_n$, and denote 
\[
Z=\pi^{-1}(Z').
\]
Note that $Z\subset X_0:=f^{-1}(0)$, and so we can postcompose the canonical natural transformation $\nu^{\mon}_f\colon \phim{f}\rightarrow (X_0\rightarrow X)_*(X_0\rightarrow X)^*$ with the restriction map 
\[
(X_0\rightarrow X)_*(X_0\rightarrow X)^*\rightarrow (Z\rightarrow X)_*(Z\rightarrow X)^*
\]
to obtain a natural transformation 
\[
\upsilon\colon \phim{f}\rightarrow (Z\rightarrow X)_*(Z\rightarrow X)^*.  
\]
\begin{theorem}\cite[Thm.A.1]{Chicago2}
\label{dimRedThm}
The natural transformation $\pi_!\upsilon\pi^*$ is a natural isomorphism.  
\end{theorem}
This is a cohomological analogue of the dimensional reduction theorem of \cite{BBS}.  It implies (see \cite[Cor.A.7]{Chicago2}) that if $X$ is the total space of a $G$-equivariant affine fibration $\pi\colon X\rightarrow X'$ for $G$ an algebraic group, and $S\subset X'$ is a $G$-invariant subspace of the base, there is a natural isomorphism in compactly supported cohomology
\[
\HO_c\left(\pi^{-1}(S)/G,\phim{f}\mathbb{Q}_{X/G}\right)\cong \HO_c\left((S\cap Z')/G,\mathbb{Q}\right)\otimes\mathbb{L}^{\dim(\pi)}.
\]
\begin{remark}
\label{getOut}
The natural transformation $\pi_!\upsilon\pi^*$ is considered as a natural transformation between two functors $\Db(\MHM(X'))\rightarrow \Db(\MMHM(X'))$ (see Remark \ref{canlift}).  However, the target functor is defined as such a functor via the embedding $\Db(\MHM(X'))\rightarrow \Db(\MMHM(X'))$.  So Theorem \ref{dimRedThm} implies that under suitable equivariance conditions, the monodromy on $\pi_!\phim{f}\pi^*$ is trivial, and we can replace $\pi_!\phim{f}\pi^*$ with the more standard functor $\pi_!\phi_f\pi^*$.
\end{remark}

\subsection{Integrality and PBW isomorphisms}\label{CoDT}
Let $Q$ be a finite quiver.  We consider $\mathbb{N}^{Q_0}$-graded monodromic mixed Hodge structures as monodromic mixed Hodge modules on the space $\mathbb{N}^{Q_0}$ in the obvious way: a monodromic mixed Hodge module on a point is just a monodromic mixed Hodge structure, and $\mathbb{N}^{Q_0}$ is a union of points $\dd\in\mathbb{N}^{Q_0}$, and so a monodromic mixed Hodge module on $\mathbb{N}^{Q_0}$ is given by a formal direct sum
\[
\bigoplus_{\dd\in\mathbb{N}^{Q_0}}\mathcal{L}_{\dd}
\]
of monodromic mixed Hodge structures.  
\smallbreak
The GIT quotient $\Msp(Q)^{\zeta\sst}$ is a monoid with monoid morphism
\[
\Msp(Q)^{\zeta\sst}\times \Msp(Q)^{\zeta\sst}\xrightarrow{\oplus}\Msp(Q)^{\zeta\sst}
\]
taking a pair of polystable representations $\rho,\rho'$ to their direct sum $\rho\oplus\rho'$.  This morphism is finite \cite[Lem.2.1]{MeRe14}.  A unit for the monoid morphism is provided by the inclusion $\Msp(Q)^{\zeta\sst}_0\hookrightarrow \Msp(Q)^{\zeta\sst}$, which at the level of complex points, corresponds to the inclusion of the zero module.  The morphism $\dim^\zeta\colon \Msp(Q)^{\zeta\sst}\rightarrow \mathbb{N}^{Q_0}$, taking a representation to its dimension vector, is a morphism of monoids, where the morphism
\[
+\colon \mathbb{N}^{Q_0}\times\mathbb{N}^{Q_0}\rightarrow\mathbb{N}^{Q_0}
\]
is the usual addition map.  If $W$ is a potential for $Q$, there is an induced function $\mathcal{T}r(W)^{\zeta}\colon \Msp(Q)^{\zeta\sst}\rightarrow\mathbb{C}$ such that the following diagram commutes
\[
\xymatrix{
\Mst(Q)^{\zeta\sst}\ar[dr]^{\mathfrak{Tr}(W)^\zeta}\ar[d]_{\p^{\zeta}}\\
\Msp(Q)^{\zeta\sst}\ar[r]_{\mathcal{T}r(W)^\zeta}&\mathbb{C}.
}
\]
If $X$ is a commutative monoid in the category of locally finite type complex schemes, with finite type monoid morphism
\[
\tau\colon X\times X\rightarrow X,
\]
then by \cite[Thm.1.9]{MSS11} the categories $\Du(\MMHM(X))$, and $\Dl(\MMHM(X))$ of Definition \ref{DerDef} carry symmetric monoidal structures defined by
\[
\mathcal{F}\boxtimes_{\tau}\mathcal{G}:=\tau_*(\mathcal{F}\boxtimes\mathcal{G}).
\]
In particular, the categories $\Du\big(\MMHM(\Msp(Q)^{\zeta\sst})\big)$ and $\Dl\big(\MMHM(\Msp(Q)^{\zeta\sst})\big)$ carry symmetric monoidal structures defined by 
\[
\mathcal{F}\boxtimes_{\oplus}\mathcal{G}:=\oplus_*(\mathcal{F}\boxtimes\mathcal{G}).
\]
The following theorem allows for the definition of BPS sheaves and BPS cohomology.  It is a cohomological lift of the property known in DT theory as \textit{integrality}.
\begin{theorem}\cite[Thm.A]{DaMe15b}
\label{IntThm}
Fix a QP $(Q,W)$ such that $\crit(\Tr(W))\subset \Tr(W)^{-1}(0)$, a slope $\theta\in(-\infty,\infty)$, and a $\theta$-generic stability condition $\zeta$.  For nonzero $\dd\in\Lambda_{\theta}^{\zeta}$, where $\Lambda_{\theta}^{\zeta}\subset\mathbb{N}^{Q_0}$ is as in (\ref{LambdaDef}), define the monodromic mixed Hodge module
\[
\DTS_{Q,W,\dd}^{\zeta}\in\Ob\!\left(\MMHM\left(\Msp(Q)^{\zeta\sst}_{\dd}\right)\right)
\]
by
\begin{equation}
\label{DTSdef}
\DTS_{Q,W,\dd}^{\zeta}=\begin{cases} \phim{\mathcal{T}r(W)^{\zeta}_{\dd}}\ICS_{\Msp(Q)^{\zeta\sst}_{\dd}}(\mathbb{Q})&\textrm{if }\Msp(Q)^{\zeta\st}_{\dd}\neq \emptyset\\
0&\textrm{otherwise,}\end{cases}
\end{equation}
and define $\DTS_{Q,W,\theta}^{\zeta}:=\bigoplus_{\dd\in\Lambda_{\theta}^{\zeta}}\DTS_{Q,W,\dd}^{\zeta}$.  Then there are isomorphisms of objects in $\Du\!\left(\MMHM\!\left(\Msp(Q)_{\theta}^{\zeta\sst}\right)\right)$ and $\Dl\!\left(\MMHM\!\left(\Msp(Q)_{\theta}^{\zeta\sst}\right)\right)$ respectively:
\begin{align}\label{starInt}
\p_{\theta,*}^{\zeta}\phin{\WWW^{\zeta}_{\theta}}\ICS_{\Mst(Q)^{\zeta\sst}_{\theta}}(\mathbb{Q})\cong&\Sym_{\boxtimes_{\oplus}}\!\left(\DTS_{Q,W,\theta}^{\zeta}\otimes\HO(\B\mathbb{C}^*,\mathbb{Q})_{\vir}\right),
\\ \label{shriekInt}
\p_{\theta,!}^{\zeta}\phin{\WWW^{\zeta}_{\theta}}\ICS_{\Mst(Q)^{\zeta\sst}_{\theta}}(\mathbb{Q})\cong&\Sym_{\boxtimes_{\oplus}}\!\left(\DTS_{Q,W,\theta}^{\zeta}\otimes\HO_c(\B\mathbb{C}^*,\mathbb{Q})_{\vir}\right).
\end{align}
\end{theorem}
Since Verdier duality naturally commutes with the vanishing cycles functor, and $\ICS_{\Msp(Q)^{\zeta\sst}_{\dd}}(\mathbb{Q})$ is Verdier self-dual, the BPS sheaf is Verdier self-dual: there is an isomorphism
\begin{equation}
\label{nvsd}
\DTS_{Q,W,\dd}^{\zeta}\cong\mathbb{D}\DTS_{Q,W,\dd}^{\zeta}.
\end{equation}
\begin{remark}
Technically, the result quoted from \cite{DaMe15b} is stated for $\Ho(\bullet)$ of the LHS of \eqref{starInt} and \eqref{shriekInt}, and not the LHS as stated above.  That there is an isomorphism $\Ho(\mathrm{LHS})\cong \mathrm{LHS}$ is a consequence of approximation by projective morphisms and the decomposition theorem, see Corollary \ref{hpc}.
\end{remark}
\subsubsection{(3d) BPS cohomology}
\label{DTSS}
Now let $\Sp$ be a Serre subcategory of the category of $\mathbb{C}Q$-modules.  Recall that we denote by $\iota'\colon\Msp(Q)^{\Sp,\zeta\sst}\hookrightarrow \Msp(Q)^{\zeta\sst}$ the inclusion of objects in $\Sp$.  We define the \textit{BPS cohomology}:
\begin{align*}
\DT^{\Sp,\zeta}_{Q,W,\dd}\coloneqq &\HO\!\left(\Msp(Q)^{\Sp,\zeta\sst}_{\dd},\iota'^!\DTS^{\zeta}_{Q,W,\dd}\right)\\
\cong&\HO_c\!\left(\Msp(Q)^{\Sp,\zeta\sst}_{\dd},\iota'^*\DTS^{\zeta}_{Q,W,\dd}\right)^{\vee}
\end{align*}
where the isomorphism follows from Verdier self-duality of the BPS sheaf.
\subsubsection{}
The cohomologically graded mixed Hodge structure
\[
\Coha_{Q,\theta}^{\Sp,\zeta}\coloneqq \Dim^{\zeta}_*\iota_*\iota^!\phin{\WWW^{\zeta}_{\theta}}\ICS_{\Mst(Q)^{\zeta\sst}_{\theta}}(\mathbb{Q})
\]
carries a Hall algebra multiplication, defined in \cite{COHA,Chicago2}, via pullback and pushforward of vanishing cycle sheaves; see \S\ref{KSCSec} for a generalisation of the construction.  Applying the natural transformation $\tau^{\leq 1}\rightarrow \id$ to \eqref{starInt} we obtain the morphism
\begin{equation}
\label{BPSIN}
\DTS_{Q,W,\theta}^{\zeta}\otimes\LL^{1/2}\rightarrow \p_{\theta,*}^{\zeta}\phin{\WWW^{\zeta}_{\theta}}\ICS_{\Mst(Q)^{\zeta\sst}_{\theta}}(\mathbb{Q}).
\end{equation}
Applying $\HO\iota'_*\iota'^!$ to \eqref{BPSIN}, we obtain the embedding
\[
\DT_{Q,W,\theta}^{\Sp,\zeta}\otimes\LL^{1/2}\hookrightarrow \Coha_{Q,W,\theta}^{\Sp,\zeta}.
\]
Since $\HO_{\mathbb{C}^*}$ acts on the target, this extends to a morphism
\[
g\colon \DT_{Q,W,\theta}^{\Sp,\zeta}\otimes\HO(\B\mathbb{C}^*,\mathbb{Q})_{\vir}\rightarrow \Coha_{Q,W,\theta}^{\Sp,\zeta}.
\]
\begin{theorem}[PBW theorem]\cite[Thm.C]{DaMe15b}\label{PBW3d}
The morphism
\[
\Sym(\DT_{Q,W,\theta}^{\Sp,\zeta}\otimes\HO(\B\mathbb{C}^*,\mathbb{Q})_{\vir})\rightarrow \Coha_{Q,W,\theta}^{\Sp,\zeta}
\] extending $g$ via the Hall algebra multiplication on the target is an isomorphism.
\end{theorem}

\subsection{Framed moduli spaces and hidden properness}
\label{framedVer}

Recall that the left hand sides of (\ref{starInt}) and (\ref{shriekInt}) are defined with respect to a chain $\ldots\subset U_i\subset U_{i+1}\subset\ldots$ of $\G_{\dd}$-equivariant varieties satisfying the conditions of Section \ref{pfs}.  In this subsection we recall a natural choice of such a chain, for which the $U_i$ themselves have a representation theoretic definition.  Via this choice of $U_i$, the morphism $\p$ behaves ``like a proper map'', although it is certainly not proper in the traditional sense.
\smallbreak
Let $Q$ be a quiver.  For the moment we do not assume that $Q$ is symmetric.  Let $\dd,\ff\in\mathbb{N}^{Q_0}$ be a pair of dimension vectors.  Following \cite[Sec.3.3]{DaMe15b} we define $Q_{\ff}$ to be the quiver obtained from $Q$ by setting
\begin{align*}
(Q_{\ff})_0:=&Q_0\cup \{\infty\}\\
(Q_{\ff})_1:=& Q_1\cup\{\beta_{i,m}|i\in Q_0,\hbox{ }1\leq m\leq \ff_i\}
\end{align*}
and $s(\beta_{i,m})=\infty$, $t(\beta_{i,m})=i$.  Given a King stability condition $\zeta$ for $Q$, and a slope $\theta\in({}-\infty,\infty)$, we extend $\zeta$ to a stability condition $\zeta^{(\theta)}$ for $Q_{\ff}$ by fixing the slope
\[
- \Re e (\zeta^{(\theta)}_{\infty})/ \Im m (\zeta^{(\theta)}_{\infty})=\theta+\epsilon
\]
for sufficiently small positive $\epsilon$.  Let $\dd\in\Lambda_{\theta}^{\zeta}$.  Then a $(1,\dd)$-dimensional representation $\rho$ of $Q_{\ff}$ is $\zeta^{(\theta)}$-semistable if and only if it is $\zeta^{(\theta)}$-stable, and this holds if and only if the underlying $Q$-representation of $\rho$ is $\zeta$-semistable, and for all proper $Q_{\ff}$-subrepresentations $\rho'\subset \rho$, if $\dim(\rho')_{\infty}=1$ then the underlying $Q$-representation of $\rho'$ has slope strictly less than $\theta$.
\sbrk
Define $V_{\ff,\dd}:=\prod_{i\in Q_0}\Hom(\mathbb{C}^{\ff_i},\mathbb{C}^{\dd_i})$.  Then $V_{\ff,\dd}$ carries a $\Gl_{\dd}$-action, given by the product of the $\Gl_{\dd_i}(\mathbb{C})$-actions on $\mathbb{C}^{\dd_i}$.  Furthermore, there is an obvious decomposition
\[
X(Q_{\ff})_{(1,\dd)}\cong X(Q)_{\dd}\times V_{\ff,\dd}.
\]
If $L,L'$ are vector spaces, we define $\Hom^{\surj}(L,L')\subset \Hom(L,L')$ to be the subvariety of surjective homomorphisms.  Then the subspace 
\[
S_{\ff,\dd}:=\prod_{i\in Q_0}\Hom^{\surj}(\mathbb{C}^{\ff_i},\mathbb{C}^{\dd_i})\subset V_{\ff,\dd}
\]
is acted on freely by $\Gl_{\dd}$, and there is a chain of $\Gl_{\dd}$-equivariant inclusions of open dense subvarieties over the coarse moduli space $\Msp(Q)^{\zeta\sst}_{\dd}$
\[
\xymatrix{
X(Q)^{\zeta\sst}_{\dd}\times S_{\ff,\dd}\ar[drr]\ar@{^(->}[r]&X(Q_{\ff})_{(1,\dd)}^{\zeta^{(\theta)}\sst}\ar[dr]\ar@{^(->}[r]&X(Q)^{\zeta\sst}_{\dd}\times V_{\ff,\dd}\ar[d]
\\
&&\Msp(Q)_{\dd}^{\zeta\sst}.
}
\]
%The first of these inclusions exists because, considered as $Q_{\ff}$-representations, the points of $X(Q)^{\zeta\sst}_{\dd}\times S_{\ff,\dd}$ correspond to those representations $\rho$ such that the underlying $Q$-representation is $\zeta$-semistable, and is spanned as a vector space by $\beta_{i,t}\cdot v$, where $i\in Q_0$, $t\in[1,\ldots,\ff(i)]$ and $v\in\rho(\infty)$.  
The group $\Gl_{\dd}$ acts freely on $X(Q_{\ff})^{\zeta^{(\theta)}\sst}_{(1,\dd)}$.  In the notation of the start of the section, we may set $U_i=X(Q_{i\cdot (1,\ldots,1)})^{\zeta^{(\theta)}\sst}_{(1,\dd)}$ to obtain our promised chain of $\Gl_{\dd}$-equivariant varieties.
\begin{definition}
We denote by $\Msp(Q)_{\ff,\dd}^{\zeta}=X(Q_{\ff})^{\zeta^{(\theta)}\sst}_{(1,\dd)}/\Gl_{\dd}$ the fine moduli space of $\ff$-framed $\zeta$-semistable representations of $Q$ of dimension $\dd$, or in other words, the fine moduli space of $\zeta^{(\theta)}$-stable $(1,\dd)$-dimensional representations of $Q_{\ff}$.  
\end{definition}
We denote by 
\begin{equation}
\label{frDef}
\pi_{\ff,\dd}^{\zeta}\colon \Msp(Q)_{\ff,\dd}^{\zeta}\rightarrow \Msp(Q)_{\dd}^{\zeta\sst}
\end{equation}
the induced map from the quotient.  
\begin{proposition}
The map $\pi_{\ff,\dd}$ above is proper.
\end{proposition}
\begin{proof}
This is standard, and follows from the valuative criterion of properness and the fact that in the following diagram over the common affinization of the domain and target of $\pi_{\ff,\dd}^{\zeta}$
\[
\xymatrix{
\Msp(Q)_{\ff,\dd}^{\zeta}\ar[d]^{\pi_{\ff,\dd}^{\zeta}}\ar[rd]\\
\Msp(Q)_{\dd}^{\zeta\sst}\ar[r]&\Msp(Q)_{\dd}
}
\]
the unmarked arrows are GIT maps, and hence proper.
\end{proof}

We define 
\[
\WW_{\ff,\dd}^{\zeta}:=\WW^{\zeta}_{\dd}\circ \pi^{\zeta}_{\ff,\dd}.  
\]
We write $\ff\mapsto \infty$ to mean that all of the individual entries of $\ff$ tend to $\infty$.  Then as $\ff\mapsto \infty$, since 
\[
\codim_{ V_{\ff,\dd}}( V_{\ff,\dd}\setminus S_{\ff,\dd})\mapsto \infty, \]
we deduce that 
\[
\codim_{X(Q)_{\dd}\times V_{\ff,\dd}}\!\left((X(Q)_{\dd}\times V_{\ff,\dd})\setminus X(Q_{\ff})^{\zeta^{(\theta)}\sst}_{(1,\dd)}\right)\mapsto \infty,
\]
and so 
\begin{equation}
\label{limIso}
\Ho\!\left(\p^{\zeta}_{\dd,!}\phim{\WWW^{\zeta}_{\dd}}\ICS_{\Mst(Q)_{\dd}^{\zeta\sst}}(\mathbb{Q})\right)=\lim_{\ff\mapsto \infty} \Ho\!\left(\pi^{\zeta}_{\ff,\dd,!}\phim{\WW^{\zeta}_{\ff,\dd}}\ICS_{\Msp(Q)_{\ff,\dd}^{\zeta}}(\mathbb{Q})\otimes\LL^{-\ff\cdot \dd/2}\right),
\end{equation}
as per the definition in Section \ref{pfs}.  Equation (\ref{limIso}) states that the cohomology of $\p^{\zeta}_{\dd,!}\phim{\WWW^{\zeta}_{\dd}}\ICS_{\Mst(Q)_{\dd}^{\zeta\sst}}(\mathbb{Q})$ is obtained as a limit of direct images of related vanishing cycle complexes along projective morphisms from smooth complex varieties.  It is in this sense that $\p^{\zeta}$ is ``approximated by proper maps'', and the outcome is that many theorems regarding projective morphisms are true of $\p^{\zeta}$.  For instance, it follows from the $W=0$ case of equation (\ref{limIso}) and the decomposition theorem of Beilinson, Bernstein, Deligne and Gabber, that $\p^{\zeta}_{\dd,!}\ICS_{\Mst(Q)_{\dd}^{\zeta\sst}}(\mathbb{Q})$ is pure.
\smallbreak
\begin{lemma}
\label{decomp2}
Let $Q$ be quiver, let $\zeta$ be a stability condition on $Q$, let $W\in\mathbb{C}Q/[\mathbb{C}Q,\mathbb{C}Q]$ be a potential, and let $\dd,\ff\in\mathbb{N}^{Q_0}$ be a pair of dimension vectors.  As above, we let $\pi^{\zeta}_{\ff,\dd}\colon\Msp(Q)^{\zeta}_{\ff,\dd}\rightarrow \Msp(Q)^{\zeta\sst}_{\dd}$ be the forgetful map taking a stable framed $\zeta$-semistable representation to its underlying $\zeta$-polystable representation.  Then there is an isomorphism
\begin{equation}
\label{isotc}
\pi_{\ff,\dd,!}^{\zeta}\phim{\mathcal{T}r(W)_{\ff,\dd}^{\zeta}}\mathbb{Q}_{\Msp(Q)_{\ff,\dd}^{\zeta}}\cong\Ho\!\left(\pi_{\ff,\dd,!}^{\zeta}\phim{\mathcal{T}r(W)_{\ff,\dd}^{\zeta}}\mathbb{Q}_{\Msp(Q)_{\ff,\dd}^{\zeta}}\right),
\end{equation}
i.e. the left hand side of (\ref{isotc}) is isomorphic to its total cohomology.
\end{lemma}
\begin{proof}
This follows from the existence of the chain of isomorphisms
\begin{align*}
\pi_{\ff,\dd,!}^{\zeta}\phim{\mathcal{T}r(W)_{\ff,\dd}^{\zeta}}\mathbb{Q}_{\Msp(Q)_{\ff,\dd}^{\zeta}}\cong&\phim{\mathcal{T}r(W)_{\dd}^{\zeta}}\pi_{\ff,\dd,!}^{\zeta}\mathbb{Q}_{\Msp(Q)_{\ff,\dd}^{\zeta}}&\substack{\textrm{commutativity of vanishing}\\ \textrm{cycles with proper maps}}\\
\cong&\phim{\mathcal{T}r(W)_{\dd}^{\zeta}}\Ho\!\left(\pi_{\ff,\dd,!}^{\zeta}\mathbb{Q}_{\Msp(Q)_{\ff,\dd}^{\zeta}}\right)&\substack{\textrm{the BBDG}\\\textrm{decomposition theorem}}\\
\cong&\Ho\!\left(\phim{\mathcal{T}r(W)_{\dd}^{\zeta}}\pi_{\ff,\dd,!}^{\zeta}\mathbb{Q}_{\Msp(Q)_{\ff,\dd}^{\zeta}}\right)&\substack{\textrm{exactness of vanishing}\\ \textrm{cycles functor}}\\
\cong&\Ho\!\left(\pi_{\ff,\dd,!}^{\zeta}\phim{\mathcal{T}r(W)_{\ff,\dd}^{\zeta}}\mathbb{Q}_{\Msp(Q)_{\ff,\dd}^{\zeta}}\right)&\substack{\textrm{commutativity of vanishing}\\ \textrm{cycles with proper maps.}}
\end{align*}
\end{proof}

Lemma \ref{decomp2} can be thought of as saying that ``one half'' of the BBDG decomposition theorem is true, even with the introduction of the vanishing cycles functor (which may destroy purity, i.e. the other half of the theorem).
\begin{corollary}
\label{hpc}
Let $Q$ be a quiver, let $\zeta$ be a stability condition on $Q$, let $W\in\mathbb{C}Q/[\mathbb{C}Q,\mathbb{C}Q]$ be a potential, and let $\dd\in\mathbb{N}^{Q_0}$ be a dimension vector.  Then there are isomorphisms in $\Dl(\MMHM(\Msp(Q)^{\zeta\sst}_{\dd}))$ and $\Du(\MMHM(\Msp(Q)^{\zeta\sst}_{\dd}))$ respectively:
\begin{align*}
\p_{\theta,*}^{\zeta}\phim{\WWW^{\zeta}_{\theta}}\ICS_{\Mst(Q)^{\zeta\sst}_{\theta}}(\mathbb{Q})\cong&\Ho\!\left(\p_{\theta,*}^{\zeta}\phim{\WWW^{\zeta}_{\theta}}\ICS_{\Mst(Q)^{\zeta\sst}_{\theta}}(\mathbb{Q})\right),
\\
\p_{\theta,!}^{\zeta}\phim{\WWW^{\zeta}_{\theta}}\ICS_{\Mst(Q)^{\zeta\sst}_{\theta}}(\mathbb{Q})\cong&\Ho\!\left(\p_{\theta,!}^{\zeta}\phim{\WWW^{\zeta}_{\theta}}\ICS_{\Mst(Q)^{\zeta\sst}_{\theta}}(\mathbb{Q})\right).
\end{align*}
\end{corollary}

\begin{proposition}
Let $\zeta$ be a $\theta$-generic stability condition, and assume that $\crit(\Tr(W))\subset \Tr(W)^{-1}(0)$.  There is an isomorphism in the category $\Dl\!\left(\MMHM\!\left(\Msp(Q)_{\theta}^{\zeta\sst}\right)\right)$
\begin{align}\label{piIso}
&\pi^{\zeta}_{Q,\ff,\theta,!}\!\left(\bigoplus_{\dd\in\Lambda_{\theta}^{\zeta}}\phim{\WW^{\zeta}_{\ff,\dd}}\QQ_{\Msp(Q)_{\ff,\dd}^{\zeta}}\otimes \LL^{(\dd,\dd)_Q/2}\right)\cong 
\\
&\Sym_{\boxtimes_{\oplus}}\!\left(\bigoplus_{\dd\in\Lambda_{\theta}^{\zeta}}\DTS_{Q,W,\dd}^{\zeta}\otimes \HO(\mathbb{C}\mathbb{P}^{\ff\cdot \dd-1},\mathbb{Q})^{\vee}\otimes\LL^{ -1/2} \right).\nonumber
\end{align}
\end{proposition}
\begin{proof}
By \cite[Prop.4.3]{MeRe14} there is an equality in the Grothendieck group of mixed Hodge modules on $\Msp(Q)^{\zeta\sst}_{\theta}$:
\begin{align}
\label{groCase1}
&\left[\pi^{\zeta}_{\ff,\theta,!}\bigoplus_{\dd\in\Lambda_{\theta}^{\zeta}}\left(\QQ_{\Msp(Q)_{\ff,\dd}^{\zeta}}\otimes \LL^{(\dd,\dd)_Q/2}\right)\right]=\\&\left[\Sym_{\boxtimes_{\oplus}}\!\left(\bigoplus_{\dd\in\Lambda_{\theta}^{\zeta}}\ICS_{\Msp(Q)^{\zeta\sst}_{\dd}}(\QQ)\otimes \HO(\mathbb{C}\mathbb{P}^{\ff\cdot \dd-1},\mathbb{Q})^{\vee}\otimes\LL^{ -1/2} \right)\right].
\label{groCase2}
\end{align}
On the other hand, the terms in square brackets in \eqref{groCase1} and \eqref{groCase2} are pure; the left hand term is pure since $\pi_{\ff,\theta}^{\zeta}$ is proper, and purity is preserved by direct image along proper maps \cite[p.324]{Saito89}, while the right hand side is pure since it is generated by pure mixed Hodge modules, and the map $\oplus\colon\Msp(Q)^{\zeta}_{\theta}\times \Msp(Q)^{\zeta}_{\theta}\rightarrow \Msp(Q)^{\zeta}_{\theta}$ is proper by \cite[Lem.2.1]{MeRe14}.  Since the category of pure mixed Hodge modules on a complex variety is Krull--Schmidt \cite{Saito89} this equality implies that there is an isomorphism 
\begin{align}
\label{isoWO}
&\pi^{\zeta}_{\ff,\theta,!}\bigoplus_{\dd\in\Lambda_{\theta}^{\zeta}}(\QQ_{\Msp(Q)_{\ff,\dd}^{\zeta}}\otimes \LL^{(\dd,\dd)/2})\cong\Sym_{\boxtimes_{\oplus}}\!\left(\bigoplus_{\dd\in\Lambda_{\theta}^{\zeta}}\ICS_{\Msp(Q)^{\zeta\sst}_{\dd}}(\QQ)\otimes \HO(\mathbb{C}\mathbb{P}^{\ff\cdot \dd-1},\mathbb{Q})^{\vee}\otimes\LL^{ -1/2} \right).
\end{align}
The proposition follows from applying $\phim{\WW_{\theta}^{\zeta}}$ to both sides of (\ref{isoWO}), and using the fact that the vanishing cycle functor commutes with taking direct image along proper maps \cite[Thm.2.14]{Saito89}, as well as commuting with the monoidal structure $\boxtimes_{\oplus}$ on $\Dl(\MMHM(\Msp(Q)_{\theta}^{\zeta\sst}))$ by Saito's version of the Thom--Sebastiani theorem \cite{Saito10}, as well as the enhancement of this monoidal structure to a symmetric monoidal structure, by \cite[Prop.3.8]{DaMe15b}.
\end{proof}

\section{Purity and supports}
In this section we prove Theorem \ref{purityThm}.  A crucial role in the proof is played by the support lemma (Lemma \ref{lem1}), which also enables us to prove Theorem \ref{2dbpsthm}.
\label{mainProof}
\subsection{Proof of Theorem \ref{purityThm}}
Fix a quiver $Q$.  We define $(\tilde{Q},\tilde{W})$ as in Section \ref{QandP}.  Define 
\begin{align}\label{DTdef}
\DT_{\tilde{Q},\tilde{W}}^{\vee}:=&\dim_!\DTS_{\tilde{Q},\tilde{W}}=\HO_c(\Msp(\tilde{Q}),\DTS_{\tilde{Q},\tilde{W}})\\
\DT_{\tilde{Q},\tilde{W}}^{\omega\nnilp,\vee}:=&\dim_!\!\left(\DTS_{\tilde{Q},\tilde{W}}|_{\Msp(\tilde{Q})^{\omega\nnilp}}\right)\nonumber
\end{align}
the compactly supported cohomology, and the restricted compactly supported cohomology, respectively, of the BPS sheaf from Theorem \ref{IntThm}.  As in Definition \ref{dimDef}, $\dim\colon\Msp(\tilde{Q})\rightarrow \mathbb{N}^{Q_0}$ is the map taking a semisimple representation to its dimension vector.  Note that no stability condition appears in \eqref{DTdef} (see Convention \ref{stabConv}).  As explained at the beginning of Section \ref{CoDT}, we consider a $\mathbb{N}^{Q_0}$-graded mixed Hodge structure as essentially the same thing as a mixed Hodge module on $\mathbb{N}^{Q_0}$, and so we consider both of the above objects equivalently as mixed Hodge module complexes on the discrete space $\mathbb{N}^{Q_0}$, or $\mathbb{N}^{Q_0}$-graded mixed Hodge structures.
\smallbreak
We break the proof of Theorem \ref{purityThm} into several steps.  We use the following three lemmas.
\begin{lemma}[Support lemma]
\label{lem1}
Let $x\in\Msp(\tilde{Q})^{\zeta\sst}_{\dd}$ lie in the support of $\DTS^{\zeta\sst}_{\tilde{Q},\tilde{W},\dd}$, corresponding to a $\dd$-dimensional semisimple $\mathbb{C}\tilde{Q}$ representation $\rho$.  Then the union of the multisets $\cup_{i\in Q_0}\{\lambda_{i,1},\ldots,\lambda_{i,\dd_i}\}$ of generalised eigenvalues of $\rho(\omega_i)$ contains only one distinct element $\lambda\in\mathbb{C}$.  Furthermore, the action of $\sum_{i\in Q_0} \omega_i$ on the underlying vector space of $\rho$ is by multiplication by the constant $\lambda$.
\end{lemma}
Recall the definition of $\Dim=\dim\circ \p_{\tilde{Q}}:\Mst(\tilde{Q})\rightarrow\mathbb{N}^{Q_0}$ from Definition \ref{dimDef}: on $K$-points, with $\mathbb{K}$ a field extension of $\mathbb{C}$, it is the map taking a $K\tilde{Q}$-module to its dimension vector.
\begin{lemma}
\label{lem2}
There are isomorphisms of $\mathbb{N}^{Q_0}$-graded mixed Hodge structures
\begin{align}
\label{nonNilpEq}
\bigoplus_{\dd\in\mathbb{N}^{Q_0}}\HO_c\!\left(\mu_{Q,\dd}^{-1}(0)/\Gl_{\dd},\mathbb{Q}\right)\otimes\mathbb{L}^{(\dd,\dd)}\cong &\Dim_!\phin{\mathfrak{Tr}(\tilde{W})}\ICS_{\Mst(\tilde{Q})}(\mathbb{Q})\\\cong&\Sym_{\boxtimes_+}\!\left(\DT^{\omega\nnilp,\vee}_{\tilde{Q},\tilde{W}}\otimes\mathbb{L}\otimes\HO_c(\B\mathbb{C}^*,\mathbb{Q})_{\vir}\right) \label{non2}
\end{align}
and
\begin{equation}
\label{nilpEq}
\Dim_!\!\left((\phin{\mathfrak{Tr}(\tilde{W})}\ICS_{\Mst(\tilde{Q})}(\mathbb{Q}))|_{\Mst(\tilde{Q})^{\omega\nnilp}}\right)\cong\Sym_{\boxtimes_+}\!\left(\DT^{\omega\nnilp,\vee}_{\tilde{Q},\tilde{W}}\otimes\HO_c(\B\mathbb{C}^*,\mathbb{Q})_{\vir}\right).
\end{equation}
\end{lemma}

\begin{lemma}
\label{lem3}
The $\mathbb{N}^{Q_0}$-graded mixed Hodge structure $\Dim_!\!\left((\phin{\mathfrak{Tr}(\tilde{W})}\ICS_{\Mst(\tilde{Q})}(\mathbb{Q}))|_{\Mst(\tilde{Q})^{\omega\nnilp}}\right)$ is pure, of Tate type.
\end{lemma}
Assuming Lemmas \ref{lem1}, \ref{lem2} and \ref{lem3}, we argue as follows.
\renewcommand*{\proofname}{Proof of Theorem \ref{purityThm}}
\begin{proof}
First, note that a graded mixed Hodge structure $\mathcal{F}$ is pure, of Tate type, if and only if $\Sym(\mathcal{F})$ is.  Lemma \ref{lem3} and (\ref{nilpEq}) thus imply that $\DT^{\omega\nnilp,\vee}_{\tilde{Q},\tilde{W}}$ is pure, of Tate type.  A tensor product of pure mixed Hodge modules is pure, and so $\DT^{\omega\nnilp,\vee}_{\tilde{Q},\tilde{W}}\otimes\mathbb{L}$ is also pure, of Tate type.  It follows from (\ref{nonNilpEq}) and (\ref{non2}) that $\Dim_!\phin{\WWW}\ICS_{\Mst(\tilde{Q})}(\mathbb{Q})$ and $\bigoplus_{\dd\in\mathbb{N}^{Q_0}}\HO_c(\mu_{Q,\dd}^{-1}(0)/\Gl_{\dd},\mathbb{Q})$ are pure, of Tate type, and the theorem follows.
\end{proof}
\medbreak
Before coming to the proof of Lemmas \ref{lem1}, \ref{lem2} and \ref{lem3} we note the following
\begin{corollary}
\label{BPSpure}
For arbitrary $Q$ and $\dd\in\mathbb{N}^{Q_0}$ the BPS cohomology $\DT_{\tilde{Q},\tilde{W},\dd}$ is pure of Tate type.
\end{corollary}
\renewcommand*{\proofname}{Proof}
\begin{proof}
From Theorem \eqref{PBW3d} we deduce that there is an inclusion of monodromic mixed Hodge structures
\[
\DT_{\tilde{Q},\tilde{W},\dd}\otimes\LL^{1/2}\hookrightarrow \HO_c(\Mst(\Pi_Q)_{\dd},\mathbb{Q})^{\vee}\otimes \LL^{-(\dd,\dd)}.
\]
By Theorem \ref{purityThm} the target is pure, of Tate type, and so $\DT_{\tilde{Q},\tilde{W},\dd}\otimes\LL^{ 1/2}$ is also pure, of Tate type.  It follows that the Tate twist $\DT_{\tilde{Q},\tilde{W},\dd}$ is pure, of Tate type.
\end{proof}
The proof of both Lemmas \ref{lem2} and \ref{lem3} will use the dimensional reduction theorem, recalled as Theorem \ref{dimRedThm}.  Let $Q^+$ be obtained from $\tilde{Q}$ by deleting all of the arrows $a^*$, and let $Q^{\op}$ be obtained from $\tilde{Q}$ by deleting all the arrows $a$ and all the loops $\omega_i$.  We decompose 
\[
X(\tilde{Q})_{\dd}=X(Q^+)_{\dd}\times X(Q^{\op})_{\dd}.
\]

If we let $\mathbb{C}^*$ act on $X(\tilde{Q})_{\dd}$ via the trivial action on $X(Q^+)_{\dd}$ and the weight one action on $X(Q^{\op})_{\dd}$, then $\Tr(\tilde{W})_{\dd}$ is $\mathbb{C}^*$-equivariant in the manner required to apply Theorem \ref{dimRedThm}.  In the notation of Theorem \ref{dimRedThm}, we have that $Z'\subset X(Q^+)_{\dd}$ is determined by the vanishing of the matrix valued functions, for $a\in Q_1$
\begin{equation}
\label{Xplus}
\partial W/\partial a^*= a\omega_{s(a)}-\omega_{t(a)}a.
\end{equation}

Concretely, the stack $Z'/\G_{\dd}$ is isomorphic to the stack of pairs $(\rho,f)$, where $\rho$ is a $\dd$-dimensional $Q$-representation, and $f\colon\rho\rightarrow\rho$ is an endomorphism in the category of $Q$-representations.
\smallbreak

We fix $X(Q^+)^{\omega\nnilp}_{\dd}\subset X(Q^+)_{\dd}$ to be the subspace of representations such that each $\rho(\omega_i)$ is nilpotent.  We deduce from Theorem \ref{dimRedThm} that there is a natural isomorphism in compactly supported cohomology
\begin{equation}
\label{firstDimRed}
\Dim_!\!\left((\phi_{\WWW}\ICS_{\Mst(\tilde{Q})_{\dd}}(\mathbb{Q}))|_{\Mst(\tilde{Q})^{\omega\nnilp}_{\dd}}\right)\cong \HO_c\!\left((Z'\cap X(Q^+)^{\omega\nnilp}_{\dd})/\Gl_{\dd},\mathbb{Q}\right).
\end{equation}
Lemma \ref{lem3} is proved by analyzing the right hand side of (\ref{firstDimRed}).  Note that there is no overall Tate twist in (\ref{firstDimRed}) --- the Tate twist in the definition of the left hand side is cancelled by the Tate twist appearing in Theorem \ref{dimRedThm}.  
\smallbreak 
The first isomorphism in Lemma \ref{lem2} is obtained in similar fashion.  Let $L\subset \tilde{Q}$ be the quiver obtained by deleting all of the arrows $a$ and $a^*$, for $a\in Q_1$.  Then we can decompose
\[
X(\tilde{Q})_{\dd}\cong X(\overline{Q})_{\dd}\times X(L)_{\dd},
\]
and let $\mathbb{C}^*$ act on $X(\tilde{Q})_{\dd}$ via the trivial action on $X(\overline{Q})_{\dd}$ and the scaling action on $X(L)_{\dd}$.  This time the role of $Z'$ in Theorem \ref{dimRedThm} is played by $\mu_{Q,\dd}^{-1}(0)\subset X(\overline{Q})_{\dd}$, and we deduce that
\begin{equation}
\label{secondDimRed}
\Dim_!\phi_{\mathfrak{Tr}(\tilde{W})_{\dd}}\ICS_{\Mst(\tilde{Q})_{\dd}}(\mathbb{Q})\cong \HO_c\!\left(\mu_{Q,\dd}^{-1}(0)/\Gl_{\dd},\mathbb{Q}\right)\otimes\mathbb{L}^{(\dd,\dd)}.
\end{equation}

\renewcommand*{\proofname}{Proof of Lemma \ref{lem3}}
\begin{proof}
This is \cite[Thm.3.4]{Chicago3}; we recall a sketch of the proof and refer the reader to \cite{Chicago3} for more details.  The space $Z'\cap X(Q^+)^{\omega\nnilp}_{\dd}\subset X(Q^+)_{\dd}$ is defined by the equations (\ref{Xplus}) and the condition that $\omega_i$ acts nilpotently, for every $i$.  It follows, as in the proof of Proposition \ref{basicEq}, that the stack $(Z'\cap X(Q^+)^{\omega\nnilp}_\dd)/\Gl_{\dd}$ is isomorphic to the stack for which the $\mathbb{C}$-points are pairs $(\rho, f)$, where $\rho$ is a $\dd$-dimensional $\mathbb{C}Q$-module, and 
\[
f\in\End_{\mathbb{C}Q\lmod}(\rho)
\]
is a nilpotent endomorphism of $\rho$.  This stack decomposes into finitely many disjoint locally closed strata indexed by multipartitions $\overline{\pi}$ of $\dd$ (i.e. $Q_0$-tuples $(\pi^{(1)},\ldots,\pi^{(n)})$ of partitions, such that $|\pi^{(i)}|=\dd_i$ for $i\in Q_0$), where a multipartition determines the Jordan normal form of each $\rho(\omega_i)$ in the obvious way.  We label these strata $\Mst_{\overline{\pi}}$.  Given a partition $\pi$ of a number $d$, containing each number $i$ a total of $\pi_i$-times, we define the $\mathbb{C}[x]$-module 
\[
N_{\pi}\coloneqq\prod_{i\geq 1}(\mathbb{C}[x]/(x^i))^{\oplus \pi_i}.
\]
So $\dim(N_{\pi})=d$.  Define 
\begin{align*}
X(Q)_{\overline{\pi}}&=\prod_{a\in Q_1}\Hom(N_{\pi^{(s(a))}},N_{\pi^{(t(a))}})\\
\Gl_{\pi}&=\prod_{i\in Q_0}\Aut(N_{\pi^{(i)}}).
\end{align*}
Then
\[
\Mst_{\overline{\pi}}\cong X(Q)_{\overline{\pi}}/\Gl_{\overline{\pi}}.
\]
Each of the stacks $\Mst_{\overline{\pi}}$ can thus be described as a stack-theoretic quotient of an affine space by a unipotent extension of a product of general linear groups, from which it follows that each $\HO_c(\Mst_{\overline{\pi}},\mathbb{Q})$ is pure, of Tate type.  It follows that the connecting maps in the long exact sequences of compactly supported cohomology associated to the stratification indexed by multipartitions are zero, and the resulting short exact sequences are split.  It follows by induction that $\HO_c\!\left(\left(Z'\cap X(Q^+)^{\omega\nnilp}_\dd\right)/\Gl_{\dd},\mathbb{Q}\right)$ is pure, of Tate type.
\end{proof}

\renewcommand*{\proofname}{Proof of Lemma \ref{lem2}}
\begin{proof}
Since the map $\dim\colon \Msp(\tilde{Q})\rightarrow\mathbb{N}^{Q_0}$ is a morphism of commutative monoids, with proper monoid maps $\oplus$ and $+$ respectively, by \cite[Sec.1.12]{MSS11} there is a natural equivalence of functors
\[
\dim_!\Sym_{\boxtimes_{\oplus}}\cong\Sym_{\boxtimes_+}\dim_!.
\]
%Fix $\dd$, and let $\ff\in\mathbb{N}^{Q_0}$.  Consider the Cartesian diagram:
%\[
%\xymatrix{
%\Msp(\tilde{Q})_{\ff,\dd}^{\omega\nnilp}\ar[d]^{\pi^{\omega\nnilp}_{\ff,\dd}}\ar@{^(->}[r]^-{\xi_{\ff,\dd}}&\Msp(\tilde{Q})_{\ff,\dd}\ar[d]^{\pi_{\ff,\dd}}\\
%\Msp(\tilde{Q})_{\dd}^{\omega\nnilp}\ar@{^(->}[r]^-{\xi_{\dd}}&\Msp(\tilde{Q})_{\dd}.
%}
%\]
%The base change isomorphism, applied to $\phin{\mathcal{T}r(\tilde{W})_{\ff,\dd}}\mathbb{Q}_{\Msp(\tilde{Q})_{\ff,\dd}}$, gives the isomorphism
%\[
%\beta\colon\pi^{\omega\nnilp}_{\ff,\dd,!}\xi_{\ff,\dd}^*\phin{\mathcal{T}r(\tilde{W})_{\ff,\dd}}\mathbb{Q}_{\Msp(\tilde{Q})_{\ff,\dd}}\cong \xi_{\dd}^*\pi_{\ff,\dd,!}\phin{\mathcal{T}r(\tilde{W})_{\ff,\dd}}\mathbb{Q}_{\Msp(\tilde{Q})_{\ff,\dd}}.
%\]
%Lemma \ref{decomp2} provides an isomorphism 
%\[
%\gamma\colon\xi_{\dd}^*\pi_{\ff,\dd,!}\phin{\mathcal{T}r(\tilde{W})_{\ff,\dd}}\mathbb{Q}_{\Msp(\tilde{Q})_{\ff,\dd}}\cong\xi_{\dd}^*\Ho\left(\pi_{\ff,\dd,!}\phin{\mathcal{T}r(\tilde{W})_{\ff,\dd}}\mathbb{Q}_{\Msp(\tilde{Q})_{\ff,\dd}}\right).
%\]
%The morphism $\HO_c((\gamma\circ\beta)\otimes\LL^{\otimes -\dim(\Mst(\tilde{Q})_{\dd})/2})$, for $\ff\gg 0$, induces the isomorphism
%\begin{align}
%\label{bci}
%\HO_c\left(\Mst(\tilde{Q})_{\dd}^{\omega\nnilp},\phin{\mathfrak{Tr}(\tilde{W})_{\dd}}\ICS_{\Mst(\tilde{Q})_{\dd}}(\QQ)\right)\cong \HO_c\left(\Msp(\tilde{Q})_{\dd},\Ho\left(\p_{\dd,!}\phin{\mathfrak{Tr}(\tilde{W})_{\dd}}\ICS_{\Mst(\tilde{Q})_{\dd}}(\QQ)\right)\mathlarger{\mathlarger{|}}_{\Msp(\tilde{Q})_{\dd}^{\omega\nnilp}}\right).
%\end{align}
We denote by $\iota'_{\dd}\colon \Msp(\tilde{Q})_{\dd}^{\omega\nnilp}\hookrightarrow \Msp(\tilde{Q})_{\dd}$ the inclusion.  Taking the direct sum over all $\dd\in\mathbb{N}^{Q_0}$, applying base change, and using the relative cohomological integrality theorem (Theorem \ref{IntThm}):
\begin{align*}
\Dim_!\!\left((\phin{\mathfrak{Tr}(\tilde{W})}\ICS_{\Mst(\tilde{Q})}(\mathbb{Q}))|_{\Mst(\tilde{Q})^{\omega\nnilp}}\right)\cong&\dim_!\iota'^* \p_{!}\phin{\mathfrak{Tr}(\tilde{W})}\ICS_{\Mst(\tilde{Q})}(\mathbb{Q})\\
\cong&\dim_!\iota'^* \Sym_{\boxtimes_{\oplus}}\!\left(\DTS_{\tilde{Q},\tilde{W}}\otimes\HO_c(\B\mathbb{C}^*,\mathbb{Q})_{\vir}\right)\\
\cong&\dim_! \Sym_{\boxtimes_{\oplus}}\!\left(\DTS_{\tilde{Q},\tilde{W}}|_{\Msp(\tilde{Q})^{\omega\nnilp}}\otimes\HO_c(\B\mathbb{C}^*,\mathbb{Q})_{\vir}\right)\\
\cong&\Sym_{\boxtimes_+}\!\left(\dim_!\DTS_{\tilde{Q},\tilde{W}}|_{\Msp(\tilde{Q})^{\omega\nnilp}}\otimes\HO_c(\B\mathbb{C}^*,\mathbb{Q})_{\vir}\right)\\
\cong&\Sym_{\boxtimes_+}\!\left(\dim_!\DTS_{\tilde{Q},\tilde{W}}|_{\Msp(\tilde{Q})^{\omega\nnilp}}\otimes\HO_c(\B\mathbb{C}^*,\mathbb{Q})_{\vir}\right)
\end{align*}
giving the isomorphism (\ref{nilpEq}).

Taking the direct sum of the isomorphisms (\ref{secondDimRed}) over $\dd\in\mathbb{N}^{Q_0}$ gives the isomorphism (\ref{nonNilpEq}).  Applying $\dim_!$ to (\ref{shriekInt}) we have the isomorphisms
\begin{align*}
\dim_!\p_{!}\phin{\mathfrak{Tr}(\tilde{W})}\ICS_{\Mst(\tilde{Q})}(\mathbb{Q})&\cong\dim_!\Sym_{\boxtimes_{\oplus}}\!\left(\DTS_{\tilde{Q},\tilde{W}}\otimes\HO_c(\B\mathbb{C}^*,\mathbb{Q})_{\vir}\right)\\
&\cong\Sym_{\boxtimes_+}\dim_!\!\left(\DTS_{\tilde{Q},\tilde{W}}\otimes\HO_c(\B\mathbb{C}^*,\mathbb{Q})_{\vir}\right)\\
&\cong\Sym_{\boxtimes_+}\!\left(\DT_{\tilde{Q},\tilde{W}}\otimes\HO_c(\B\mathbb{C}^*,\mathbb{Q})_{\vir}\right).
\end{align*}
To prove the existence of the isomorphism (\ref{non2}), then, it is sufficient to prove that $\DT_{\tilde{Q},\tilde{W}}\cong\DT_{\tilde{Q},\tilde{W}}^{\omega\nnilp}\otimes\mathbb{L}$.  Fix a dimension vector $\dd$.  We let $\mathbb{A}^1$ act on $\Msp(\tilde{Q})_{\dd}$ as follows
\[
z\cdot\rho(a)=\begin{cases} \rho(a)+z\id_{\dd_i\times \dd_i}&\textrm{if }a=\omega_i\textrm{ for some }i\\ \rho(a)&\textrm{otherwise.}\end{cases}
\]
Then $\mathcal{T}r(\tilde{W})_{\dd}$ is invariant with respect to the $\mathbb{A}^1$-action and it follows that the underlying perverse sheaf of $\DTS_{\tilde{Q},\tilde{W},\dd}$ can be obtained from an $\mathbb{A}^1$-equivariant MHM via the forgetful map.  In particular, if we let $\DTS_{\tilde{Q},\tilde{W},\dd}'$ be the restriction of $\DTS_{\tilde{Q},\tilde{W},\dd}$ to the locus $\Msp\subset\Msp(\tilde{Q})_{\dd}$ where the union of the sets of generalized eigenvalues of all of the $\omega_i$ has only one element, and let $m\colon \mathbb{A}^1\times\Msp(\tilde{Q})^{\omega\nnilp}_{\dd}\xrightarrow{\cong} \Msp$ be the restriction of the action map, we have $\DTS_{\tilde{Q},\tilde{W},\dd}'\cong m_*(\mathbb{Q}_{\mathbb{A}^1}\boxtimes\DTS^{\omega\nnilp}_{\tilde{Q},\tilde{W}})$.  By the support lemma (Lemma \ref{lem1}) we have $\DTS_{\tilde{Q},\tilde{W},\dd}=\DTS_{\tilde{Q},\tilde{W},\dd}'$ and so we deduce that
\begin{align*}
\DT_{\tilde{Q},\tilde{W},\dd}\cong&\DT_{\tilde{Q},\tilde{W},\dd}^{\omega\nnilp}\otimes (\mathbb{A}^1\rightarrow\pt)_!\mathbb{Q}_{\mathbb{A}^1}\\
\cong&\DT_{\tilde{Q},\tilde{W},\dd}^{\omega\nnilp}\otimes\mathbb{L}
\end{align*}
as required.
\end{proof}
We complete the proof of Theorem \ref{purityThm} by proving the support lemma.  
\renewcommand*{\proofname}{Proof of Lemma \ref{lem1}}
\begin{proof}
%As in the proof of Lemma \ref{lem2}, we denote by $\DTS'_{\tilde{W}}$ the restriction of $\DTS_{\tilde{W}}$ to the locus where the set of endomorphisms $\{\rho(\omega_i)|i\in Q_0\}$ has only one generalised eigenvalue.  Then it is sufficient to show that the natural restriction map $\DTS_{\tilde{Q},\tilde{W}}\rightarrow \DTS'_{\tilde{Q},\tilde{W}}$ is an isomorphism.  This is equivalent to the statement that
%\[
%\Sym_{\boxtimes_{\oplus}}\DTS_{\tilde{Q},\tilde{W}}\rightarrow\Sym_{\boxtimes_{\oplus}}\DTS'_{\tilde{Q},\tilde{W}}
%\]
%is an isomorphism.  
To ease the notation we prove the lemma under the assumption that $\zeta$ is the degenerate stability condition: the proof for the general case is unchanged.
\smallbreak
Since the support of $\DTS_{\tilde{Q},\tilde{W}}$ is the same as the support of the underlying perverse sheaf, and all complexes that we encounter in the following proof are quasi-isomorphic to their total cohomology, throughout the proof we work in the category of cohomologically graded perverse sheaves.  
\smallbreak
Let $x\in\Msp(\tilde{Q})_{\dd}$ be a point corresponding to a semisimple $\mathbb{C}\tilde{Q}$-module $\rho$, and assume that there are at least two distinct eigenvalues $\epsilon_1,\epsilon_2$ for the set of operators $\{\rho(\omega_i)|i\in Q_0\}$.  Assume, for a contradiction, that $x\in\supp(\DTS_{\tilde{Q},\tilde{W}})$, so that in particular
\[
x\in\supp \!\left(\p_*\phip{\mathfrak{Tr}(\tilde{W})}\ICS_{\Mst(\tilde{Q})}(\mathbb{Q})\right)
\]
and so by (\ref{jacCrit}) and Remark \ref{critRem}, there exists a $\Jac(\tilde{Q},\tilde{W})$ module with semisimplification given by $\rho$, and so $\rho$ is a semisimple $\Jac(\tilde{Q},\tilde{W})$-module.  

\smallbreak
Under our assumptions, there are disjoint (analytic) open sets $U_1,U_2\subset\mathbb{C}$ with $\epsilon_1\in U_1$ and $\epsilon_2\in U_2$, and with all of the generalised eigenvalues of $\rho$ contained in $U_1\cup U_2$.  Given an (analytic) open set $U\subset \mathbb{C}$, we denote by $\Msp^{U}(\tilde{Q})_{\dd}\subset \Msp(\tilde{Q})_{\dd}$ the subspace consisting of those $\rho$ such that all of the generalised eigenvalues of $\{\rho(\omega_i)|i\in Q_0\}$ belong to $U$, and we define $\Mst^{U}(\tilde{Q})$ similarly.  
\smallbreak
Given a point $x\in \Mst^{U_1\cup U_2}(\Jac(\tilde{Q},\tilde{W}))$, the associated $\Jac(\tilde{Q},\tilde{W})$-module $M$ admits a canonical direct sum decomposition 
\[
M=M_{1}\oplus M_2
\]
where all of the eigenvalues of all of the $\omega_i$, restricted to $M_i$, belong to $U_i$\footnote{Note that this is not true of a general point in $\Mst^{U_1\cup U_2}(\tilde{Q})$ --- the crucial fact is that the operation $\sum_{i\in Q_0}\rho(\omega_i)\cdot$ defines a module homomorphism for a $\Jac(\tilde{Q},\tilde{W})$-module $\rho$, since $\sum_{i\in Q_0}\omega_i$ is central in $\Jac(\tilde{Q},\tilde{W})$.}.  In particular, there is an isomorphism of complex analytic stacks
\begin{equation}
\label{sopp}
\Mst^{U}(\Jac(\tilde{Q},\tilde{W}))\cong \Mst^{U_1}(\Jac(\tilde{Q},\tilde{W}))\times\Mst^{U_2}(\Jac(\tilde{Q},\tilde{W})).
\end{equation}
We claim that there is an isomorphism
\begin{align}
\label{tle}
\p_*\phip{\mathfrak{Tr}(\tilde{W})}|_{\Mst^{U_1\cup U_2}(\tilde{Q})}\cong &\p_*\phip{\mathfrak{Tr}(\tilde{W})}\ICS_{\Mst(\tilde{Q})}|_{\Mst^{U_1}(\tilde{Q})}\boxtimes_{\oplus}\p_*\phip{\mathfrak{Tr}(\tilde{W})}\ICS_{\Mst(\tilde{Q})}|_{\Mst^{U_2}(\tilde{Q})}.
\end{align}
This follows from Lemma \ref{techLemma} below.  Assuming the claim, applying \eqref{starInt} to the RHS of \eqref{tle} we have isomorphisms
\begin{align}
\nonumber
\p_*\phip{\mathfrak{Tr}(\tilde{W})}|_{\Mst^{U_1\cup U_2}(\tilde{Q})}\cong& \Sym_{\boxtimes_{\oplus}}\!\left(\left(\DTS_{\tilde{Q},\tilde{W}}\otimes\HO(\B\mathbb{C}^*,\mathbb{Q})_{\vir}\right)\mathlarger{|}_{\Msp^{U_1}(\tilde{Q})}\right)\boxtimes_{\oplus}\\
\nonumber
&\Sym_{\boxtimes_{\oplus}}\!\left(\left(\DTS_{\tilde{Q},\tilde{W}}\otimes\HO(\B\mathbb{C}^*,\mathbb{Q})_{\vir}\right)\mathlarger{|}_{\Msp^{U_2}(\tilde{Q})}\right)\\
\label{hupa}
\cong & \Sym_{\boxtimes_{\oplus}}\!\left(\left(\DTS_{\tilde{Q},\tilde{W}}|_{\Msp^{U_1}(\tilde{Q})}\oplus\DTS_{\tilde{Q},\tilde{W}}|_{\Msp^{U_2}(\tilde{Q})}\right)\otimes\HO(\B\mathbb{C}^*,\mathbb{Q})_{\vir}\right).
\end{align}
On the other hand, restricting the isomorphism of (\ref{starInt}) to the LHS of \eqref{tle} yields
\begin{align}
\label{hupb}
\p_*\phip{\mathfrak{Tr}(\tilde{W})}|_{\Mst^{U_1\cup U_2}(\tilde{Q})}\cong&\Sym_{\boxtimes_{\oplus}}\!\left(\DTS_{\tilde{Q},\tilde{W}}|_{\Msp^{U_1\cup U_2}(\tilde{Q})}\otimes \HO(\B\mathbb{C}^*,\mathbb{Q})_{\vir}\right).
\end{align}
Comparing (\ref{hupa}) and (\ref{hupb}), we deduce that
\[
\DTS_{\tilde{Q},\tilde{W}}|_{\Msp^{U_1\cup U_2}(\tilde{Q})}\cong \DTS_{\tilde{Q},\tilde{W}}|_{\Msp^{U_1}(\tilde{Q})}\oplus\DTS_{\tilde{Q},\tilde{W}}|_{\Msp^{U_2}(\tilde{Q})}.
\]
We deduce that
\begin{align*}
\supp\!\left(\DTS_{\tilde{Q},\tilde{W}}|_{\Msp^{U_1\cup U_2}(\tilde{Q})}\right)=&\supp\!\left(\DTS_{\tilde{Q},\tilde{W}}|_{\Msp^{U_1}(\tilde{Q})}\oplus\DTS_{\tilde{Q},\tilde{W}}|_{\Msp^{U_2}(\tilde{Q})}\right)\\
= &\supp\!\left(\DTS_{\tilde{Q},\tilde{W}}|_{\Msp^{U_1}(\tilde{Q})}\right)\cup \supp\!\left(\DTS_{\tilde{Q},\tilde{W}}|_{\Msp^{U_2}(\tilde{Q})}\right)\\
\subset&\Msp^{U_1}(\tilde{Q})\cup \Msp^{U_2}(\tilde{Q}),
\end{align*}
and so since 
\[
x\in \Msp^{U_1\cup U_2}(\tilde{Q})\setminus(\Msp^{U_1}(\tilde{Q})\cup \Msp^{U_2}(\tilde{Q})), 
\]
the restriction of $\DTS_{\tilde{Q},\tilde{W}}$ to $x$ is zero, which is the required contradiction.  

\smallbreak
For the final statement of the lemma, it suffices to prove that if $\rho$ is a simple $\Jac(\tilde{Q},\tilde{W})$-module, then $\sum_{i\in Q_0}\rho(\omega_i)$ acts via scalar multiplication.  In the decomposition of $\rho$ into generalised eigenspaces for the action of the operator $\sum_{i\in Q_0}\rho(\omega_i)\cdot$ we have already shown that there is only one generalised eigenvalue, which we denote $\lambda$.  Then $\rho$ is filtered by the nilpotence degree of the nilpotent operator 
\[
\Psi:=\sum_{i\in Q_0}\rho(\omega_i)\cdot -\lambda\Id_{\rho},
\]
and so since $\rho$ is simple, $\Psi=0$ and we are done.
\end{proof}
\subsection{Proof of Theorem \ref{2dbpsthm}}
\label{thbproof}
We have now introduced all of the ingredients required to prove Theorem \ref{2dbpsthm}.  Firstly, by Lemma \ref{lem1}, the support of $\DTS_{\tilde{Q},\tilde{W},\dd}^{\zeta}$ lies in the image of the morphism
\[
\AA^1\times\Msp(\overline{Q})^{\zeta\sst}\rightarrow \Msp(\tilde{Q})^{\zeta\sst}
\]
defined as in the statement of Theorem \ref{2dbpsthm}.  The support of $\DTS_{\tilde{Q},\tilde{W},\dd}^{\zeta}$ also lies within the locus of polystable $\Jac(\tilde{Q},\tilde{W})$-modules, so by Proposition \ref{basicEq}, the support of $\DTS_{\tilde{Q},\tilde{W},\dd}^{\zeta}$ lies within the image of 
\[
m\colon \AA^1\times\Msp(\Pi_Q)_{\dd}^{\zeta\sst}\hookrightarrow \Msp(\tilde{Q})_{\dd}^{\zeta\sst}.
\]
We have seen in the proof of Lemma \ref{lem2} that $\DTS_{\tilde{Q},\tilde{W},\dd}^{\zeta}$ is $\AA^1$-equivariant, where the $\AA^1$-action on the subspace $\AA^1\times\Msp(\Pi_Q)_{\dd}^{\zeta\sst}$ is via translation in the first factor.  It follows that we can write
\begin{equation}
\label{ttys}
\DTS_{\tilde{Q},\tilde{W},\dd}^{\zeta}\cong \ICS_{\AA^1}(\QQ)\boxtimes \DTS_{\Pi_Q,\dd}^{\zeta}.
\end{equation}
Finally, $\DTS_{\tilde{Q},\tilde{W},\dd}^{\zeta}=\phim{\mathfrak{T}r(W)}\ICS_{\Msp(\tilde{Q})^{\zeta\sst}_{\dd}}$ is Verdier self-dual \cite{DaMe15b}, as is $\ICS_{\AA^1}(\QQ)$.  So from \eqref{ttys} we deduce
\[
\ICS_{\AA^1}(\QQ)\boxtimes \DTS_{\Pi_Q,\dd}^{\zeta}\cong \ICS_{\AA^1}(\QQ)\boxtimes \mathbb{D}\DTS_{\Pi_Q,\dd}^{\zeta}
\]
and Verdier self-duality of $\DTS_{\Pi_Q,\dd}^{\zeta}$ follows.
\subsubsection{}
Although we will not use it, we remark that via the same proof as Lemma \ref{lem1} we deduce the following
\begin{lemma}
Let $(Q,W)$ be a QP, with $Q=Q'\coprod Q''$ a disjoint union of quivers, and let $\zeta$ be a generic stability condition.  Denote by 
\begin{align*}
\pi'\colon &\mathbb{N}^{Q_0}\rightarrow \mathbb{N}^{Q'_0}\\
\pi''\colon &\mathbb{N}^{Q_0}\rightarrow \mathbb{N}^{Q''_0}
\end{align*}
the projections.  If $\pi'(\dd)\neq 0$ and $\pi''(\dd)\neq 0$ then 
\[
\DTS^{\zeta}_{Q,W,\dd}=0.
\]
\end{lemma}

\subsubsection{}
We finish this section with the technical lemma appearing in the proof of Lemma \ref{lem1}.  Fix a decomposition $\dd=\dd'+\dd''$.  Then via \eqref{sopp} there is an open and closed inclusion
\[
i\colon \Mst^{U_1}(\Jac(\tilde{Q},\tilde{W}))_{\dd'}\times\Mst^{U_2}(\Jac(\tilde{Q},\tilde{W}))_{\dd''}\rightarrow \Mst^{U_1\cup U_2}(\Jac(\tilde{Q},\tilde{W}))_{\dd}.
\]
\begin{lemma}
\label{techLemma}
Let $U_1,U_2$ be disjoint analytic open subspaces of $\AA^1$.  There is a natural isomorphism of perverse sheaves
\[
i^*\phip{\mathfrak{Tr}(\tilde{W})}\ICS_{\Mst^{U_1\cup U_2}(\tilde{Q})^{\zeta\sst}_{\dd}}(\mathbb{Q})\cong \phip{\mathfrak{Tr}(\tilde{W})}\ICS_{\Mst^{U_1}(\tilde{Q})^{\zeta\sst}_{\dd'}}(\mathbb{Q})\boxtimes \phip{\mathfrak{Tr}(\tilde{W})}\ICS_{\Mst^{U_2}(\tilde{Q})^{\zeta\sst}_{\dd''}}(\mathbb{Q}).
\]
\end{lemma}
This isomorphism does \textit{not} follow directly from the Thom--Sebastiani isomorphism, since we need to compare the vanishing cycle sheaf of the function $\mathfrak{Tr}(\tilde{W})\boxplus\mathfrak{Tr}(\tilde{W})$ on $\Mst^{U_1}(\tilde{Q})_{\dd'}\times \Mst^{U_2}(\tilde{Q})_{\dd''}$ with the vanishing cycle sheaf for the function $\mathfrak{Tr}(\tilde{W})$ on $\Mst^{U_1\cup U_2}(\tilde{Q})_{\dd}$, and these ambient smooth stacks are different.
\renewcommand*{\proofname}{Proof of Lemma \ref{techLemma}}
\begin{proof}
Again, it is sufficient to prove the lemma for the degenerate stability condition (the general case then follows by restriction to the $\zeta$-semistable locus).  Writing
\begin{align*}
Y=&X^{U_1\cup U_2}(L)_{\dd}\times X(\overline{Q})_{\dd}\\
B=&X^{U_1}(\tilde{Q})_{\dd'}\times X^{U_2}(\tilde{Q})_{\dd''}
\end{align*}
we have 
\begin{align*}
\Mst^{U_1\cup U_2}(\tilde{Q})_{\dd}\cong&Y/\Gl_{\dd'\times \dd''}\\
\Mst^{U_1}(\tilde{Q})_{\dd'}\times \Mst^{U_1}(\tilde{Q})_{\dd''}\cong&B/\Gl_{\dd'\times\dd''}.
\end{align*}
The space $Y$ is the total space of the $\Gl_{\dd'\times\dd''}$-equivariant vector bundle $V^{+}\oplus V^{-}$ on $B$, where 
\begin{align*}
V^+=&\prod_{a\in \overline{Q}_1}\Hom(\mathbb{C}^{\dd'_{s(a)}},\mathbb{C}^{\dd''_{t(a)}})\\
V^-=&\prod_{a\in \overline{Q}_1}\Hom(\mathbb{C}^{\dd'_{t(a)}},\mathbb{C}^{\dd''_{t(a)}}).
\end{align*}
Note that $\rank(V^+)=\rank(V^-)$.  Denote by $z\colon B\rightarrow Y$ the inclusion of the zero section.  Writing $f,g$ for the functions on $Y,B$ induced by $\Tr(\tilde{W})$, it is sufficient to show that there is an isomorphism of $\Gl_{\dd'\times\dd''}$-equivariant perverse sheaves
\[
\phip{f}\ICS_{Y}(\mathbb{Q})\cong z_*\phip{g}\ICS_{B}(\mathbb{Q}).
\]
\smallbreak
Let $\mathbb{C}^*$ act on $V^{+}$ and $V^{-}$ with weights $1$ and $-1$ respectively, then $f$ is $\mathbb{C}^*$-invariant.  It follows that $g\lvert_{\Tot(V^+)}=f\circ\pi^+$, where $\pi^+\colon\Tot(V^+)\rightarrow B$ is the projection.  So there is a natural isomorphism 
\begin{equation}
\label{hl1}
\phip{g}\mathbb{Q}_{B}\cong \pi^+_!\phip{f}\mathbb{Q}_{\Tot(V^+)}[2\rank(V^+)].
\end{equation}
We claim that the natural morphism 
\begin{equation}
\label{hl2}
\pi_!\phip{f}\left(\mathbb{Q}_{Y}\rightarrow i^+_*\mathbb{Q}_{\Tot(V^+)}\right)
\end{equation}
is an isomorphism, where we denote by $i^+\colon \Tot(V^+)\rightarrow Y$ the inclusion.  This can be checked locally on the base $B$.  Pick $b\in B$, and let $x_1,\ldots,x_{\alpha},y_1,\ldots,y_{\beta},z_1,\ldots,z_{\beta}$ be a set of elements of the local ring $\mathcal{O}_{X(\tilde{Q}),b}$, providing a basis for $\mathfrak{m}_b/\mathfrak{m}_b^2$, where $x_i$ all have weight zero, $y_i$ have weight $1$, and $z_i$ have weight $-1$ for the $\mathbb{C}^*$-action.  The weight $-1$ partial derivatives of $f$ are provided by $\partial g/\partial y_i$, and so since the critical locus of $g$ (restricted to a neighbourhood of $b$) lies on the zero section $B$, it follows that we can change coordinates and pick $z_i=\partial g/\partial y_i$.  Then we have $g=h+k$ in $\mathcal{O}_{X(\tilde{Q}),b}$, with $h\in\mathbb{C}[x_1,\ldots,x_{\alpha}]$ and 
\[
k=\sum_{1\leq i\leq \beta}y_iz_i.
\]
By the Thom--Sebastiani theorem, after restricting to a neighbourhood $U\ni b$
\[
\phip{g}\mathbb{Q}_{U\times V}\cong \phip{h}\mathbb{Q}_U\boxtimes \phip{k}\mathbb{Q}_V
\]
and the claim reduces to the claim that
\[
\pi_!\phip{k}(\mathbb{Q}_V\rightarrow (V^+\hookrightarrow V)_*\mathbb{Q}_{V^+})
\]
is an isomorphism, which is a simple calculation, or a trivial application of the dimensional reduction theorem.

Combining \eqref{hl1} and \eqref{hl2} yields the isomorphism $\phip{g}\mathbb{Q}_B\cong\pi_!\phip{f}\mathbb{Q}_Y[\codim_Y(B)]$.  Since $\phip{f}\mathbb{Q}_Y$ is supported on $B$, we obtain the required isomorphism by applying $z_*$ to this isomorphism and shifting cohomological degree by $\dim B$.
\end{proof}

\subsection{Calculating $\HO_c(\Mst(\Pi_Q)_{\dd},\mathbb{Q})$}\label{hssection}
We use Theorem \ref{purityThm} and existing results on the $E$ series of $\Mst(\Pi_Q)_{\dd}$ to determine the compactly supported cohomology of $\Mst(\Pi_Q)_{\dd}$, along with its mixed Hodge structure.  The $E$ series (see Subsection \ref{TPT}) of $\HO_c(\Mst(\Pi_Q)_{\dd},\mathbb{Q})$ was calculated in \cite{Moz11}.  Before we recall this series, we recall some definitions.  
\smallbreak
Recall the plethystic exponential defined in \S\ref{KacIntro}.  For our purposes, it is profitable to think of the plethystic exponential as the decategorification of the endofunctor of tensor categories taking an object to the underlying object of the free symmetric algebra generated by that object.  We define the ring $\mathbb{Z}\dbl X_1,\ldots, X_m\dbr[\![Y_1,\ldots,Y_n]\!]$ of formal Laurent power series $g(X_1,\ldots,X_m,Y_1,\ldots,Y_n)$ such that for each $(a_1,\ldots, a_n)\in\mathbb{N}^n$ the $Y_1^{a_1}\cdots Y_n^{a_n}$ coefficient of 
\[
g(X_1,\ldots,X_m,Y_1,\ldots,Y_n)X_1^{c_1}\cdots X_m^{c_m}
\]
is in $\mathbb{Z}[\![X_1,\ldots,X_m]\!]$ for sufficiently large $c_1,\ldots,c_m$.  This is isomorphic to the Grothendieck ring of the category $\Dub^{\diamond}(\Vect_{\mathbb{Z}^m\oplus\mathbb{Z}^n})$, which we define to be the subcategory of the unbounded derived category of $\mathbb{Z}^m\oplus\mathbb{Z}^n$-graded vector spaces $V$ such that
\begin{enumerate}
\item
For each $(\ee,\dd)\in\mathbb{Z}^m\oplus\mathbb{Z}^n$ the total cohomology $\HO(V)_{\ee,\dd}$ is finite-dimensional
\item
$\HO(V_{\ee,\dd})\neq 0$ only if $\dd\in\mathbb{N}^{n}$
\item
For each $\dd\in\mathbb{N}^{n}$ there exists $\ee\in\mathbb{Z}^m$ such that $\HO(V)_{\ee',\dd}=0$ if $\ee'_i\leq \ee_i$ for some $i=1,\ldots,m$.
\end{enumerate}
This isomorphism is induced by the character function
\[
\chi\colon [V]\mapsto\sum_{i\in\mathbb{Z}}\sum_{(\ee,\dd)\in\mathbb{Z}^m\oplus\mathbb{Z}^n}(-1)^i\dim\left(\HO^i(V)_{\ee,\dd}\right)X^{\ee}Y^{\dd}.
\]
We define $\Dub^{\diamond}(\Vect^+_{\mathbb{Z}^m\oplus\mathbb{Z}^n})\subset \Dub^{\diamond}(\Vect_{\mathbb{Z}^m\oplus\mathbb{Z}^n})$ to be the full subcategory satisfying the extra condition that the total cohomology $\HO(V)_{(\ee,0)}$ is zero for all $\ee\in\mathbb{Z}^m$.  Then $\chi$ induces an isomorphism
\[
\chi\colon \Kk\!\left(\Dub^{\diamond}(\Vect^+_{\mathbb{Z}^m\oplus\mathbb{Z}^n})\right)\rightarrow \mathfrak{m}\mathbb{Z}\dbl X_1,\ldots, X_m\dbr[\![Y_1,\ldots,Y_n]\!]
\]
where $\mathfrak{m}$ is the maximal ideal generated by $Y_1,\ldots,Y_n$.  We may define plethystic exponentiation via the formula
\begin{equation}
\label{aped}
\Exp(\chi([V]))=\chi[\Sym(V)]
\end{equation}
for $V\in \Dub^{\diamond}(\Vect^+_{\mathbb{Z}^m\oplus\mathbb{Z}^n})$.  Then the E series for $\HO_c(\Mst(\Pi_Q)_{\dd},\mathbb{Q})$ is given by \cite{Moz11}
\begin{equation}
\label{Kacy}
\sum_{\dd\in\mathbb{N}^{Q_0}}\EE\!\left(\HO_c\!\left(\Mst(\Pi_Q)_{\dd},\mathbb{Q}\right),x,y\right)(xy)^{(\dd,\dd)}t^{\dd}=\Exp\left(\sum_{0\neq \dd\in\mathbb{N}^{Q_0}}\kac_{Q,\dd}(xy)(1-x^{-1}y^{-1})^{-1}t^{\dd}\right).
\end{equation}
Here $x^{-1}$ and $y^{-1}$ are the invertible commuting variables, and $\{t_i\}_{i\in Q_0}$ are the other commuting variables.  Each of the $(xy)$ terms arises from the E polynomial $\EE(\LL,x,y)=xy$.  Given a polynomial $b(q)=\sum_{i\geq 0}b_iq^i\in\mathbb{N}[q]$ and an object $\mathcal{F}$ in a tensor category $\mathscr{C}$ we define
\[
b(\mathcal{F})=\bigoplus_{i\in\mathbb{N}}(\mathcal{F}^{\otimes i})^{\oplus b_i}.
\]
By Theorem \ref{purityThm} the mixed Hodge structure on $\HO_c(\Mst(\Pi_Q)_{\dd},\mathbb{Q})$ is entirely determined by its E series, and we deduce from \eqref{Kacy}:
\begin{theorem}
There is an isomorphism of cohomologically graded, $\mathbb{N}^{Q_0}$-graded mixed Hodge structures
\begin{equation}
\label{KacHH}
\bigoplus_{\dd\in\mathbb{N}_{Q_0}}\HO_c(\Mst(\Pi_Q)_{\dd},\mathbb{Q})\otimes \LL^{(\dd,\dd)_Q}\cong \Sym\!\left(\bigoplus_{\dd\in\mathbb{N}^{Q_0}\setminus \{0\}}\kac_{Q,\dd}(\LL)\otimes \HO(\B\mathbb{C}^*,\mathbb{Q})^{\vee}\right).
\end{equation}
\end{theorem}
Taking Hodge series of \eqref{KacHH} yields the following refinement of (\ref{Kacy}):
\begin{align*}
&\sum_{\dd\in\mathbb{N}^{Q_0}}\hp\!\left(\HO_c\!\left(\Mst(\Pi_Q)_{\dd},\mathbb{Q}\right),x,y,z\right)(xyz^2)^{(\dd,\dd)_Q}t^{\dd}\\&=\Exp\left(\sum_{0\neq \dd\in\mathbb{N}^{Q_0}}\kac_{Q,\dd}(xyz^2)(1-x^{-1}y^{-1}z^{-2})^{-1}t^{\dd}\right)&\in \mathbb{Z}\dbl x^{-1}  y^{-1} z^{-1}\dbr[\![ t_i\lvert i\in Q_0]\!].
\end{align*}

\renewcommand*{\proofname}{Proof}

\section{Degree zero DT theory}
\label{Jordan}
\subsection{Degree zero BPS sheaves}
For $n\in\mathbb{N}$ we define $Q^{(n)}$ to be a quiver with one vertex, which we denote $0$, and $n$ loops.  We will be particularly interested in the quiver $Q_{\Jor}:=Q^{(1)}$: the \textit{Jordan quiver}.  We identify
\[
Q^{(3)}=\widetilde{Q_{\Jor}}.
\]
We denote by $x,y,z$ the three arrows of $Q^{(3)}$.  Then $\tilde{W}=x[y,z]$.  The ideas in the proof of Theorem \ref{purityThm} allow us to prove rather more for the QP $(\widetilde{Q_{\Jor}},\tilde{W})$, essentially because this QP is invariant (up to sign) under permutation of the loops, so that we can apply the support lemma (Lemma \ref{lem1}) three times. 
\smallbreak
Let $d\in\mathbb{N}$ with $d\geq 1$.  The support of $\p_{\widetilde{Q_{\Jor}},!}\phin{\mathfrak{Tr}(\tilde{W})_d}\ICS_{\Mst(\widetilde{Q_{\Jor}})_{d}}(\mathbb{Q})$ is given by the coarse moduli space of $d$-dimensional representations of the Jacobi algebra $\mathbb{C}[x,y,z]$, i.e. the space of semisimple representations of $\mathbb{C}[x,y,z]$.  This space is in turn isomorphic to $\Sym^d(\AA^3)$, since any simple representation $\rho$ of $\mathbb{C}[x,y,z]$ is one-dimensional, and determined up to isomorphism by the three complex numbers $\rho(x),\rho(y),\rho(z)$.
\begin{theorem}
\label{dzBPS}
There is an isomorphism in $\MHM(\Msp(\widetilde{Q_{\Jor}})_d)$
\[
\DTS_{\widetilde{Q_{\Jor}},\tilde{W},d}=\Delta_{\mathbb{A}^3,d,*}\ICS_{\mathbb{A}^3}(\QQ)
\]
for all $d$.  
\end{theorem}
\begin{proof}
\smallbreak
By the same argument as for Lemma \ref{lem1}, the support of $\DTS_{\widetilde{Q_{\Jor}},\tilde{W},d}$ is contained in the image of the morphism
\begin{align*}
\Delta_{\AA^3,d}\colon&\AA^3\rightarrow \Msp(\widetilde{Q_{\Jor}})_d\\
&(z_1,z_2,z_3)\mapsto\left(z_1\cdot\Id_{d\times d},z_2\cdot\Id_{d\times d},z_3\cdot\Id_{d\times d}\right).  
\end{align*}
By the argument in the proof of Lemma \ref{lem2}, $\DTS_{\widetilde{Q_{\Jor}},\tilde{W},d}$ is constant on its support, so
\[
\DTS_{\widetilde{Q_{\Jor}},\tilde{W},d}\cong \Delta_{\AA^3,d,*}\ICS_{\AA^3}(\mathbb{Q})\otimes \mathcal{L}_d
\]
for some monodromic mixed Hodge structure $\mathcal{L}_d$.  It follows that 
\begin{equation}
\label{beforea}
\DT^{\vee}_{\widetilde{Q_{\Jor}},\tilde{W},d}\cong \mathcal{L}_d^{\vee}\otimes\mathbb{L}^{3/2}.
\end{equation}
On the other hand, by \cite[Prop.1.1]{BBS} we have, after passing to classes in the Grothendieck ring of mixed Hodge structures
\begin{equation}
\label{beforer}
[\DT^{\vee}_{\widetilde{Q_{\Jor}},\tilde{W},d}]=[\mathbb{L}^{ 3/2}].
\end{equation}
The monodromic mixed Hodge structure $\DT_{\widetilde{Q_{\Jor}},\tilde{W},d}$ is pure by Corollary \ref{BPSpure}.  From (\ref{beforer}) we deduce that 
\[
\DT^{\vee}_{\widetilde{Q_{\Jor}},\tilde{W},d}\cong \LL^{ 3/2}
\]
and so from (\ref{beforea}) there is an isomorphism
\[
\mathcal{L}_d\otimes\mathbb{L}^{3/2}\cong\mathbb{L}^{3/2},
\]
and we finally deduce that $\mathcal{L}_d\cong\mathbb{Q}$, with the standard pure weight zero mixed Hodge structure, as required.
\end{proof}

For any constructible inclusion $\varepsilon\colon U\hookrightarrow \mathbb{C}^3$ there is an inclusion of triples of diagonal matrices with entries in $U$
\[
\iota_{U,d}\colon \Sym^{d}(U)\hookrightarrow \Msp(\widetilde{Q_{\Jor}})_d
\]
as well as an inclusion
\[
\Delta_{U,d}\colon U\hookrightarrow \Sym^{d}(U)\hookrightarrow \Msp(\widetilde{Q_{\Jor}})_d
\]
of the small diagonal.  Taking disjoint unions of all these inclusions we define the inclusions
\[
\iota_{U}\colon \Sym(U)\hookrightarrow \Msp(\widetilde{Q_{\Jor}}).
\]
and
\[
\Delta_{U}\colon \coprod_{d\geq 1} U\hookrightarrow \Msp(\widetilde{Q_{\Jor}}).
\]

We denote by $\Mst(\mathbb{C}[x,y,z])_d^U$ the preimage of $\iota_{U,d}(\Sym^d(U))$ under the map 
\[
\p_{\widetilde{Q_{\Jor}},d}\colon \Mst(\widetilde{Q_{\Jor}})_d\rightarrow \Msp(\widetilde{Q_{\Jor}})_d.  
\]
We set $\mathfrak{C}oh^U(\AA^3)=\left(\coprod_{d\geq 1}\Mst(\mathbb{C}[x,y,z])_d^U\right)\coprod\Mst(\widetilde{Q_{\Jor}})_0$.  Then define the $\mathbb{N}$-graded, cohomologically graded mixed Hodge structure
\[
\Coha_{\mathfrak{C}oh^U(\AA^3)}\coloneqq \HO\!\left(\mathfrak{C}oh^U(\AA^3),\phin{\mathfrak{Tr}(\tilde{W})}\ICS_{\Mst(\widetilde{Q_{\Jor}})}\right).
\]
Combining Theorems \ref{IntThm}, \ref{PBW3d} and \ref{dzBPS} gives the following
\begin{corollary}
\label{stackyUthm}
There is an isomorphism in $\Dl\!\left(\MMHM(\Msp(\widetilde{Q_{\Jor}}))\right)$
\begin{align}
\label{UIso}
\left(\p_{\widetilde{Q_{\Jor}},!}\phim{\mathfrak{Tr}(\tilde{W})}\ICS_{\Mst(\widetilde{Q_{\Jor}})}(\mathbb{Q})\right)\mathlarger{\mathlarger{|}}_{\Sym(U)}\cong \Sym_{\boxtimes_{\oplus}}\!\left(\Delta_{U,*}\QQ_{\coprod_{d\geq 1}U}\otimes\LL^{ -1}\otimes\HO_c(\B\mathbb{C}^*,\mathbb{Q})\right),
\end{align}
and a PBW isomorphism of $\mathbb{N}$-graded mixed Hodge structures
\begin{align}\label{UAIso}
\Sym\!\left(\bigoplus_{d\in\mathbb{Z}_{\geq 1}}\HO^{\BM}(U,\mathbb{Q})\otimes \HO(\B\mathbb{C}^*,\mathbb{Q})\otimes\mathbb{L}^{2}\right)\xrightarrow{\cong}\Coha_{\mathfrak{C}oh^U(\AA^3)} .
\end{align}
\end{corollary}
\begin{proof}
We just construct isomorphism (\ref{UIso}) as a special case of \eqref{shriekInt}; via the same argument we then realise \eqref{UAIso} as a special case of Theorem \ref{PBW3d}.  In fact it is sufficient to construct the isomorphism in the case $U=\mathbb{C}^3$, since then the general case is given by restriction to $\iota_{U}(\Sym(U))$.  In this case, since $\supp (\p_{\widetilde{Q_{\Jor}},!}\phim{\mathfrak{Tr}(\tilde{W})}\ICS_{\Mst(\widetilde{Q_{\Jor}})}(\mathbb{Q}))=\Sym(\AA^3)$, the proposed isomorphism becomes
\[
\p_{\widetilde{Q_{\Jor}},!}\phim{\mathfrak{Tr}(\tilde{W})}\ICS_{\Mst(\widetilde{Q_{\Jor}})}(\mathbb{Q})\cong \Sym_{\boxtimes_{\oplus}}\!\left(\Delta_{\mathbb{A}^3,*}\ICS_{\coprod_{d\geq 1}\mathbb{A}^3}(\QQ)\otimes\HO_c(\B\mathbb{C}^*,\mathbb{Q})_{\vir}\right),
\]
which follows from (\ref{shriekInt}) and Theorem \ref{dzBPS}.
\end{proof}
\subsection{Applications to surfaces and character stacks}
\label{appsss}
Let $j\colon V\hookrightarrow \mathbb{A}^2$ be the inclusion of a constructible subset, and write
\[
U=V\times \mathbb{A}^1\subset \mathbb{A}^3.
\]
We consider the commutative diagram
\[
\xymatrix{
\Mst(\widetilde{Q_{\Jor}})\ar[r]^{\pi}\ar[d]^{\p_{\widetilde{Q_{\Jor}}}}&\Mst(\overline{Q_{\Jor}})\ar[d]^{\p_{\overline{Q_{\Jor}}}}\\
\Msp(\widetilde{Q_{\Jor}})\ar[r]^{\varpi}&\Msp(\overline{Q_{\Jor}})
}
\]
where the horizontal morphisms are the forgetful morphisms.  We denote by $\mathfrak{C}_d=\mathcal{C}_d/\Gl_d(\mathbb{C})$ the stack of commuting pairs of matrices, and set $\mathfrak{C}=\coprod_{d\in\mathbb{N}}\mathfrak{C}_d$.  We define the inclusions
\begin{align*}
\iota_{V}\colon &\Sym(V)\hookrightarrow \Msp(\overline{Q_{\Jor}})\\
\Delta_{V}\colon &\coprod_{d\geq 1} V\hookrightarrow \Sym(V)
\end{align*}
as in the previous section.  We denote by $\cup\colon \Sym(V)\times\Sym(V)\rightarrow \Sym(V)$ the morphism taking a pair of multisets of points to their union (so that $\iota_V$ is a morphism of monoids in the category of schemes).  We define
\[
\mathcal{F}\boxtimes_{\cup}\mathcal{G}\coloneqq \cup_*\!\left(\mathcal{F}\boxtimes\mathcal{G}\right).
\]
We denote by $i\colon \mathfrak{C}\hookrightarrow \Mst(\overline{Q_{\Jor}})$ the inclusion.  By Theorem \ref{dimRedThm}, there is an isomorphism of complexes of mixed Hodge modules
\begin{equation}
\label{dimr0}
\pi_!\phin{\mathfrak{Tr}(\tilde{W})}\ICS_{\Mst(\widetilde{Q_{\Jor}})}(\mathbb{Q})\cong i_*\underline{\mathbb{Q}}_{\mathfrak{C}}.
\end{equation}
We denote by $\mathfrak{C}oh^V(\AA^2)$ the reduced substack of coherent sheaves on $\mathbb{A}^2$ set-theoretically supported on $V$ with zero-dimensional support, and by $p\colon \mathfrak{C}oh^V(\AA^2)\rightarrow \Sym(V)$ the morphism taking such a sheaf to its support, counted with multiplicity, so that $p$ restricts to a morphism
\[
p_d\colon \mathfrak{C}oh^V_d(\AA^2)\rightarrow \Sym^d(V)
\]
from the stack of coherent sheaves of length $d$.  We define
\[
\Coha_{\mathfrak{C}oh^V(\AA^2)}\cong \bigoplus_{d\geq 0} \HO^{\BM}(\mathfrak{C}oh^V_d(\AA^2),\mathbb{Q}).
\]
\begin{corollary}
\label{go2d}
Let $j\colon V\hookrightarrow \mathbb{A}^2$ be the inclusion of a constructible subset.  Then there is an isomorphism of complexes of mixed Hodge modules
\[
p_!\mathbb{Q}_{\mathfrak{C}oh^V(\AA^2)}\cong \Sym_{\cup}\!\left(\Delta_{V,*}\QQ_{\coprod_{d\geq 1}V}\otimes \HO_c(\B\mathbb{C}^*,\mathbb{Q})\right)
\]
and a PBW isomorphism of $\mathbb{N}$-graded mixed Hodge structures
\[
\Sym\left(\bigoplus_{d\geq 1}\HO^{\BM}(V,\mathbb{Q})\otimes\LL\otimes \HO(\B\mathbb{C}^*,\mathbb{Q})\right)\xrightarrow{\cong}\Coha_{\mathfrak{C}oh^V(\AA^2)} .
\]
\end{corollary}
\begin{proof}
We denote by
\[
\iota_{V}\colon \Sym(V)\rightarrow \Msp(\overline{Q_{\Jor}})
\]
the inclusion.  Then we compose the ismorphisms
\begin{align}
p_{!}\mathbb{Q}_{\mathfrak{C}oh^V(\AA^2)}\cong&\iota_V^*\p_{\overline{Q_{\Jor}},!}\mathbb{Q}_{\mathfrak{C}}
\nonumber
\\
\label{dimr}
\cong&\iota_V^*\p_{\overline{Q_{\Jor}},!}\pi_!\phin{\mathfrak{Tr}(\tilde{W})}\ICS_{\Mst(\widetilde{Q_{\Jor}})}(\mathbb{Q})\\
\nonumber
\cong&\iota_V^*\varpi_!\p_{\widetilde{Q_{\Jor}},!}\phin{\mathfrak{Tr}(\tilde{W})}\ICS_{\Mst(\widetilde{Q_{\Jor}})}(\mathbb{Q})\\
\label{quickl}
\cong&\iota_V^*\varpi_!\Sym_{\boxtimes_{\oplus}}\left(\Delta_{\mathbb{A}^3,*}\ICS_{\coprod_{d\geq 1}\mathbb{A}^3}(\QQ)\otimes\HO_c(\B\mathbb{C}^*,\mathbb{Q})_{\vir}\right)\\
\nonumber
\cong&\Sym_{\boxtimes_{\oplus}}\left(\iota_V^*\varpi_!\Delta_{\mathbb{A}^3,*}\ICS_{\coprod_{d\geq 1}\mathbb{A}^3}(\QQ)\otimes\HO_c(\B\mathbb{C}^*,\mathbb{Q})_{\vir}\right)\\
\nonumber
\cong&\Sym_{\boxtimes_{\oplus}}\left(\varpi_!\Delta_{U,*}\QQ_{\coprod_{d\geq 1}U}\otimes\LL^{ -1}\otimes\HO_c(\B\mathbb{C}^*,\mathbb{Q})\right)\\
\cong&\Sym_{\boxtimes_{\oplus}}\left(\Delta_{V,*}\mathbb{Q}_{\coprod_{d\geq 1}V}\otimes\HO_c(\B\mathbb{C}^*,\mathbb{Q})\right)\nonumber
\end{align}
where \eqref{dimr} comes from \eqref{dimr0} and the isomorphism \eqref{quickl} comes from Corollary \eqref{stackyUthm}.  This establishes the first isomorphism, the PBW isomorphism follows by the same argument, and \eqref{UAIso}.
\end{proof}

In nonabelian Hodge theory, an interesting special case of Corollary \ref{stackyUthm} comes from setting
\[
V=(\mathbb{C}^*)^2.
\]
Set $A=\mathbb{C}\langle x^{\pm 1},y^{\pm 1}\rangle$.  Then there is a natural identification of substacks of $\Mst(\overline{Q_{\Jor}})$
\[
\Mst(A)=\mathfrak{C}oh((\mathbb{C}^*)^2)=\Mst(\pi_1(\Sigma_1))
\]
where the final stack is the stack of finite-dimensional representations of the fundamental group of a genus 1 closed Riemann surface.  As a special case of Corollary \ref{go2d} we deduce
\begin{corollary}
\label{oneCorr}
There is a PBW isomorphism of $\mathbb{N}$-graded mixed Hodge structures
\begin{equation}
\label{g1PBW}
\bigoplus_{d\in\mathbb{N}}\HO^{\BM}(\Mst(\pi_1(\Sigma_1))_d,\mathbb{Q})\cong \Sym\!\left(\bigoplus_{d\in\mathbb{Z}_{\geq 1}}\HO^{\BM}((\mathbb{C}^*)^2,\mathbb{Q})\otimes \LL\otimes\HO(\B\mathbb{C}^*,\mathbb{Q})\right).
\end{equation}
\end{corollary}
The CoHA structure on the left hand side of \eqref{g1PBW} is introduced and studied in \cite{Da15Co}.
\smallbreak
Given $g,d\in\mathbb{Z}_{\geq 1}$, consider the stack theoretic quotient
\[
\Rep_d^{\tw}(\Sigma_g):=\left\{\substack{(A_1,\ldots,A_g,B_1,\ldots,B_g)\in\Gl_d(\mathbb{C})^{\times 2g}|\\\prod_{n=1}^g(A_n,B_n)=\exp(2\pi i/d)\cdot\Id_{d\times d}}\right\}/\Gl_d(\mathbb{C}),
\]
where the action is the usual simultaneous conjugation action.  The action of $\Gl_d(\mathbb{C})$ on the variety in brackets is not free, but it factors through the conjugation action by $\PGL_d(\mathbb{C})$, which is scheme-theoretically free by \cite[Cor.2.2.7]{HRV08}, and the quotient
\[
\overline{\Rep}_d^{\tw}(\Sigma_g):=\left\{\substack{(A_1,\ldots,A_g,B_1,\ldots,B_g)\in\Gl_d(\mathbb{C})^{\times 2g}|\\\prod_{n=1}^g(A_n,B_n)=\exp(2\pi i/d)\cdot\Id_{d\times d}}\right\}/\PGL_d(\mathbb{C})
\]
is a smooth quasiprojective variety.  It follows that there is an isomorphism
\[
\HO(\Rep_d^{\tw}(\Sigma_g),\mathbb{Q})\cong \HO(\overline{\Rep}_d^{\tw}(\Sigma_g,\mathbb{Q}))\otimes\HO(\B\mathbb{C}^*,\mathbb{Q}).
\]
In the $g=1$ case, we have by \cite[Thm.2.2.17]{HRV08} that 
\begin{equation}
\label{HRV1}
\overline{\Rep}_d^{\tw}(\Sigma_1)\cong(\mathbb{C}^*)^2.
\end{equation}
In the general case, we have the following conjecture \cite[Conj.1.1]{Da15Co}.
\begin{conjecture}
\label{genConj}
There is an isomorphism of $\mathbb{N}$-graded cohomologically graded mixed Hodge structures
\[
\bigoplus_{d\in\mathbb{N}}\HO_c\!\left(\Mst(\pi_1(\Sigma_g))_d,\mathbb{Q}\right)\otimes\mathbb{L}^{(1-g)n^2}\cong \Sym\!\left(\bigoplus_{d\geq 1}\HO_c\!\left(\overline{\Rep}_d^{\tw}(\Sigma_g),\mathbb{Q}\right)\otimes \HO_c(\B\mathbb{C}^*,\mathbb{Q})\otimes \mathbb{L}^{(1-g)n^2}\right).
\]
\end{conjecture}
From Corollary \ref{oneCorr} and (\ref{HRV1}) we deduce the $g=1$ part of the following; the genus zero case follows from \cite[Sec.1]{COHA}.
\begin{theorem}
Conjecture \ref{genConj} is true for $g\leq 1$.
\end{theorem}

\section{Generalisations of the purity theorem}
\label{conseq}
\subsection{The wall crossing isomorphism in DT theory}
The wall crossing isomorphism in cohomological DT theory (e.g. \cite[Thm.B]{DaMe15b}) provides a powerful way to deduce purity of Borel--Moore homology of moduli spaces of semistable quiver representations, for some stability condition $\zeta$, from purity of Borel--Moore homology for some other stability condition $\zeta'$ (see e.g. \cite{QCP} for an application of this principle for quantum cluster algebras).  We will use this idea to prove a generalisation of Theorem \ref{mmCo} incorporating stability conditions.
\smallbreak

Fix a quiver $Q$, and a stability condition $\zeta\in\mathbb{H}_+^{Q_0}$.  Let $\rho$ be a finite-dimensional $\mathbb{C}Q$-module, then $\rho$ admits a unique \textit{Harder--Narasimhan} filtration
\[
0=\rho^0\subset\ldots\subset \rho^s=\rho
\]
such that each $\rho^t/\rho^{t-1}$ is $\zeta$-semistable, and the slopes $\slope(\rho^1/\rho^0),\ldots,\slope(\rho^s/\rho^{s-1})$ are strictly descending.  Given a dimension vector $\dd\in\mathbb{N}^{Q_0}$, we denote by 
\[
\HN_{\dd}:=\left\{(\mathbf{d}^1,\ldots,\mathbf{d}^s)\in(\mathbb{N}^{Q_0})^s\setminus \{0\}|\hbox{ }\slope (\mathbf{d}^1)>\slope (\mathbf{d}^2)>\ldots>\slope (\mathbf{d}^s),\textrm{ }\sum_{ j\leq s}\mathbf{d}^j=\dd\right\}
\]
the set of Harder--Narasimhan types for $\mathbb{C}Q$-modules of dimension $\dd$.  For $\alpha=(\dd^1,\ldots,\dd^s)\in\HN_{\dd}$, we denote $\dd^j$ by $\alpha^j$, and write $s(\alpha)=s$.  For each $\alpha\in\HN_{\dd}$, there is a locally closed quasiprojective subvariety
\[
X(Q)_{[\alpha]}\subset X(Q)
\]
for which the closed points correspond exactly to those $\mathbb{C}Q$-modules $\rho$ of Harder--Narasimhan type $\alpha$.  For $\alpha\in\HN_{\dd}$, define by
\[
X(Q)_{\alpha}\subset X(Q)_{\dd}
\]
the subspace of linear maps preserving the $Q_0$-graded flag
\[
0\subset \mathbb{C}^{\alpha^1}\subset \mathbb{C}^{\alpha^1+\alpha^2}\subset\ldots\subset\mathbb{C}^{\dd},
\]
and such that each subquotient is $\zeta$-semistable, and denote by $\Bl_{\alpha}\subset \Gl_{\dd}$ the subgroup preserving this same flag.  Then the natural map
\[
X(Q)_{\alpha}/\Bl_{\alpha}\rightarrow X(Q)_{[\alpha]}/\Gl_{\dd}
\]
is an isomorphism.  We set 
\[
\Mst(Q)_{\alpha}:=X(Q)_{\alpha}/\Bl_{\alpha}
\]
and denote by 
\[
i_{\alpha}\colon \Mst(Q)_{\alpha}\rightarrow\Mst(Q)_{\dd}
\]
the locally closed inclusion of substacks.  By \cite[Prop.3.4]{Reineke_HN} there is a decomposition into locally closed substacks
\[
\Mst(Q)_{\dd}\cong \coprod_{\alpha\in\HN_{\dd}} \Mst(Q)_{\alpha}.
\]
%We define $\mathfrak{Tr}(W)_{\mathbf{d}}:=\mathfrak{Tr}(W)_d\circ i_{\mathbf{d}}$.  
\smallbreak
The following are the relative and absolute versions of the cohomological wall-crossing isomorphism, respectively \cite[Thm.B]{DaMe15b}.  Since we state them in the general case, which may involve nontrivial monodromy, we first state them in terms of the functor $\phim{\mathfrak{Tr}(W)}$ of Section \ref{DTtheory}.  When we come to use the theorem, we will be back in the trivial monodromy situation, and we will be able to revert to using the functor $\phin{\mathfrak{Tr}(W)}$, as explained in Remark \ref{getOut}.
\begin{theorem}
\label{paraProp}
For $Q$ a quiver, $W\in \mathbb{C}Q/[\mathbb{C}Q,\mathbb{C}Q]$ a potential, and stability condition $\zeta$, there is an isomorphism in $\Dl\left(\MMHM(\Msp(Q))\right)$:
\begin{equation}
\label{paraDecomp}
\p_!\phim{\mathfrak{Tr}(W)}\ICS_{\Mst(Q)}(\mathbb{Q})\cong\bigoplus_{\substack{\dd\in\mathbb{N}^{Q_0}\\ \alpha\in\HN_{\dd}}} \left(\Boxtimes_{\oplus,\mathsmaller{1\leq j\leq s(\alpha)}}q^{\zeta}_{\alpha^j,!}\p^{\zeta}_{\alpha^j,!}\phim{\mathfrak{Tr}(W)^{\zeta}_{\alpha^j}}\ICS_{\Mst(Q)^{\zeta\sst}_{\alpha^j}}(\mathbb{Q})\right)\otimes \mathbb{L}^{f(\alpha)/2}
\end{equation}
where
\[
f((\mathbf{d}^1,\ldots,\mathbf{d}^s)):=\sum_{1\leq j'<j''\leq s}\langle \mathbf{d}^{j'},\mathbf{d}^{j''}\rangle
\]
and $q^{\zeta}_{\dd}\colon \Msp(Q)^{\zeta\sst}_{\dd}\rightarrow \Msp(Q)_{\dd}$ is the affinization morphism.  Taking the direct image to $\mathbb{N}^{Q_0}$, there is an isomorphism in $\Dl(\MMHM(\NN^{Q_0}))$:
\begin{equation}
\label{paraDecompAbs}
\Dim_!\phim{\mathfrak{Tr}(W)}\ICS_{\Mst(Q)}(\mathbb{Q})\cong\bigoplus_{\substack{\dd\in\mathbb{N}^{Q_0}\\ \alpha\in\HN_{\dd}}} \left(\Boxtimes_{1\leq j\leq s(\alpha)}\Dim^{\zeta}_{\alpha^j,!}\phim{\mathfrak{Tr}(W)^{\zeta}_{\alpha^j}}\ICS_{\Mst(Q)^{\zeta\sst}_{\alpha^j}}(\mathbb{Q})\right)\otimes \mathbb{L}^{f(\alpha)/2}
\end{equation}
with $\Dim\colon \Mst(Q)\rightarrow\mathbb{N}^{Q_0}$ as in Definition \ref{dimDef}.
\end{theorem}
If $Q$ is symmetric, the function $f$ in the above proposition is identically zero.\begin{corollary}
\label{withSt}
For any stability condition $\zeta\in\mathbb{H}_+^{Q_0}$, the cohomologically graded mixed Hodge structure
\begin{equation}
\label{puStab}
\HO_c\!\left(\Mst(\tilde{Q})_{\dd}^{\zeta\sst},\phin{\mathfrak{Tr}(\tilde{W})_{\dd}^{\zeta}}\ICS_{\Mst(\tilde{Q})_{\dd}^{\zeta\sst}}(\mathbb{Q})\right)\in\Dl(\MHS)
\end{equation}
is pure of Tate type.
\end{corollary}
\begin{proof}
Firstly, strictly speaking, the left hand side of (\ref{puStab}), as well as both sides of (\ref{paraDecompAbs}), should be considered as monodromic mixed Hodge structures.  By Lemma \ref{lem2}, for the case in which our QP is of the form $(\tilde{Q},\tilde{W})$ for some quiver $Q$, the left hand side of (\ref{paraDecompAbs}) in fact belongs to the full subcategory of mixed Hodge structures.  For each $\dd\in\mathbb{N}^{\tilde{Q}_0}$ the Harder--Narasimhan type $(\dd)$ contributes the summand 
\begin{equation}
\label{puPart}
\HO_c\!\left(\Mst(\tilde{Q})_{\dd}^{\zeta\sst},\phin{\mathfrak{Tr}(\tilde{W})_{\dd}^{\zeta}}\ICS_{\Mst(\tilde{Q})_{\dd}^{\zeta\sst}}(\mathbb{Q})\right)
\end{equation}
to the right hand side of (\ref{paraDecompAbs}), and so we deduce that as a sub monodromic mixed Hodge module of a monodromic mixed Hodge module that is both an ordinary mixed Hodge module, and pure of Tate type by Lemma \ref{lem2} and Theorem \ref{purityThm}, the mixed Hodge module (\ref{puPart}) is a pure element of $\Dl(\MHS)$, of Tate type.
\end{proof}
\subsection{Purity for stacks of semistable $\Pi_Q$-modules}
Fix a quiver $Q$ and a dimension vector $\dd$.  There is a natural projection
\[
\tau_{Q,\dd}\colon \Mst(\widetilde{Q})_{\dd}\rightarrow \Mst(\overline{Q})_{\dd}
\]
induced by forgetting $\rho(\omega_i)$ for all $i\in Q_0$.  Let $\zeta\in\mathbb{H}_+^{Q_0}$ be a stability condition.  The inclusion
\begin{equation}
\label{troubleInc}
\tau_{Q,\dd}^{-1}\left(\Mst(\overline{Q})_{\dd}^{\zeta\sst} \right)\subset   \Mst(\widetilde{Q})_{\dd}^{\zeta\sst} 
\end{equation}
is strict in general.  This is because the underlying $\overline{Q}$-representation of a $\widetilde{Q}$-representation may have a nontrivial Harder--Narasimhan filtration that is not preserved by the action of the loops $\omega_i$.  We nonetheless have the following useful
\begin{lemma}
\label{supportProp}
For $Q$ an arbitrary finite quiver, $\zeta\in\mathbb{H}_+^{Q_0}$ a stability condition, $\dd\in\mathbb{N}^{Q_0}$ a dimension vector, and $\tau_{Q,\dd}\colon \Mst(\tilde{Q})_{\dd}\rightarrow\Mst(\overline{Q})_{\dd}$ the natural projection, the inclusion
\begin{equation}
\label{trivInc}
\left(\tau^{-1}_{Q,\dd}\left(\Mst(\overline{Q})_{\dd}^{\zeta\sst} \right)\cap\crit\!\left(\mathfrak{Tr}(\tilde{W})^{\zeta}_{\dd}\right)\right)\hookrightarrow   \left(\Mst(\widetilde{Q})_{\dd}^{\zeta\sst} \cap\crit\!\left(\mathfrak{Tr}(\tilde{W})^{\zeta}_{\dd}\right)\right)
\end{equation}
is the identity.  
\end{lemma}
\begin{proof}
Let $\rho$ be a $\Jac(\tilde{Q},\tilde{W})$-representation represented by a closed point of the complement of the inclusion (\ref{trivInc}).  Then via Proposition \ref{basicEq}, $\rho$ corresponds to a pair $(M,f)$, where $M$ is a $\Pi_{Q}$-module, and $f\in\End_{\Pi_Q}(M)$.  By assumption, the Harder-Narasimhan filtration of $M$, considered as a $\Pi_Q$-module, is nontrivial, i.e. it takes the form
\[
0=M_0\subset M_1\subset\ldots\subset M_s=M,
\]
where $s\geq 2$.  Since each $\mu(M_j/M_{j-1})$ for $j\geq 2$ has slope strictly less than $\mu(M_1)$, each $\Hom_{\Pi_Q\lmod}(M_1,M_j/M_{j-1})=0$, and so the restriction 
\[
f|_{M_1}\colon M_1\rightarrow M
\]
factors through the inclusion $M_1\subset M$.  So the pair $(M_1,f|_{M_1})$ is a proper subobject of the pair $(M,f)$ in the category $\mathcal{C}_{\Pi_Q}$ of Proposition \ref{basicEq}.  But then by Proposition \ref{basicEq}, $\rho$ is not a $\zeta$-semistable $\tilde{Q}$-representation, a contradiction.
\end{proof}
This lemma enables us to prove purity for stacks of semistable $\Pi_Q$-modules.
\begin{theorem}
\label{2dstab}
Let $Q$ be a finite quiver, let $\zeta\in\mathbb{H}_+^{Q_0}$ be a stability condition, and let $\dd\in\mathbb{N}^{Q_0}$ be a dimension vector.  There is a natural isomorphism in $\Dl(\MHS)$
\begin{align}\label{uselat}
&\HO_c\!\left(\Mst(\tilde{Q})_{\dd}^{\zeta\sst},\phin{\mathfrak{Tr}(\tilde{W})^{\zeta}_{\dd}}\mathbb{Q}_{\Mst(\tilde{Q})_{\dd}^{\zeta\sst}}\right)\cong\HO_c\!\left(\left(\mu^{-1}_{Q,\dd}(0)\cap X(\overline{Q})^{\zeta\sst}_{\dd}\right)/\Gl_{\dd},\mathbb{Q}\right)\otimes\mathbb{L}^{\dd\cdot\dd},
\end{align}
and so by Theorem \ref{purityWithStab}, taking duals, the mixed Hodge structure
\[
\HO^{\BM}(\Mst(\Pi_Q)^{\zeta\sst}_{\dd},\QQ)\cong\HO^{\BM}\!\left(\left(\mu^{-1}_{Q,\dd}(0)\cap X(\overline{Q})^{\zeta\sst}_{\dd}\right)/\Gl_{\dd},\mathbb{Q}\right)
\]
is pure of Tate type.
\end{theorem}
\begin{proof}
Write 
\begin{equation}
\label{Vdef}
V=\tau^{-1}_{Q,\dd}\left(\Mst(\overline{Q})_{\dd}^{\zeta\sst} \right).
\end{equation}
By Theorem \ref{dimRedThm} there is an isomorphism
\begin{equation}
\label{twtr}
\HO_c\!\left(\left(\mu^{-1}_{Q,\dd}(0)\cap X(\overline{Q})^{\zeta\sst}_{\dd}\right)/\Gl_{\dd},\mathbb{Q}\right)\otimes\mathbb{L}^{\dd\cdot\dd}\cong \HO_c\!\left( V,\phin{\mathfrak{Tr}(\tilde{W})^{\zeta}_{\dd}}\mathbb{Q}_{V}\right). 
\end{equation}
There are equalities
\begin{align*}
\supp\left(\phin{\mathfrak{Tr}(\tilde{W})^{\zeta}_{\dd}}\mathbb{Q}_{V}\right)=&\left(V\cap\crit\left(\mathfrak{Tr}(\tilde{W})^{\zeta}_{\dd}\right)\right)\\
=&\left(\Mst(\widetilde{Q})_{\dd}^{\zeta\sst} \cap\crit\left(\mathfrak{Tr}(\tilde{W})^{\zeta}_{\dd}\right)\right)&(\textrm{Lemma }\ref{supportProp})\\
=&\supp\left(\phin{\mathfrak{Tr}(\tilde{W})_{\dd}}\mathbb{Q}_{\Mst(\tilde{Q})_{\dd}^{\zeta\sst}}\right).
\end{align*}
Thus the natural map
\begin{equation}
\label{twtri}
\HO_c\!\left( V,\phin{\mathfrak{Tr}(\tilde{W})^{\zeta}_{\dd}}\mathbb{Q}_{V}\right)\rightarrow \HO_c\!\left(\Mst(\tilde{Q})_{\dd}^{\zeta\sst},\phin{\mathfrak{Tr}(\tilde{W})_{\dd}^{\zeta}}\mathbb{Q}_{\Mst(\tilde{Q})_{\dd}^{\zeta\sst}}\right)
\end{equation}
is an isomorphism.  Combining \eqref{twtr} and \eqref{twtri} with Corollary \ref{withSt} we deduce the result.
\end{proof}

\subsection{Framed quivers}
\label{framedSec}
For $Q'$ a quiver, $\ff,\dd\in Q'_0$ a pair of dimension vectors, and $\zeta\in\mathbb{H}_+^{Q'_0}$ a stability condition for $Q'$, recall from Section \ref{framedVer} the construction of the moduli space $\Msp(Q')^{\zeta}_{\ff,\dd}$ of $\ff$-framed $\zeta$-semistable $\dd$-dimensional $Q'$-representations.  We consider this construction in the case where $Q'=\tilde{Q}$, the tripled quiver associated to a quiver $Q$.  As in Equation (\ref{frDef}) we define 
\[
\pi^{\zeta}_{\tilde{Q},\ff,\dd}\colon \Msp(\tilde{Q})_{\ff,\dd}^{\zeta}\rightarrow \Msp(\tilde{Q})_{\dd}^{\zeta\sst}
\]
to be the map forgetting the framing and remembering the associated graded object of the Jordan--H\"older filtration (in the category of $\zeta$-semistable $\tilde{Q}$-representations) of the underlying $\tilde{Q}$-representation.
\begin{theorem}
\label{purityWithStab}
Fix a finite quiver $Q$, a dimension vector $\ff\in\mathbb{N}^{Q_0}$, a King stability condition $\zeta\in\mathbb{H}_+^{Q_0}$, and $\dd\in\mathbb{N}^{Q_0}$.  Then the $\mathbb{N}^{Q_0}$-graded mixed Hodge structure on the total vanishing cycle cohomology
\[
\HO_c\!\left(\Msp(\tilde{Q})^{\zeta}_{\ff,\dd},\phin{\mathcal{T}r(\tilde{W})^{\zeta}_{\ff,\dd}}\ICS_{\Msp(\tilde{Q})^{\zeta}_{\ff,\dd}}(\mathbb{Q})\right)
\]
on the fine moduli space of $\zeta$-semistable $\ff$-framed $\mathbb{C}\tilde{Q}$-modules is pure, of Tate type.
\end{theorem}
\begin{proof}
Applying $\dim^{\zeta}_{\theta,!}$ to the isomorphism (\ref{piIso}) we obtain the isomorphism
\begin{align}\label{absHint}
&(\dim^{\zeta}_{\theta}\circ \pi^{\zeta}_{\tilde{Q},\ff,\theta})_!\bigoplus_{\dd\in\Lambda_{\theta}^{\zeta}}\phin{\mathcal{T}r(\tilde{W})_{\ff,\dd}^{\zeta}}\QQ_{\Msp(\tilde{Q})^{\zeta}_{\ff,\dd}}\otimes\LL^{(\dd,\dd)_{\tilde{Q}}/2}\cong\\&
\Sym_{\boxtimes_+}\!\left(\bigoplus_{\dd\in\Lambda_{\theta}^{\zeta}} \DT^{\zeta}_{\tilde{Q},\tilde{W},\dd}\otimes\HO(\mathbb{P}^{\ff\cdot \dd-1},\mathbb{Q})^{\vee}\otimes\LL^{-1/2}\right).\nonumber
\end{align}
On the other hand, from Corollary \ref{BPSpure} each of the complexes $\DT^{\zeta}_{\tilde{Q},\tilde{W},\dd}$ are pure.  The purity of the right hand side of (\ref{absHint}) follows, and so does the theorem.
\end{proof}
\subsection{Critical cohomology of $\Hilb(\AA^3)$}
We consider again the special case in which $Q=Q_{\Jor}$, and so $\tilde{Q}$ is a quiver with one vertex and three loops, which we label $x,y,z$, and $\tilde{W}=x[y,z]$.  Setting $\ff=1$, there is a natural isomorphism of schemes (see \cite{BBS})
\begin{equation}
\label{HilbIso}
\Msp(\tilde{Q})_{1,n}\cap\crit(\mathcal{T}r(\tilde{W})_{n})\cong\Hilb_n(\AA^3)
\end{equation}
where the right hand side of (\ref{HilbIso}) is the usual Hilbert scheme parameterising codimension $n$ ideals $I\subset \mathbb{C}[x,y,z]$.  The following is then a corollary of Theorem \ref{purityWithStab}:
\begin{corollary}
\label{nchP}
The mixed Hodge structure
\[
\HO_c\!\left(\Hilb_n(\AA^3),\phi_{\mathcal{T}r(\tilde{W})_n}\mathbb{Q}_{\Msp(\tilde{Q})_{1,n}}\right)
\]
is pure of Tate type for all $n$.  
\end{corollary}
The calculation of the class of 
\[
\left[\HO_c\!\left(\Hilb_n(\AA^3),\phi_{\mathcal{T}r(\tilde{W})_n}\mathbb{Q}_{\Msp(\tilde{Q})_{1,n}}\right)\right]
\]
in a suitable completion of the Grothendieck group of mixed Hodge modules is one of the main results of \cite{BBS}, following on from the earlier paper \cite{DS09}, where an in depth analysis of the $n=4$ case was undertaken.  It follows from our purity result that the Hodge polynomial $\hp\!\left(\HO_c\left(\Hilb_n(\AA^3),\phi_{\mathcal{T}r(\tilde{W})_d}\mathbb{Q}_{\Msp(\tilde{Q})_{1,n}}\right),x,y,z\right)$ is equal to the weight polynomial $\wt\!\left(\HO_c\left(\Hilb_n(\AA^3),\phi_{\mathcal{T}r(\tilde{W})_n}\right),q\right)$ after the substitution $q^2=xyz^2$, and we deduce from \cite[Thm.2.7]{BBS} the following generating function equation for the Hodge polynomial of the vanishing cycle cohomology for $\Hilb_n(\AA^3)$:
\[
\sum_{n\geq 0}\hp\!\left(\HO_c\!\left(\Hilb_n(\AA^3),\phi_{\mathcal{T}r(\tilde{W})_n}\mathbb{Q}_{\Msp(\tilde{Q})_{1,n}}\right),x,y,z\right)(xyz^2)^{-n-n^2}t^n=\prod_{n=1}^{\infty}\prod_{k=0}^{n-1}(1-(xyz^2)^{1-k}t^n)^{-1}.
\]
Indeed, we can determine the critical cohomology of $\Hilb_n(\AA^3)$ itself, without passing to any motivic invariants:
\begin{corollary}
There is an isomorphism of $\mathbb{N}$-graded, cohomologically graded mixed Hodge structures:
\begin{equation}
\label{stg}
\bigoplus_{n\in\mathbb{N}}\HO_c\!\left(\Hilb_n(\AA^3),\phi_{\mathcal{T}r(\tilde{W})_n}\mathbb{Q}_{\Msp(\tilde{Q})_{1,n}}\right)\otimes \LL^{-n-n^2}\cong \Sym\left(\bigoplus_{n\geq 1} \bigoplus_{0\leq k\leq n-1}\LL^{1-k}\right).
\end{equation}
\end{corollary}
\begin{proof}
By Corollary \ref{nchP} the left hand side of \eqref{stg} is pure of Tate type, as is the right hand side (by definition).  A cohomologically graded mixed Hodge structure that is pure, of Tate type, is determined by its weight polynomial.  The required equality of weight polynomials follows from the main result of \cite{BBS}.
\end{proof}
\subsection{Nakajima quiver varieties}
In Section \ref{framedSec} we considered the mixed Hodge structures on the vanishing cycle cohomology of framed representations of the quiver $\tilde{Q}$, where the framing results in a quiver that is not symmetric, i.e. we perform the operation of framing the quiver \textit{after} the operation $Q\mapsto \tilde{Q}$.  By reversing the order of these operations, we derive our results on Nakajima quiver varieties.

\smallbreak

Let $Q$ be an arbitrary quiver, and let $\zeta\in\mathbb{H}_+^{Q_0}$ be a stability condition.  Let $\ff\in\mathbb{N}^{Q_0}$ be a framing vector.  Throughout this section we assume that $\ff\neq 0$.  Consider the quiver $\widetilde{Q_{\ff}}$, where the tilde covers the $\ff$ as well as the $Q$; this is the quiver obtained by doubling the framed quiver $Q_{\ff}$ and then adding a loop $\omega_i$ at every vertex (including the vertex $\infty$).  

\smallbreak

Fix a slope $\theta\in(-\infty,\infty)$.  We define the stability condition $\zeta^{(\theta)}$ as in Section \ref{framedVer}.  Assume that $\dd\in\Lambda_{\theta}^{\zeta}\subset\mathbb{N}^{Q_0}$.  Then a $(1,\dd)$-dimensional $\widetilde{Q_{\ff}}$-representation $\rho$ is $\zeta^{(\theta)}$-stable if and only if the underlying $\tilde{Q}$-representation is $\zeta$-semistable, and for every proper subrepresentation $\rho'\subset \rho$ such that $\dim(\rho')_{\infty}=1$, the underlying $\tilde{Q}$-representation of $\rho'$ has slope strictly less than $\theta$.  In addition, $\zeta^{(\theta)}$-stability for $\widetilde{Q_{\ff}}$-representations of dimension $(1,\dd)$ is equivalent to $\zeta^{(\theta)}$-semistability.
\smallbreak
For each of the vertices $i\in Q_0$, the condition $\mu_{(1,\dd)}(\rho)=0$ imposes the conditions 
\begin{equation}
\label{essRel}
T_i:=\sum_{t(a)=i}\rho(a)\rho(a^*)-\sum_{s(a)=i}\rho(a^*)\rho(a)+\sum_{i\in Q_0}\sum_{1\leq n\leq \ff_i}\rho(\beta_{i,n})\rho(\beta^*_{i,n})=0
\end{equation}
which are the usual Nakajima quiver variety relations \cite{Nak94} \cite{Nak98}, while at the vertex $\infty$, the relation imposed is 
\begin{equation}
\label{redRel}
T_{\infty}:={}-\sum_{i\in Q_0}\sum_{1\leq n\leq \ff_i}\rho(\beta^*_{i,n})\rho(\beta_{i,n})=0.
\end{equation}
By cyclic invariance of the trace, we have
\[
\sum_{i\in (Q_{\ff})_0}\Tr(T_i)=0
\]
and so $T_{\infty}=\Tr(T_{\infty})=0$ follows already from the relations (\ref{essRel}), and (\ref{redRel}) is redundant.  It follows that 
\[
\left(\mu^{-1}_{Q_{\ff},(1,\dd)}(0)\cap X(\overline{Q}_{\ff})^{\zeta^{(\theta)}\sst}_{(1,\dd)}\right)/\Gl_{\dd}
\]
is the usual Nakajima quiver variety, which we will denote $\Nak^{\zeta}(\dd,\ff)$.  There is an isomorphism
\begin{equation}
\label{naksBack}
\HO_c\!\left(\left(\mu^{-1}_{Q_\ff,(1,\dd)}(0)\cap X(\overline{Q}_{\ff})^{\zeta^{(\theta)}\sst}_{(1,\dd)}\right)/\Gl_{(1,\dd)},\mathbb{Q}\right)\cong \HO_c(\Nak^{\zeta}(\dd,\ff),\mathbb{Q})\otimes\HO_c(\B \mathbb{C}^*,\mathbb{Q}).
\end{equation}
Each $\Nak^{\zeta}(\dd,\ff)$ is smooth, and so we have $\HO_c(\Nak^{\zeta}(\dd,\ff),\mathbb{Q})\cong \HO(\Nak^{\zeta}(\dd,\ff),\mathbb{Q})^{\vee}\otimes\mathbb{L}^{\dim(\Nak^{\zeta}(\dd,\ff))}$, and we recover the following corollary.
\begin{corollary}
\label{oriPure}
For an arbitrary quiver $Q$, nonzero dimension vectors $\ff,\dd\in\mathbb{N}^{Q_0}$, and a King stability condition $\zeta\in\mathbb{H}_+^{Q_0}$, the singular cohomology of the Nakajima quiver variety
\[
\HO(\Nak^{\zeta}(\dd,\ff),\mathbb{Q})
\]
is pure of Tate type.
\end{corollary}
%The implications of the cohomological integrality theorem for the cohomology of Nakajima quiver varieties and their associated geometric representation theory will be further developed in a subsequent paper.  

\section{The PBW and wall crossing isomorphisms}\label{BC}
%Using both the support lemmas (Lemmas \ref{lem1} and \ref{supportProp}), we now construct the wall crossing and PBW isomorphisms for the Borel--Moore homology of stacks of representations of $\Pi_Q$, for an arbitrary finite quiver $Q$.
\subsection{Serre subcategories}
\label{PProof}
Let $\mathcal{S}\subset \mathbb{C}\overline{Q}\lmod$ be a Serre subcategory of the category of finite-dimensional $\mathbb{C}\overline{Q}$-modules, i.e. we choose a property $P$ of $\mathbb{C}\overline{Q}$-modules such that for every short exact sequence
\[
0\rightarrow M'\rightarrow M\rightarrow M''\rightarrow 0
\]
inside $\mathbb{C}\overline{Q}\lmod$, $M'$ and $M''$ have property $P$ if and only if $M$ does.  Then $\mathcal{S}\subset \mathbb{C}\overline{Q}\lmod$ is the full subcategory of modules having property $P$.  We assume that there is an inclusion of algebraic stacks $\iota\colon \Mst(\overline{Q})^{\mathcal{S}}\hookrightarrow \Mst(\overline{Q})$ which induces the inclusion of the objects of $\mathcal{S}$ into the objects of $\mathbb{C}\overline{Q}\lmod$ after passing to $\mathbb{C}$-points.  

\smallbreak
The standard construction for $P$ is as follows.  For a quiver $Q$, let $C(Q)$ denote the set of equivalence classes of cycles in $Q$, i.e. the set of cyclic paths, where if $ll'$ and $l'l$ are both cyclic paths, they are considered to be equivalent.  For every cycle $c\in C(Q)$, we pick a constructible subset $U_c\subset \mathbb{C}$, and we say that a $\mathbb{C}\overline{Q}$-module $\rho$ has property $P$ if and only if the generalised eigenvalues of $\rho(\overline{c})$ belong to $U_c$, for each $\overline{c}$ a representative of $c\in C(U)$.
\begin{example}
\label{nilp}
Setting all $U_c=\{0\}$, $\mathcal{S}\subset \mathbb{C}\overline{Q}\lmod$ is the subcategory of nilpotent modules, i.e. those modules $M$ for which there exists some $n\in\mathbb{N}$ such that 
\[
\mathbb{C}\overline{Q}_{\geq n} \cdot M=0.
\]
\end{example}
\begin{example}
\label{snilp}
Setting
\[
U_c=\begin{cases} \mathbb{C}&\textrm{if }c\in C(Q)\\ \{0\}&\textrm{otherwise.}\end{cases}
\]
we obtain the condition for the Lusztig nilpotent variety if $Q$ has no loops.  In general, the Serre subcategory $\mathcal{S}\subset\mathbb{C}\overline{Q}\lmod$ determined by this choice of $U_c$ is the subcategory of modules $M$ for which there exists a filtration by $Q_0$-graded vector spaces $0\subset L^1\subset \ldots\subset L^n$ of the underlying $Q_0$-graded vector space of $M$, such that $a\cdot L^s\subset L^s$ for all $s$, and $a^*\cdot L^s\subset L^{s-1}$.  This second property is obviously of Serre type.  
\smallbreak
The equivalence of these two Serre properties is demonstrated as follows.  Say $M$ is a $\mathbb{C}\overline{Q}$-module in the Serre subcategory determined by the above choices of $U_c$.  Then every $p\in \mathbb{C}\overline{Q}_{\geq 1}\setminus\mathbb{C}Q_{\geq 1}$ acts on $M$ via a nilpotent operator.  By Engel's theorem there is a filtration of vector spaces $0=M_0\subset M^1\subset\ldots\subset M^n$ such that for each such $p$ there is an inclusion $p\cdot M^s\subset M^{s-1}$.  Now set $L^s=\mathbb{C}Q\cdot M^s$ to obtain the required filtration by $\mathbb{C}\overline{Q}$-modules, observing that \[
a^*\cdot\mathbb{C}Q\cdot M^s\in(\mathbb{C}\overline{Q}_{\geq 1}\setminus \mathbb{C}Q_{\geq 1})\cdot M^s\subset M^{s-1}\subset \mathbb{C}Q\cdot M^{s-1}.\]
This condition is introduced under the name of *-semi-nilpotency in \cite{BSV17}.
\end{example}
\begin{example}
\label{ssnilp}
For a final example we turn to \cite{Bo14}.  Set 
\[
U_c=\begin{cases} \mathbb{C} &\textrm{if }c\in\mathbb{C}Q'\textrm{ where }Q'\textrm{ is a subquiver of }$Q$\textrm{ containing only one vertex}\\
0&\textrm{otherwise.}\end{cases}
\]
A $\mathbb{C}\overline{Q}$-module is called *-strongly semi-nilpotent\footnote{In fact this is the modified terminology of \cite{BSV17}.} if it possesses a filtration as in Example \ref{snilp}, for which each subquotient $L^s/L^{s-1}$ is supported at a single vertex.  These are exactly the modules in the Serre subcategory corresponding to the above choices of $U_c$.  For this equivalence one argues as in Example \ref{snilp}.
\end{example}

%We denote by $\Msp(\overline{Q})^{\mathcal{S},\zeta\sst}\subset \Msp(\overline{Q})^{\zeta\sst}$ the subspace of $\zeta$-semistable points $x$ corresponding to modules $\rho$ belonging to $\mathcal{S}$, and denote by $\Mst(\overline{Q})^{\mathcal{S},\zeta\sst}\subset\Mst(\overline{Q})^{\zeta\sst}$ the analogous substack.  
\subsection{Proof of Theorem \ref{PoinThm} and \ref{2dint}}
Applying the functor $\dim_!\overline{\iota}^*$ to the isomorphism constructed in the next theorem yields Theorem \ref{PoinThm}.
\begin{theorem}
\label{2dWC}
Pick a stability condition $\zeta\in\mathbb{H}_+^{Q_0}$.  There is an isomorphism in $\Dl(\MHM(\Msp(\overline{Q})))$
\begin{align*}
&\bigoplus_{\dd\in\mathbb{N}^{Q_0}}\p_{\overline{Q},\dd,!}\QQ_{\Mst(\Pi_Q)_{\dd}}\otimes \LL^{(\dd,\dd)}
\cong\\&
\Boxtimes_{\oplus,\theta\in(-\infty,\infty)}\bigoplus_{\dd\in\dvst}\left( q_{\overline{Q},\dd,!}^{\zeta}\p_{\overline{Q},\dd,!}^{\zeta}\QQ_{\Mst(\Pi_Q)_{\dd}^{\zeta\sst}}\otimes \LL^{(\dd,\dd)}
\right).
\end{align*}
\end{theorem}

\renewcommand*{\proofname}{Proof of Theorem \ref{2dWC}}
\begin{proof}
We consider the commutative diagram
\[
\xymatrix{
\Mst(\tilde{Q})^{\zeta\sst}_{\dd}\ar[d]^{\p^{\zeta}_{\tilde{Q},\dd}}&\ar@{_{(}->}[l]_-jV\ar[r]^-{\tau_{Q,\dd}}&\Mst(\overline{Q})_{\dd}^{\zeta\sst}\ar[d]^{\p^{\zeta}_{\overline{Q},\dd}}\\
\Msp(\tilde{Q})^{\zeta\sst}_{\dd}\ar[d]^{q_{\tilde{Q},\dd}^{\zeta}}&&\Msp(\overline{Q})^{\zeta\sst}_{\dd}\ar[d]^{q_{\overline{Q},\dd}^{\zeta}}\\
\Msp(\tilde{Q})_{\dd}\ar[rr]^{\tau'_{Q,\dd}}&&\Msp(\overline{Q})_{\dd}.
}
\]
With $V$ defined as in \eqref{Vdef}.  By Theorem \ref{dimRedThm} there are isomorphisms
\begin{align}
\label{hwga}
\p^{\zeta}_{\overline{Q},\dd,!}\QQ_{\Mst(\Pi_Q)^{\zeta\sst}_{\dd}}\otimes\LL^{\dd\cdot \dd}\cong \p^{\zeta}_{\overline{Q},\dd,!}\tau_{Q,\dd,!}\phin{\mathfrak{Tr}(\tilde{W})^{\zeta}_{\dd}}\QQ_{V}
\\
\p_{\overline{Q},\dd,!}\QQ_{\Mst(\Pi_Q)^{\zeta\sst}_{\dd}}\otimes\LL^{ \dd\cdot \dd}\cong\tau'_{Q,\dd,!}\p_{\tilde{Q},\dd,!}\phin{\mathfrak{Tr}(\tilde{W})_{\dd}}\QQ_{\Mst(\tilde{Q})_{\dd}}.\nonumber
\end{align}
By Lemma \ref{supportProp} the support of $\phin{\mathfrak{Tr}(\tilde{W})^{\zeta}_{\dd}}\QQ_{\Mst(\tilde{Q})^{\zeta\sst}_{\dd}}$ is contained in the image of the natural inclusion $j$, and so from \eqref{hwga} and the above commutative diagram we obtain the isomorphism
\[
q_{\overline{Q},\dd,!}^{\zeta}\p^{\zeta}_{\overline{Q},\dd,!}\QQ_{\Mst(\Pi_Q)^{\zeta\sst}_{\dd}}\cong \tau'_{Q,\dd,!}q^{\zeta}_{\tilde{Q},\dd,!}\p^{\zeta}_{\tilde{Q},\dd,!}\phin{\mathfrak{Tr}(\tilde{W})_{\dd}^{\zeta}}\QQ_{\Mst(\tilde{Q})^{\zeta\sst}_{\dd}}.
\]
Thus, applying $\tau'_{Q,!}$ to the isomorphism \eqref{paraDecomp} applied to the QP $(\tilde{Q},\tilde{W})$ yields the required isomorphism.
\end{proof}
\renewcommand*{\proofname}{Proof}

We again use Lemma \ref{supportProp} to deduce Theorem \ref{2dint} from the analogous result for general quivers with potential:
\begin{proof}[Proof of Theorem \ref{2dint}]
Let $m\colon\AA^1\times \Msp(\Pi_Q)_{\dd}^{\zeta\sst}\rightarrow \Msp(\tilde{Q})_{\dd}^{\zeta\sst}$ be the morphism extending a $\Pi_Q$-module to a $\Jac(\tilde{Q},\tilde{W})$-module by letting all of the loops $\omega_i$ act by multiplication by a fixed scalar in $\AA^1$.  Then by Theorem \ref{2dbpsthm} there is an isomorphism
\begin{equation}
\label{2dhwc}
\DTS^{\zeta}_{\tilde{Q},\tilde{W},\dd}\cong m_*(\ICS_{\AA^1}(\QQ)\boxtimes\DTS_{\Pi_Q}^{\zeta})
\end{equation}
where $\DTS_{\Pi_Q}^{\zeta}$ is a mixed Hodge module on $\Msp(\Pi_Q)^{\zeta\sst}_{\dd}$. 
\smallbreak
Consider the commutative diagram
\[
\xymatrix{
\Mst(\Jac(\tilde{Q},\tilde{W}))^{\zeta\sst}_{\dd}\ar[d]^{\tilde{\p}^{\zeta}_{\dd}}\ar[r]^-{\tau_{\dd}}&\Mst(\Pi_Q)^{\zeta\sst}_{\dd}\ar[d]^{\overline{\p}^{\zeta}_{\dd}}\\
\Msp(\Jac(\tilde{Q},\tilde{W}))^{\zeta\sst}_{\dd}\ar[r]^-{\tau'_{\dd}}&\Msp(\Pi_Q)^{\zeta\sst}_{\dd}.
}
\]
Arguing as in the proof of Theorem \ref{2dWC}, there are isomorphisms
\begin{align*}
\bigoplus_{\dd\in\dvst}\overline{\p}^{\zeta}_{\dd,!}\QQ_{\Mst(\Pi_Q)^{\zeta\sst}_{\dd}}\otimes\LL^{(\dd,\dd)}\cong &\overline{\p}^{\zeta}_{\theta,!}\tau_{\theta,!}\phin{\mathfrak{Tr}(\tilde{W})^{\zeta}_{\theta}}\ICS_{\Mst(\tilde{Q})^{\zeta\sst}_{\theta}}(\QQ)\\
\cong &\tau'_{\theta,!}\tilde{\p}_{\theta,!}^{\zeta}\phin{\mathfrak{Tr}(\tilde{W})^{\zeta}_{\theta}}\ICS_{\Mst(\tilde{Q})^{\zeta\sst}_{\theta}}(\QQ)\\
\cong &\tau'_{\theta,!}\Sym_{\boxtimes_{\oplus}}\!\left(\DTS_{\tilde{Q},\tilde{W},\theta}^{\zeta}\otimes\HO_c(\B \mathbb{C}^*,\QQ)_{\vir}\right)\\
\cong &\Sym_{\boxtimes_{\oplus}}\!\left(\tau'_{\theta,!}\DTS_{\tilde{Q},\tilde{W},\theta}^{\zeta}\otimes\HO_c(\B \mathbb{C}^*,\QQ)_{\vir}\right)\\
\cong &\Sym_{\boxtimes_{\oplus}}\!\left(\DTS_{\Pi_Q,\theta}^{\zeta}\otimes\HO(\B \mathbb{C}^*,\QQ)^{\vee}\right),
\end{align*}
giving the isomorphism \eqref{rel2dI}. 

The construction of the PBW isomorphism is similar; via dimensional reduction and Lemma \ref{supportProp} there is an isomorphism
\[
\Coha_{\Pi_Q,\theta}^{\Sp,\zeta}\cong \Coha_{\tilde{Q},\tilde{W},\theta}^{\tilde{\Sp},\zeta},
\]
where $\tilde{\Sp}$ is the Serre subcategory of $\mathbb{C}\tilde{Q}$-modules for which the underlying $\mathbb{C}\overline{Q}$-module is in $\Sp$, defining a Hall algebra structure on $\Coha_{\Pi_Q,\theta}^{\Sp,\zeta}$.  Then the required PBW isomorphism is constructed from Theorem \ref{PBW3d} and the isomorphisms
\begin{align*}
\DT^{\tilde{\Sp},\zeta}_{\tilde{Q},\tilde{W},\theta}= &\HO_c(\Msp(\tilde{Q})^{\tilde{\Sp},\zeta\sst}_{\theta},\DTS^{\zeta}_{\tilde{Q},\tilde{W},\theta})^{\vee}\\
\cong &\HO_c(\Msp(\overline{Q})^{\Sp,\zeta\sst}_{\theta},\DTS^{\zeta}_{\Pi_Q,\theta})^{\vee}\otimes\LL^{ -1/2}.
\end{align*}
following from the isomorphism \eqref{2dhwc}.
\end{proof}

\subsection{Applications for Nakajima quiver varieties}
\label{AppBC}
We explain the special case of Theorem \ref{PoinThm} which gives rise to Hausel's original formula for the Poincar\'e polynomials of Nakajima quiver varieties.  In brief, we choose $\Pi_{Q_{\ff}}$ to be the preprojective algebra for a framed quiver $Q_{\ff}$, pick $\zeta$ to be the usual stability condition defining the Nakajima quiver variety, set $\mathcal{S}=\mathbb{C}Q_{\ff}\lmod$ and specialise the Hodge series to the Poincar\'e series, to derive Hausel's result.  For this set of choices, an analogue of equation (\ref{decaf}) was demonstrated by Dimitri Wyss \cite{Wyss15}, working in the naive Grothendieck ring of exponential motives.  We describe in a little more detail how our derivation runs.

\smallbreak

Firstly, let $Q$ be a quiver, and let $\mathcal{S}\subset \mathbb{C}\overline{Q}\lmod$ be a Serre subcategory.  Let $\ff\in\mathbb{N}^{Q_0}$ be a framing vector, assumed nonzero, and let $\mathcal{S}_{\ff}\subset \Pi_{Q_{\ff}}\lmod$ be the Serre subcategory consisting of those modules for which the underlying $\mathbb{C}\overline{Q}$-module is in $\mathcal{S}$.  We let $\zeta=(i,\ldots,i)$ be the degenerate stability condition on $Q$, and define $\zeta^{(0)}$ as in Section \ref{framedVer}.  If $X$ is an Artin stack, we define its Poincar\'e series via
\[
\poi(X,q)=\hp(\HO_c(X,\mathbb{Q}),1,1,q).
\]
Equating coefficients in (\ref{decaf}) for which $\dd_{\infty}=1$, and specialising, we obtain
\begin{align}\nonumber
&\sum_{\dd\in\mathbb{N}^{Q_0}}\poi(\Mst(\Pi_{Q_{\ff}})^{\mathcal{S}_{\ff}}_{(1,\dd)},q)q^{2((\dd,\dd)-\ff\cdot \dd+1)} x^{\dd}=
\\ \nonumber
&\left(\sum_{\dd\in\mathbb{N}^{Q_0}} \poi(\Mst(\Pi_Q)^{\mathcal{S}}_{\dd},q)q^{2(\dd,\dd)}x^{\dd}\right)\left(\sum_{\dd\in\mathbb{N}^{Q_0}} \poi(\Mst(\Pi_{Q_{\ff}})^{\mathcal{S}_{\ff},\zeta^{(0)}\sst}_{(1,\dd)},q)q^{2((\dd,\dd)-\ff\cdot \dd+1)}x^{\dd}  \right)
\nonumber 
\end{align}
or by (\ref{naksBack})
\begin{align}
\label{HuaInf}
&\sum_{\dd\in\mathbb{N}^{Q_0}}\poi(\Mst(\Pi_{Q_{\ff}})^{\mathcal{S}_{\ff}}_{(1,\dd)},q)q^{2((\dd,\dd)-\ff\cdot \dd+1)} x^{\dd}=
\\
&\left(\sum_{\dd\in\mathbb{N}^{Q_0}}  \poi(\Mst(\Pi_Q)^{\mathcal{S}}_{\dd},q)q^{2(\dd,\dd)}x^{\dd}\right)\left(\sum_{\dd\in\mathbb{N}^{Q_0}}  \poi(\Nak(\ff,\dd)^{\mathcal{S}},q)q^{2((\dd,\dd)-\ff\cdot \dd+1)}x^{\dd}(q^2-1)^{-1}  \right)\label{HuaIn}
\end{align}
where we have used the isomorphism (\ref{naksBack}) for the final equality, and $\Nak(\ff,\dd)^{\mathcal{S}}$ is the subvariety of the Nakajima quiver variety for the dimension vector $\dd$ and framing vector $\ff$ corresponding to those points for which the underlying $\overline{Q}$-representation is in $\mathcal{S}$.  Putting $\mathcal{S}=\mathbb{C}\overline{Q}\lmod$ (or, equivalently, removing $\mathcal{S}$ from the above formulae) and using Hua's formula \cite{Hua00} to rewrite (\ref{HuaInf}) and the first term of (\ref{HuaIn}) as rational functions in $q$ defined in terms of Kac polynomials, we recover Theorem 5 of \cite{Hau06}.  The advance that Theorem \ref{PoinThm} gives us is an upgrade from an equality of generating series to an isomorphism in cohomology, i.e. it tells us that \eqref{HuaInf} is induced by a graded isomorphism of (pure) Hodge structures
\begin{align*}
&\bigoplus_{\dd\in\mathbb{N}^{Q_0}}\HO_c\!\left(\Mst(\Pi_{Q_{\ff}})^{\mathcal{S}_{\ff}}_{(1,\dd)},\mathbb{Q}\right)\otimes\LL^{ ((\dd,\dd)-\ff\cdot\dd+1)}\cong 
\\&\left(\bigoplus_{\dd\in\mathbb{N}^{Q_0}}\HO_c\!\left(\Mst(\Pi_Q)^{\mathcal{S}}_{\dd},\mathbb{Q}\right)\otimes\LL^{ (\dd,\dd)}\right)\otimes\left( \bigoplus_{\dd\in\mathbb{N}^{Q_0}}\HO_c\!\left(\Nak(\ff,\dd)^{\mathcal{S}},\mathbb{Q}\right)\otimes\LL^{ ((\dd,\dd)-\ff\cdot\dd+1)}\otimes\HO_c(\B\mathbb{C}^*,\mathbb{Q})\right)
\end{align*}
by taking Poincar\'e series of the two sides of the isomorphism.
\section{Restricted Kac polynomials}
\label{AppKac}
\subsection{Definitions}
In this section we explain how Theorem \ref{2dint} enables one to define and categorify the Kac polynomial $\kac_{Q,\dd}^{\mathcal{S}}(q^{1/2})$ associated to a quiver $Q$, a Serre subcategory $\mathcal{S}\subset \mathbb{C}\overline{Q}$, and a dimension vector $\dd$.  Furthermore, we explain a general mechanism for deducing positivity of such Kac polynomials from purity, and in particular prove Theorem \ref{rkacp}.
\sbrk
Defining $\DT_{\Pi_Q}^{\mathcal{S}}$, as per our conventions, as $\DT^{{\Sp},\zeta}_{\Pi_Q}$, for the degenerate stability condition $\zeta=(i,\ldots,i)$ (equivalently, without any stability condition), the dual of isomorphism \eqref{bmirror} yields
\begin{equation}
\label{rkq}
\bigoplus_{\dd\in\mathbb{N}^{Q_0}}\HO_c\!\left(\Mst(\Pi_Q)_{\dd}^{\mathcal{S}},\mathbb{Q}\right)\otimes\LL^{ (\dd,\dd)}\cong\Sym\!\left(\DT^{\mathcal{S},\vee}_{\Pi_Q}\otimes\LL\otimes \HO_c(\B\mathbb{C}^*,\mathbb{Q})\right).
\end{equation}

Isomorphism (\ref{rkq}) can be restated as saying that $\DT^{\mathcal{S},\vee}_{\Pi_Q}$ categorifies the \textit{restricted Kac polynomials} $\kac^{\mathcal{S}}_{Q,\dd}(q^{1/2})$, defined by the plethystic logarithm (the inverse to the plethystic exponential)
\[
q(q-1)^{-1}\sum_{\dd\in\mathbb{N}^{Q_0}}\kac^{\mathcal{S}}_{Q,\dd}(q^{1/2})t^{\dd}=\Log\!\left(\sum_{\dd\in\mathbb{N}^{Q_0}}\wt\!\left(\HO_c(\Mst(\Pi_Q)_{\dd}^{\mathcal{S}},\mathbb{Q}),q^{1/2}\right)q^{(\dd,\dd)}t^{\dd}\right).
\]
Isomorphism \eqref{rkq} and \eqref{aped} imply
\[
\kac^{\mathcal{S}}_{Q,\dd}(q^{1/2})=\wt(\DT^{\mathcal{S}}_{\Pi_Q,\dd},q^{-1/2}).
\]
This is indeed a polynomial: despite its high-tech definition, $\DT^{\mathcal{S}}_{\Pi_Q,\dd}$ is, after all, the hypercohomology of a bounded complex of mixed Hodge modules on an algebraic variety.
\smallbreak
%\begin{example}
%Set $\mathcal{S}=\Pi_Q\lmod$.  In this case the required purity statement is exactly Theorem \ref{purityThm}, and $a^{\mathcal{S}}_{\dd}(\sqrt{q})$ is the original polynomial of Kac, which is shown in \cite{kac83} to be a polynomial in $q$.  We deduce that this polynomial has only positive coefficients, a result originally proved in \cite{HLRV13}.
%\end{example}
\subsection{Positivity of Kac polynomials}
A corollary of the existence of the isomorphism (\ref{rkq}) is that if $\bigoplus_{\dd\in\mathbb{N}_{Q_0}}\HO_c(\Mst(\Pi_Q)_{\dd}^{\mathcal{S}},\mathbb{Q})$ is pure, then so is $\DT_{\Pi_Q}^{\mathcal{S}}$, and as a result, $\kac^{\mathcal{S}}_{Q,\dd}(q^{1/2})$ has only positive coefficients, when expressed as a polynomial in $-{q^{1/2}}$.
  \sbrk
This brings us to the special case of Theorem \ref{2dint} that, along with Theorem \ref{purityThm}, implies the Kac positivity conjecture, first proved by Hausel, Letellier and Villegas in \cite{HLRV13} via arithmetic Fourier analysis for \textit{smooth} Nakajima quiver varieties.  Namely, we set $\mathcal{S}=\mathbb{C}\overline{Q}\lmod$, and we set $\zeta=(i,\ldots,i)$ to be the degenerate stability condition.  Then Theorem \ref{2dint} states that there is an isomorphism
\begin{equation}
\label{util}
\bigoplus_{\dd\in\mathbb{N}^{Q_0}}\HO_c(\Mst(\Pi_Q)_{\dd},\mathbb{Q})\otimes\LL^{ (\dd,\dd)}\cong\Sym\!\left(\DT_{\Pi_Q}^{\vee}\otimes\HO(\B\mathbb{C}^*,\mathbb{Q})^{\vee}\right)
\end{equation}
while Theorem \ref{dimRedThm} states that there is an isomorphism
\[
\bigoplus_{\dd\in\mathbb{N}^{Q_0}}\HO_c\!\left(\Mst(\Pi_Q)_{\dd},\mathbb{Q}\right)\otimes\LL^{ (\dd,\dd)}\cong\bigoplus_{\dd\in\mathbb{N}^{Q_0}}\HO_c\!\left(\Mst(\tilde{Q})_{\dd},\phin{\mathfrak{T}r(\tilde{W})}\ICS_{\Mst(\tilde{Q})_{\dd}}(\mathbb{Q})\right).
\]
On the other hand by \cite[Thm.5.1]{Moz11} there is an equality
\begin{align*}
\sum_{\dd\in\mathbb{N}^{Q_0}}\wt\!\left(\HO_c\!\left(\Mst(\tilde{Q})_{\dd},\phin{\mathfrak{T}r(\tilde{W})}\ICS_{\Mst(\tilde{Q})_{\dd}}(\mathbb{Q})\right),q^{1/2}\right)t^{\dd}=\Exp\!\left(\sum_{\dd\in\mathbb{N}^{Q_0}\setminus \{0\}}\kac_{Q,\dd}(q)(1-q^{-1})^{-1}t^{\dd}\right)
\end{align*}
where $\kac_{Q,\dd}(q)$ is Kac's original polynomial, from which we deduce that 
\[
\wt(\DT_{\Pi_Q,\dd},q^{1/2})=\kac_{Q,\dd}(q^{-1}).
\]
On the other hand, from Corollary \ref{BPSpure}, we deduce that each $\DT_{\Pi_Q,\dd}\cong \DT_{\tilde{Q},\tilde{W},\dd}\otimes\LL^{1/2}$ is pure, and so $\wt(\DT_{\Pi_Q,\dd},q^{1/2})$ is a polynomial in $-q^{1/2}$ with positive coefficients.  In particular, since $\kac_{Q,\dd}(q)$ is a polynomial in $q$, we have reproved the following theorem:
\begin{theorem}\cite{HLRV13}\label{KP}
For a finite quiver $Q$, and a dimension vector $\dd\in\mathbb{N}^{Q_0}$, the Kac polynomial $\kac_{Q,\dd}(q)$ has positive coefficients.
\end{theorem}
\subsection{Positivity of restricted Kac polynomials}
For \textit{new} positivity results, we turn to the examples of Serre subcategories appearing in the work of Bozec, Schiffmann and Vasserot --- see Examples \ref{nilp}, \ref{snilp} and \ref{ssnilp} for the definitions.  Setting $\N,\SN,\SSN\subset \mathbb{C}\overline{Q}\lmod$ to be the full subcategory of nilpotent, *-semi-nilpotent and *-strongly-semi-nilpotent $\mathbb{C}\overline{Q}$-modules, respectively, we define 
\begin{align}
\label{kpdef}
\kac^{\sharp}_{Q,\dd}(q^{1/2})=\wt(\HO_c(\Msp(\overline{Q})_{\dd},\iota'^{\sharp,*}_{\dd}\DTS_{\Pi_Q,\dd}),q^{1/2})
\end{align}
for $\sharp=\N,\SN,\SSN$, where $\iota'^{\sharp}\colon\Msp(\overline{Q})_{\dd}^{\sharp}\hookrightarrow \Msp(\overline{Q})_{\dd}$ is the inclusion.  In this way we obtain a new description of the nilpotent, semi-nilpotent and strongly-semi-nilpotent Kac polynomials of \cite{BSV17}.  Via the results of \cite{BSV17} the polynomials $\kac^{\SN}_{Q,\dd}(q)$ and $\kac_{Q,\dd}^{\SSN}(q)$ have an enumerative definition when $q$ is a prime power: the former counts absolutely indecomposable $\dd$-dimensional $\mathbb{F}_qQ$-modules such that each loop acts via a nilpotent operator, while the latter counts absolutely indecomposable $\dd$-dimensional nilpotent $\mathbb{F}_qQ$-modules.
\sbrk
Note that the above proof of the Kac positivity conjecture (Theorem \ref{KP}) used only the purity of $\HO_c(\Pi_Q,\mathbb{Q})$.  We deduce the following variant of the positivity theorem for Kac polynomials, conjectured in \cite{BSV17}.
\begin{theorem}\label{mpos}
For a finite quiver $Q$, the Kac polynomials $\kac_{Q,\dd}^{\SN}(q)$ and $\kac_{Q,\dd}^{\SSN}(q)$ have positive coefficients.
\end{theorem}
\begin{proof}
The proof proceeds exactly as in the above reproof of Theorem \ref{KP}, using the results and proofs of \cite{BSV17} and \cite{SV17} to deduce purity of $\HO_c(\Mst(\Pi_Q)^{\SN},\mathbb{Q})$ and $\HO_c(\Mst(\Pi_Q)^{\SSN},\mathbb{Q})$.  For example one may extract this purity result as follows: let $\sharp$ be either of the conditions $\SN$ or $\SSN$.  By \cite[Thm.3.2.d]{SV17} the Serre spectral sequence $E_2^{p,q}=\HO^p_T(\pt,\QQ)\otimes \HO^q_{c,\Gl_{\dd}}(\mu_{\dd}^{-1}(0)^{\sharp},\mathbb{Q})^{\vee}=\HO^p_T(\pt,\QQ)\otimes \HO_c(\Mst(\Pi_Q)^{\sharp},\mathbb{Q})^{\vee}$ converging to $\HO_{c,\Gl_{\dd}\times T}(\mu_{\dd}^{-1}(0)^{\sharp},\mathbb{Q})^{\vee}$ degenerates at the second sheet (here $T$ is an extra complex torus acting on all relevant varieties, and is a special case of one of the tori $T^{\tau}$ that we consider in Section \ref{KSCSec}).  In particular, purity of $\HO_c(\Mst(\Pi_Q)^{\sharp},\mathbb{Q})^{\vee}$ follows from purity of $\HO_{c,\Gl_{\dd}\times T}(\mu_{\dd}^{-1}(0)^{\sharp},\mathbb{Q})^{\vee}$, which is \cite[Thm.3.2.b]{SV17}.
\end{proof}

\subsection{Verdier duality and nilpotent Kac polynomials}
We finish our discussion of restricted Kac polynomials with a result relating $\kac_{Q,\dd}(q)$ with $\kac_{Q,\dd}^{\N}(q)$, providing a cohomological refinement of a Kac polynomial identity \cite[Thm.1.4]{BSV17}, which in turn extended the main result of \cite{Sch12} from the case of a quiver without loops.  In the case of quivers without loops, this Kac polynomial identity is explained \cite{Sch12} by Poincar\'e duality for smooth Nakajima quiver varieties.  In general it may be explained by self-Verdier duality of 2d BPS sheaves.
\begin{proposition}\label{ST}
For a quiver $Q$, and a dimension vector $\dd\in\mathbb{N}^{Q_0}$, there is an isomorphism
\begin{equation}
\label{kpip}
\HO_c\!\left(\Msp(\overline{Q}),\DTS_{\Pi_Q,\dd}\right)\cong \HO_c\!\left(\Msp(\overline{Q})^{\N},\DTS_{\Pi_Q,\dd}\right)^{\vee}
\end{equation}
providing a cohomological refinement of the identity
\begin{equation}
\label{kpi}
\kac^{\N}_{Q,\dd}(q)=\kac_{Q,\dd}(q^{-1}).
\end{equation}
\end{proposition}
\begin{proof}
The torus $T=(\mathbb{C}^*)^2$ acts on $\Msp(\tilde{Q})$ via the rescaling action
\[
(z_1,z_2)\cdot\rho(b)=\begin{cases} z_1\rho(b) &\textrm{if }b\in Q_1\\ z_2\rho(b)&\textrm{if }b^*\in Q_1\\ (z_1z_2)^{-1}\rho(b)&\textrm{if }\exists i\in Q_0\textrm{ such that }b=\omega_i.\end{cases}
\]
This action preserves $\Tr(\tilde{W})$, so that $\DTS_{\tilde{Q},\tilde{W},\dd}$ lifts to a $T$-equivariant mixed Hodge module on $\Msp(\tilde{Q})_{\dd}$, and so $\DTS_{\Pi_Q,\dd}$ lifts to a $T$-equivariant MHM on $\Msp(\overline{Q})_{\dd}$.  By Theorem \ref{2dbpsthm} the MHM $\DTS_{\Pi_Q,\dd}$ is Verdier self-dual, so that there is an isomorphism
\[
\HO_c\!\left(\Msp(\overline{Q})_{\dd},\DTS_{\Pi_Q,\dd}\right)\cong \HO\!\left(\Msp(\overline{Q})_{\dd},\DTS_{\Pi_Q,\dd}\right)^{\vee}.
\]
Since $T$ contracts $\Msp(\overline{Q})_{\dd}$ to the point $\Msp(\overline{Q})^{\N}_{\dd}$, which is a proper variety, there are isomorphisms
\begin{align*}
\HO\!\left(\Msp(\overline{Q})_{\dd},\DTS_{\Pi_Q,\dd}\right)\cong &\HO\!\left(\Msp(\overline{Q})^{\N}_{\dd},\DTS_{\Pi_Q,\dd}\right)
\\
\cong&\HO_c\!\left(\Msp(\overline{Q})_{\dd}^{\N},\DTS_{\Pi_Q,\dd}\right).
\end{align*}
Combining these isomorphisms gives the required isomorphism \eqref{kpip}.  The identity \eqref{kpi} then follows from the definition \eqref{kpip}.
\end{proof}
\begin{remark}
\label{kacfin}
Combining Theorems \ref{KP} and \ref{mpos} with Proposition \ref{ST} we conclude that all of the Kac polynomials $\kac_{Q,\dd}(q),\kac_{Q,\dd}^{\N}(q),\kac_{Q,\dd}^{\SN}(q)$ and $\kac_{Q,\dd}^{\SSN}(q)$ have positive coefficients.
\end{remark}

\section{Deformations of Hall algebras}
\label{CoHASec}
\subsection{Kontsevich--Soibelman CoHAs}
\label{KSCSec}
In \cite{COHA}, a method was given for associating a cohomological Hall algebra (CoHA for short) to the data of an arbitrary QP $(Q,W)$.  The construction provides a mathematically rigorous approach to defining algebras of BPS states --- see \cite{HM98} for the physical motivation.  We will work with a slight generalisation of the original definition, denoted $\Coha_{\tau,Q,W}$, incorporating extra parameters depending on a weight function $\tau$.
\begin{definition}
If $(Q,W)$ is a QP, a $W$-invariant grading for $Q$ is a function $\tau\colon Q_1\rightarrow \mathbb{Z}^s$ such that every cyclic word appearing in $W$ is homogeneous of weight zero.
\end{definition}

\begin{example}\label{KST}
For $s=0$, the function $\tau=0\colon Q_1\rightarrow\mathbb{Z}^0$ gives a $W$-invariant grading for any potential $W$, and we will recover below the original definition of Kontsevich and Soibelman, by considering this grading.
\end{example}

\begin{example}\label{SVT}
For a quiver $Q$ and $s=2$, the weight function
\begin{align*}
\tau(a)=(1,0)&\textrm{ for all }a\in Q_1\\
\tau(a^*)=(0,1)&\textrm{ for all }a\in Q_1\\
\tau(\omega_i)=(-1,-1)&\textrm{ for all }i\in Q_0
\end{align*}
is a $\tilde{W}$-invariant grading for the tripled quiver $\tilde{Q}$. % For this weighting we will recover below a CoHA that is isomorphic (by \cite[Cor.4.5]{DRS15}) to the CoHA considered by Schiffmann and Vasserot in \cite[Sec.4]{ScVa12} in their work on the conjectures of Alday, Gaiotto and Tachikawa.
\end{example}
From now on we will only consider the case in which our quiver with potential is $(\tilde{Q},\tilde{W})$ for some quiver $Q$.  Since $\tilde{Q}$ is symmetric we avoid some troublesome Tate twists, and since the potential $\tilde{W}$ is linear in the $\omega_i$ direction we also avoid the notion of monodromic mixed Hodge modules; in this section we deal only with monodromic mixed Hodge modules for which Theorem \ref{dimRedThm} applies, and so we use the usual mixed Hodge module complex  $\phin{\mathfrak{Tr}(\tilde{W})}\mathbb{Q}$, as opposed to the monodromic mixed Hodge module complex $\phim{\mathfrak{Tr}(\tilde{W})}\mathbb{Q}$ (see Remark \ref{getOut}).  
\subsubsection{}
Given a grading $\tau\colon Q_1\rightarrow \mathbb{Z}^s$, define $T^{\tau}:=\Hom(\mathbb{Z}^s,\mathbb{C}^*)$.  Given a dimension vector $\dd\in\mathbb{N}^{Q_0}$ we form the extended gauge group
\[
\Gl^{\tau}_{\dd}:=\Gl_{\dd}\times T^{\tau}.
\]
The group $\Gl^{\tau}_{\dd}$ acts on $X(Q)_{\dd}$ via 
\begin{equation}
\label{exAct}
\left((\{g_i\}_{i\in Q_0},\upsilon)\cdot \rho\right)(a)=\upsilon\left(\tau(a)\right)g_{t(a)}\rho(a)g^{-1}_{s(a)}
\end{equation}
extending the action of $\Gl_{\dd}$ on $X(Q)_{\dd}$.  Similarly, if $\dd',\dd''\in\mathbb{N}^{Q_0}$ we define
\[
\Gl^{\tau}_{\dd',\dd''}:=\Gl_{\dd',\dd''}\times T^{\tau},
\]
the parabolic gauge group, acting on $X(Q)_{\dd',\dd''}$ via the same formula as (\ref{exAct}), and
\[
\Gl^{\tau}_{\dd'\times \dd''}:=\Gl_{\dd'}\times \Gl_{\dd''}\times T^{\tau}
\]
acting on $X(Q)_{\dd'}\times X(Q)_{\dd''}$ via 
\[
\left((\{g'_i\}_{i\in Q_0},\{g''_i\}_{i\in Q_0},\upsilon)\cdot (\rho',\rho'')\right)(a)=\upsilon(\tau(a))\left(g'_{t(a)}\rho'(a)g'^{-1}_{s(a)},g''_{t(a)}\rho''(a)g''^{-1}_{s(a)}\right).
\]
\smallbreak
For fixed $\upsilon\in T^{\tau}$, the action of $\upsilon$ on the category of $\mathbb{C} Q$-modules is functorial, and preserves dimension vectors.  It follows that if $\zeta\in\mathbb{H}_+^{Q_0}$ is a stability condition, the spaces $X(Q)^{\zeta\sst}_{\dd}$ and $X(Q)^{\zeta\st}_{\dd}$ are preserved by $\Gl^{\tau}_{\dd}$.  We define the stack
\[
{}^{\tau}\Mst(Q)^{\zeta\sst}_{\dd}=X(Q)^{\zeta\sst}_{\dd}/\Gl^{\tau}_{\dd}.
\]
\smallbreak
For the rest of the section we will only consider the degenerate stability condition $\zeta=(i,\ldots,i)$ and so we drop $\zeta$ from our notation, as per Convention \ref{stabConv}.  We denote by 
\[
\Dim^\tau\colon {}^{\tau}\Mst(\tilde{Q})\rightarrow\mathbb{N}^{Q_0}
\]
the map taking a $\tilde{Q}$-representation to its dimension vector.  

\smallbreak
Assume that the grading $\tau\colon \tilde{Q}_1\rightarrow\mathbb{Z}^s$ is $\tilde{W}$-invariant.  The function $\Tr(\tilde{W})$ induces a function $\mathfrak{Tr}(\tilde{W})$ on ${}^{\tau}\Mst(\tilde{Q})$.  Let $\mathcal{S}$ be a Serre subcategory of the category of $\mathbb{C}\tilde{Q}$-modules, which we assume to be invariant under the action of $T^{\tau}$, with induced morphism
\[
\iota\colon {}^{\tau}\Mst(\tilde{Q})^{\mathcal{S}}\hookrightarrow {}^{\tau}\Mst(\tilde{Q}).
\]
We define
\begin{align*}
\Coha_{\tau,\tilde{Q},\tilde{W}}\coloneqq&\Dim^{\tau}_*\iota^!\phin{\mathfrak{Tr}(\tilde{W})}\ICS_{{}^{\tau}\Mst(Q)}(\mathbb{Q})\otimes\LL^{-\dim(T^{\tau})/2}\in\Du(\MHM(\mathbb{N}^{Q_0})),
\end{align*}
the underlying cohomologically graded mixed Hodge module, equivalently, $\mathbb{N}^{Q_0}$-graded mixed Hodge structure, of $\Coha_{\tau,\tilde{Q},\tilde{W}}$.  
%The superscript $\vee$ means, as ever, that we take the dual mixed Hodge module.  Since the base is a disjoint union of points, we may alternatively define this mixed Hodge module as the Verdier dual:
%\begin{equation}
%\Coha_{\tau,\tilde{Q},\tilde{W}}:=\mathbb{D}_{\mathbb{N}^{Q_0}}\Ho(\Dim^{\tau}_!\phi_{{}^\tau\mathfrak{Tr}(\tilde{W})}\ICS_{{}^{\tau}\Mst(Q)}(\mathbb{Q})\otimes\LL^{\otimes -\dim(T^{\tau})/2}).
%\end{equation}
\begin{remark}
The reason for the peculiar twist in the definition of $\Coha_{\tau,\tilde{Q},\tilde{W}}$ is that we think of this algebra as being a version of the Kontsevich--Soibelman cohomological Hall algebra with extra parameters.  So, given that the correct sheaf for the cohomological Hall algebra is perverse, the correct sheaf for this extended version should be a family of perverse sheaves on the fibres of the projection 
\[
{}^{\tau}\Mst(\tilde{Q})_{\dd}\rightarrow \pt/T^{\tau}, 
\]
as opposed to a perverse sheaf on ${}^{\tau}\Mst(\tilde{Q})_{\dd}$ itself.  Note that the overall Tate twist in the definition of $\ICS_{{}^{\tau}\Mst(Q)}(\mathbb{Q})\otimes\LL^{ -\dim(T^{\tau})/2}$ is a full (not half) Tate twist.
\end{remark}
\subsubsection{}
We endow $\Coha_{\tau,\tilde{Q},\tilde{W}}$ with the structure of an algebra object in the category of complexes of mixed Hodge modules on $\mathbb{N}^{Q_0}$.  In order to achieve this, as in \S \ref{pfs}, a little care has to be taken to approximate morphisms of stacks by morphisms of varieties, so that we can apply Saito's theory of mixed Hodge modules to these morphisms.  We spell this out in detail.  
\smallbreak
We define 
\begin{align*}
V_{\dd,N}=&\left(\bigoplus_{i\in Q_0}\Hom(\mathbb{C}^N, \mathbb{C}^{\dd_i})\right)\\
V_{\tau,\dd,N}=&\left(\bigoplus_{i\in Q_0}\Hom(\mathbb{C}^N, \mathbb{C}^{\dd_i})\right)\oplus \Hom(\mathbb{C}^N,\mathfrak{t}^{\tau}).
\end{align*}
We let $\Gl^{\tau}_{\dd}$ act on $V_{\tau,\dd,N}$ via the product of the natural action of $\Gl_{\dd}$ on the first component, and the action of $T^{\tau}$ on $\mathfrak{t}^{\tau}$ given by the embedding $(\mathbb{C}^*)^s\subset\mathbb{C}^{s}=\mathfrak{t}^{\tau}$, and componentwise multiplication.  We define $U_{\tau,\dd,N}\subset V_{\tau,\dd,N}$ to be the subset consisting of those $(\{g_i\}_{i\in Q_0},f)\in V_{\tau,\dd,N}$ such that each $g_i$ is surjective, and $f$ is too.  Then $\Gl^{\tau}_{\dd}$ acts freely on $U_{\tau,\dd,N}$.
\smallbreak
We break the multiplication into two parts.  Fix a pair of dimension vectors $\dd',\dd''$ and set $\dd=\dd'+\dd''$.  We write
\[
\Gl_{\dd'\times \dd''}:=\Gl_{\dd'}\times \Gl_{\dd''}.  
\]
We embed $\Gl_{\dd'\times \dd''}$ and $\Gl_{\dd',\dd''}$ into $\Gl_{\dd}$ as a $Q_0$-indexed product of Levi or parabolic subgroups, respectively.   We define $\Gl_{\dd}^{\tau},\Gl_{\dd',\dd''}^{\tau}$ and $\Gl_{\dd'\times\dd''}^{\tau}$, as the product of $T^{\tau}$ with $\Gl_{\dd},\Gl_{\dd',\dd''}$ and $\Gl_{\dd'\times\dd''}$, respectively.  
\sbrk
For $G$ an algebraic group with a fixed embedding $G\subset \Gl^{\tau}_{\dd}$, we define the functor on $G$-equivariant varieties $X$
\begin{align*}
A_{N}(X,G)\coloneqq &X\times_G U_{\tau,\dd,N}.
\end{align*}
If $f\colon X\rightarrow Y$ is a $G$-invariant morphism, we denote by $f_N\colon A_N(X)\rightarrow Y$ the induced morphism.  For $\iota \colon Y\hookrightarrow X$ a $G$-invariant subvariety, then as discussed in Section \ref{pfs}, for fixed $i$ the mixed Hodge structure $\Ho^i\!\left((Y/G\rightarrow \pt)_*\iota^!\phi_f\mathbb{Q}_{X/G}\right)$ is defined as 
\[
\Ho^i\!\left((A_N(Y,G)\rightarrow\pt)_*A_N(\iota,G)^!\phi_{f_N}\mathbb{Q}_{A_N(X,G)}\right),
\]
for $N\gg 0$ depending on $i$.  Consider the commutative diagram
\begin{equation}
\label{CoHACorr}
\xymatrix{
&A_N(X(\tilde{Q})_{\dd',\dd''},\Gl^{\tau}_{\dd'\times \dd''})\ar[d]_{q_1}\ar[ld]_-{q_2}\\
A_N(X(\tilde{Q})_{\dd'}\times X(\tilde{Q})_{\dd''},\Gl^{\tau}_{\dd'\times \dd''})\ar[d]_{(\Dim^{\tau}\times \Dim^{\tau})_N}&A_N(X(\tilde{Q})_{\dd',\dd''},\Gl^{\tau}_{\dd',\dd''})\ar[d]_{\Dim^{\tau,\circ}_N}\\
\mathbb{N}^{Q_0}\times \mathbb{N}^{Q_0}\ar[r]^+&\mathbb{N}^{Q_0},
}
\end{equation}
where $q_1$ and $q_2$ are the natural affine fibrations, inducing isomorphisms
\begin{align*}
\alpha_{\dd',\dd''}\colon &+_*(\Dim^{\tau}\times\Dim^{\tau})_*\left(\iota^!\phi_{\mathfrak{Tr}(\tilde{W})\boxtimes \mathfrak{Tr}(\tilde{W})}\ICS_{{}^{\tau}\Mst(\tilde{Q})_{\dd'}\times_{\B T^{\tau}}\Mst(\tilde{Q})_{\dd''}}(\mathbb{Q})\right)\\\rightarrow &\Dim^{\tau,\circ}_*\iota^!\phin{\mathfrak{Tr}(\tilde{W})}\ICS_{{}^{\tau}\Mst(\tilde{Q})_{\dd',\dd''}}(\mathbb{Q})\otimes\mathbb{L}^{-(\dd',\dd'')_{\tilde{Q}}/2}.
\end{align*}
Consider the composition of proper maps
\begin{equation}
\label{propcomp}
A_N(X(\tilde{Q})_{\dd',\dd''},\Gl^{\tau}_{\dd',\dd''})\xrightarrow{r_N} A_N(X(\tilde{Q})_{\dd},\Gl^{\tau}_{\dd',\dd''})\xrightarrow{s_N} A_N(X(\tilde{Q})_{\dd},\Gl^{\tau}_{\dd})
\end{equation}
where $r_N$ is induced by the inclusion $X(\tilde{Q})_{\dd',\dd''}\hookrightarrow X(\tilde{Q})_{\dd}$ and $s_N$ is induced by the inclusion $\Gl^{\tau}_{\dd',\dd''}\hookrightarrow\Gl^{\tau}_{\dd}$.  Since $r_N$ and $s_N$ are proper, there is a natural morphism
\begin{equation}
\label{ProRes}
s_{N,*}r_{N,*}\QQ_{A_N(X(\tilde{Q})_{\dd',\dd''},\Gl^{\tau}_{\dd',\dd''})}\otimes \LL^{-(\dd',\dd'')_{\tilde{Q}}}\rightarrow \QQ_{A_N(X(\tilde{Q})_{\dd},\Gl^{\tau}_{\dd})}.
\end{equation}
Applying $\Dim^{\tau}_{N,*}\phin{\Tr(\tilde{W})}$ and letting $N\mapsto\infty$, the morphism (\ref{ProRes}) induces the morphism
\[
\beta_{\dd',\dd''}\colon \Dim^{\tau,\circ}_*\iota^!\phi_{\mathfrak{Tr}(\tilde{W})}\ICS_{{}^{\tau}\Mst(\tilde{Q})_{\dd',\dd''}}(\mathbb{Q})\otimes\mathbb{L}^{-(\dd',\dd'')_{\tilde{Q}}/2}\rightarrow \Dim^{\tau}_*\iota^!\phi_{\mathfrak{Tr}(\tilde{W})}\ICS_{{}^{\tau}\Mst(\tilde{Q})_{\dd}}(\mathbb{Q}).
\]
Defining $m_{\dd',\dd''}= (\beta_{\dd',\dd''}\otimes \LL^{-\dim(T^{\tau})/2})\circ (\alpha_{\dd',\dd''}\otimes\LL^{-\dim(T^{\tau})/2})\circ \TS$, where $\TS$ is the Thom--Sebastiani isomorphism, gives the multiplication 
\[
m\colon \Coha^{\mathcal{S}}_{\tau,\tilde{Q},\tilde{W}}\otimes_{\HO_{T^{\tau}}} \Coha^{\Sp}_{\tau,\tilde{Q},\tilde{W}}\rightarrow \Coha^{\Sp}_{\tau,\tilde{Q},\tilde{W}}.
\]
We write $\Coha^{\Sp}_{\tilde{Q},\tilde{W}}$ for the special case in which $T^{\tau}$ is the zero-dimensional torus (as in Example \ref{KST}).  In this case, the above multiplication is exactly the multiplication defined by Kontsevich and Soibelman in \cite{COHA}.  The proof that for general $T^{\tau}$ the multiplication is associative is standard, and is in particular unchanged from the proof given in \cite[Sec.7]{COHA}, to which we refer for fuller details.
%First we define\footnote{By self duality of $\phi_{\mathfrak{Tr}}(\tilde{W})$ and $\ICS_{\Mst(\tilde{Q})}(\mathbb{Q})$ under Verdier duality, we could alternatively make the definition $\Coha_{\tilde{Q},\tilde{W}}=\Ho\left(\Dim_*\phi_{\mathfrak{Tr}(\tilde{W})}\ICS_{\Mst(\tilde{Q})}(\mathbb{Q})\right)$.}
%\[
%\Coha_{\tilde{Q},\tilde{W}}:=\Ho\left(\Dim_!\phi_{\mathfrak{Tr}(\tilde{W})}\ICS_{\Mst(\tilde{Q})}(\mathbb{Q})\right)^{\vee}
%\]
%which we consider as a $\mathbb{N}^{Q_0}$-graded cohomologically graded mixed Hodge structure in the usual way.  As ever, the superscript $\vee$ means that we consider the dual cohomologically graded mixed Hodge module, and so since $\Ho\left(\Dim_!\phi_{\mathfrak{Tr}(\tilde{W})}\ICS_{\Mst(\tilde{Q})}(\mathbb{Q})\right)\in\Dl(\MHM(\mathbb{N}^{Q_0}))$, we see that $\Coha_{\tilde{Q},\tilde{W}}\in\Du(\MHM(\mathbb{N}^{Q_0}))$.
%\medbreak

\subsection{The degeneration result}
The extra equivariant parameters arising from the torus action on $\Mst(\tilde{Q})$ are not considered in the original paper \cite{COHA}, but were introduced, for the particular cohomological Hall algebras we are considering, in \cite{DRS15} and \cite{YaZh14,YaZh16}.  In general, such extra parameters are of most interest when they provide a geometric deformation of the original algebra, i.e. when they provide a flat family of algebras over $\Spec(\HO_{T})$, such that the specialization at the central fibre is our original algebra, which in this case is $\Coha_{\tilde{Q},\tilde{W}}$.  For $T^{\tau}$ the torus associated to a $\tilde{W}$-invariant grading of $\tilde{Q}$ this is precisely the result we prove in this section.  
%\begin{remark}
%Pushing these ideas a little further, we can show that for general $Q$, the deformed preprojective cohomological Hall algebra considered by Schiffmann and Vasserot in \cite{ScVa12} in the special case of the Jordan quiver injects into the $\mathbb{Q}[t_1,t_2]$-deformed shuffle algebra, confirming \cite[Conj.4.4]{ScVa12} --- since this takes us a little bit too far from the central thread of this paper, we leave the details to a sequel paper.
%\end{remark}

\smallbreak
 Let $\tau\colon \tilde{Q}_1\rightarrow \mathbb{Z}^s$ be a $\tilde{W}$-invariant grading, with associated torus $T$.  Let $\upsilon\colon \mathbb{Z}^s\rightarrow \mathbb{Z}^{s'}$ be a surjective morphism of groups, inducing the inclusion of tori $T'\hookrightarrow T$, where $T'$ is the torus associated to the $\tilde{W}$-invariant grading $\tau'=\upsilon\circ\tau$.  Write
\[
s''=s-s'.
\]
Then picking a splitting of $\upsilon$, i.e. an extension of $\upsilon$ to an isomorphism $\mathbb{Z}^s\rightarrow \mathbb{Z}^{s'}\oplus\mathbb{Z}^{s''}$, induces an isomorphism 
\begin{equation}
\label{Tchidef}
\HO_{T}\cong \HO_{T'}\otimes \HO_{T^{\chi}},
\end{equation}
where $\chi\colon\tilde{Q}_1\rightarrow \mathbb{Z}^{s''}$ is induced by the splitting.  The splitting of $\upsilon$ induces a splitting $\mathfrak{t}\cong \mathfrak{t}'\oplus \mathfrak{t}^{\chi}$.

\smallbreak
We define 
\[
Y_{\tau,\dd,N}:=X(\tilde{Q})_{\dd}\times_{\Gl^{\tau}_\dd}U_{\tau,\dd,N}.
\]
We consider the natural maps 
\[
v_{\dd,N}\colon Y_{\tau,\dd,N} \rightarrow \Hom^{\surj}(\mathbb{C}^N,\mathfrak{t}^{\chi})/T^{\chi}\eqqcolon S_{\chi,N}
\]
defined by the morphism
\begin{align*}
&\Hom^{\surj}(\mathbb{C}^N,\mathfrak{t})\rightarrow \Hom^{\surj}(\mathbb{C}^N,\mathfrak{t}^{\chi})\\
&f\mapsto \pi_{\mathfrak{t}^{\chi}}\circ f.
\end{align*}
The function $\Tr(\tilde{W})$ induces functions
\[
\Tr(\tilde{W})_{\tau,\dd,N}\colon Y_{\tau,\dd,N}\rightarrow\mathbb{C}.
\]
\begin{lemma}\label{Sgoo}
The space $S_{\chi,N}$ is $s''(N-1)$-dimensional, simply connected, and $\HO(S_{\chi,N},\mathbb{Q})$ is pure.
\end{lemma}
\begin{proof}
By choosing a splitting $\mathfrak{t}^{\chi}=\mathbb{C}^{\oplus s''}$ and considering the entries of a morphism $f\in S_{\chi,N}$ one by one, we obtain a sequence of morphisms
\[
S_{\chi,N}=H_{s''}\xrightarrow{l_{s''-1}}H_{s''-1}\xrightarrow{l_{s''-2}}\ldots\xrightarrow{l_{1}}H_1=\mathbb{P}^{N-1}
\]
where $l_e$ is a $(\mathbb{A}^{N}\setminus \mathbb{A}^e)/\mathbb{C}^*$-fibration, with $\mathbb{C}^*$ acting on $\AA^N$ via scaling.  All the claims follow from this description.
\end{proof}
Each of the maps $v_{\dd,N}$ is a fibre bundle with fibre $Y_{\tau',\dd,N}$.  Picking 
\[
i\colon \Upsilon\hookrightarrow S_{\chi,N}
\]
the inclusion of a sufficiently small open ball (in the analytic topology) contained in the base, we may write 
\[
\Tr(\tilde{W})_{\tau,\dd,N}|_{v_{\dd,N}^{-1}(\Upsilon)}\colon \Upsilon\times_{v_{\dd,N}} Y_{\tau,\dd,N} \cong \Upsilon\times Y_{\tau',\dd,N}\rightarrow \mathbb{C}
\]
as $\Tr(\tilde{W})_{\tau',\dd,N}\circ\pi$ where 
\[
\pi\colon v^{-1}_{\dd,N}(\Upsilon)\rightarrow Y_{\tau',\dd,N}
\]
is the projection, and so we deduce that the mixed Hodge modules
\[
\Ho^q(v_{\dd,N,*}\phi_{\Tr(\tilde{W})_{\tau,\dd,N}}\mathbb{Q}_{Y_{\tau,\dd,N}})
\]
are locally trivial in the analytic topology, with fibre given by $\HO^q(Y_{\tau',\dd,N},\phi_{\Tr(\tilde{W})_{\tau',\dd,N}}\mathbb{Q}_{Y_{\tau',\dd,N}})$, and are furthermore globally trivial by the rigidity theorem \cite[Thm.4.20]{StZu85}, since the base of $v_{\dd,N}$ is simply connected.  
\smallbreak
As a result, the Leray spectral sequence $E_{\upsilon,\dd,N,\bullet}^{\bullet,\bullet}$ converging to $\HO(Y_{\tau,\dd,N},\phi_{\Tr(\tilde{W})_{\tau,\dd,N}}\ICS_{Y_{\tau,\dd,N}}(\mathbb{Q})\otimes\mathbb{L}^{(\dim(V_{\tau,\dd,N})-s)/2})$ satisfies
\begin{align}
\label{E2Def}
E_{\upsilon,\dd,N,2}^{p,q}=&\HO^p\!\left(S_{\chi,N},\mathbb{Q}\right)\otimes\HO^q\!\left(Y_{\tau',\dd,N},\phi_{\Tr(\tilde{W})_{\tau',\dd,N}}\ICS_{Y_{\tau',\dd,N}}(\mathbb{Q})\otimes\mathbb{L}^{(\dim(V_{\tau',\dd,N})-s')/2}\right).
\end{align}
%\begin{remark}
%Some care has to be taken here in keeping track of the total Tate twists.  If we were using the Leray spectral sequence instead to calculate $\HO_c(Y_{\tau,\dd,N},\phi_{\Tr(\tilde{W})_{\tau,\dd,N}}\mathbb{Q}_{Y_{\tau,\dd,N}}(\mathbb{Q}))$, which is
%\[
%\HO_c(Y_{\tau,\dd,N},\phi_{\Tr(\tilde{W})_{\tau,\dd,N}}\ICS_{Y_{\tau,\dd,N}}(\mathbb{Q})\otimes\mathbb{L}^{-(\dim(V_{\tau,\dd,N})+s)/2})\otimes\LL^{\otimes \xi}
%\]
%where 
%\begin{align*}
%\xi=&(\dim(X(\tilde{Q})_{\dd})-\dim(\Gl_{\dd}^{\tau}))/2+\dim(V_{\tau,\dd,N})+s/2\\
%=&(\dim(X(\tilde{Q})_{\dd})-\left(\dim(\Gl_{\dd})+s\right))/2+\left(\dim(V_{\dd,N})+N\cdot s\right)+s/2\\
%=&(\dim(X(\tilde{Q})_{\dd})-\dim(\Gl_{\dd}))/2+\dim(V_{\dd,N})+N\cdot s
%\end{align*}
%we would obtain the $E^{\bullet,\bullet}_2$ sheet with terms
%\begin{align}
%\label{E2clean}
%\HO^p\left(\mathbb{C}\mathbb{P}^{N-1},\mathbb{Q}\right)^{\otimes s''}\otimes\HO_c^q\left(Y_{\tau',\dd,N},\phi_{\Tr(\tilde{W})_{\tau',\dd,N}}\mathbb{Q}_{Y_{\tau',\dd,N}}\right)
%\end{align}
%instead of (\ref{E2Def}).  On the other hand, $(\ref{E2clean})\cong(\ref{E2Def})\otimes\LL^{\otimes \xi'}$, where
%\begin{align*}
%\xi'=&N\cdot s''+\left(\dim(X(\tilde{Q})_{\dd})-\dim(\Gl_{\dd}^{\tau'})\right)/2+\dim(V_{\tau',\dd,N})+s'/2\\
%=&N\cdot s''+\left(\dim(X(\tilde{Q})_{\dd})-\left(\dim(\Gl_{\dd})+s'\right)\right)/2+(\dim(V_{\dd,N})+N\cdot s')+s'/2\\
%=&N\cdot s''+\left(\dim(X(\tilde{Q})_{\dd})-\dim(\Gl_{\dd})\right)/2+\dim(V_{\dd,N})+N\cdot s'\\
%=&\xi
%\end{align*}
%as required.
%\end{remark}
Set
\begin{align*}
\heartsuit^{(')}=&(\dim(U_{\tau^{(')},\dd,N}-s^{(')}N-(\dd',\dd'')_{\tilde{Q}})/2\\
\spadesuit^{(')}=&(\dim(U_{\tau^{(')},\dd,N}-s^{(')}N)/2.
\end{align*}
In similar fashion, we obtain spectral sequences $E^{\bullet,\bullet}_{\upsilon,N,\dd',\dd'',\bullet}$ and $E^{\bullet,\bullet}_{\upsilon,N,\dd'\times \dd'',\bullet}$ satisfying
\begin{align*}
E^{p,q}_{\upsilon,\dd',\dd'',N,2}=& \HO^p\!\left(S_{\chi,N},\mathbb{Q}\right)\otimes\HO^q\!\left(Y_{\tau',\dd',\dd'',N},\:\phi_{\Tr(\tilde{W})_{\tau',\dd',\dd'',N}}\ICS_{Y_{\tau',\dd',\dd'',N}}(\mathbb{Q})\otimes\mathbb{L}^{ \heartsuit'}\right)\\
E^{p,q}_{\upsilon,\dd'\times \dd'',N,2}=&\HO^p\!\left(S_{\chi,N},\mathbb{Q}\right)\otimes \HO^q\!\left( Y_{\tau',\dd'\times \dd'',N},\:\phi_{\Tr(\tilde{W})_{\tau',\dd'\times \dd'',N}}\ICS_{Y_{\tau',\dd'\times \dd'',N}}(\mathbb{Q})\otimes\mathbb{L}^{\spadesuit'}\right)
\end{align*}
converging to 
\[
\HO(Y_{\tau,\dd',\dd'',N},\:\phi_{\Tr(\tilde{W})_{\tau,\dd',\dd'',N}}\ICS_{Y_{\tau,\dd',\dd'',N}}(\mathbb{Q})\otimes\mathbb{L}^{\heartsuit})
\]
and 
\[
\HO(Y_{\tau,\dd'\times\dd'',N},\:\phi_{\Tr(\tilde{W})_{\tau,\dd'\times\dd'',N}}\ICS_{Y_{\tau,\dd'\times\dd'',N}}(\mathbb{Q})\otimes\mathbb{L}^{\spadesuit})
\]
respectively.  As in the construction of $\Coha_{\tau,\tilde{Q},\tilde{W}}$ we obtain a commutative diagram of morphisms of spectral sequences, with vertical morphisms provided by restriction morphisms in cohomology
\begin{equation}
\label{specMult}
\xymatrix{
E^{\bullet,\bullet}_{\upsilon,\dd,N+1,\bullet}\ar[d] &E^{\bullet,\bullet}_{\upsilon,\dd',\dd'',N+1,\bullet}\ar[l]\ar[d]&E^{\bullet,\bullet}_{\upsilon,\dd'\times \dd'',N+1,\bullet}\ar[d]\ar[l]
\\
E^{\bullet,\bullet}_{\upsilon,\dd,N,\bullet} &E^{\bullet,\bullet}_{\upsilon,\dd',\dd'',N,\bullet}\ar[l]&E^{\bullet,\bullet}_{\upsilon,\dd'\times \dd'',N,\bullet}.\ar[l]
}
\end{equation}
Each of the spectral sequences $E^{\bullet,\bullet}_{\upsilon,\dd,N,\bullet}$, $E^{\bullet,\bullet}_{\upsilon,\dd',\dd'',N,\bullet}$, $E^{\bullet,\bullet}_{\upsilon,\dd'\times \dd'',N,\bullet}$ is a first quadrant spectral sequence, and each of the limits 
\begin{align*}
&\lim_{N\mapsto \infty}E^{p,q}_{\upsilon,\dd,N,2}\\
&\lim_{N\mapsto \infty}E^{p,q}_{\upsilon,\dd',\dd'',N,2}\\
&\lim_{N\mapsto \infty}E^{p,q}_{\upsilon,\dd'\times \dd'',N,2}
\end{align*}
exists as in Section \ref{pfs}.  We claim the following commutativity of limits
\begin{align}\label{commss}
\mathcal{A}_{\tau,\tilde{Q},\tilde{W},\dd}\cong&\lim_{N\mapsto\infty}\lim_{l\mapsto \infty}E^{p,q}_{\upsilon,\dd,N,l}&\cong \lim_{l\mapsto\infty}\lim_{N\mapsto \infty}E^{p,q}_{\upsilon,\dd,N,l}
\\\nonumber
\mathcal{A}_{\tau,\tilde{Q},\tilde{W},\dd',\dd''}\cong&\lim_{N\mapsto\infty}\lim_{l\mapsto \infty}E^{p,q}_{\upsilon,\dd',\dd'',N,l}&\cong \lim_{l\mapsto\infty}\lim_{N\mapsto \infty}E^{p,q}_{\upsilon,\dd',\dd'',N,l}
\\\nonumber
\mathcal{A}_{\tau,\tilde{Q},\tilde{W},\dd'}\otimes_{\HO_T} \mathcal{A}_{\tau,\tilde{Q},\tilde{W},\dd''}\cong&\lim_{N\mapsto\infty}\lim_{l\mapsto \infty}E^{p,q}_{\upsilon,\dd'\times \dd'',N,l}&\cong\lim_{l\mapsto\infty}\lim_{N\mapsto \infty}E^{p,q}_{\upsilon,\dd'\times \dd'',N,l},
\end{align}
using the shorthand 
\[
\mathcal{A}_{\tau,\tilde{Q},\tilde{W},\dd',\dd''}\cong \mathcal{A}_{\tau,\tilde{Q},\tilde{W},\dd'}\otimes_{\HO_T} \mathcal{A}_{\tau,\tilde{Q},\tilde{W},\dd''}\otimes\mathbb{L}^{{}-(\dd',\dd'')_{\tilde{Q}}/2}.
\]
The argument for all three statements is the same: fixing $p$ and $q$, the limit $E^{p,q}_{\upsilon,\dd,N,\infty}$ depends only on a finite portion of $E^{p,q}_{\upsilon,\dd,N,s}$, which therefore stabilises for sufficiently large $N=N_{p,q}$.  The $(p,q)$-term of both the second and third expression of (\ref{commss}) are then given by $E^{p,q}_{\upsilon,\dd,N_{p,q},\infty}$.  
\sbrk
We may define the cohomological Hall algebra multiplication on $\Coha_{\tau,\tilde{Q},\tilde{W}}$ via the commutative diagram obtained from (\ref{CoHACorr}) or as the morphism induced in the double limit by the composition of the horizontal morphisms in (\ref{specMult}).  Via the morphism
\[
E^{\bullet,\bullet}_{\upsilon,\dd,\infty,2}\rightarrow E^{0,\bullet}_{\dd,\infty,2}
\]
to the degenerate spectral sequence concentrated on the first nontrivial column, and the analogous morphisms for the spectral sequences $E^{\bullet,\bullet}_{\upsilon,\dd',\dd'',\infty,\bullet}$ and $E^{\bullet,\bullet}_{\upsilon,\dd'\times \dd'',\infty,\bullet}$ we obtain a commutative diagram of double limits
\[
\xymatrix{
\lim_{l\mapsto \infty} E^{\bullet,\bullet}_{\upsilon,\dd,\infty,l}\ar[r] \ar[d]&\lim_{l\mapsto \infty}E^{\bullet,\bullet}_{\upsilon,\dd',\dd'',\infty,l}\ar[r]\ar[d]&\lim_{l\mapsto \infty} E^{\bullet,\bullet}_{\upsilon,\dd'\times \dd'',\infty,l}\ar[d]\\
\Coha_{\tau',\tilde{Q},\tilde{W},\dd}\ar[r]& \Coha_{\tau',\tilde{Q},\tilde{W},\dd',\dd''}\ar[r]&\Coha_{\tau',\tilde{Q},\tilde{W},\dd'}\otimes_{\HO_T'} \Coha_{\tau',\tilde{Q},\tilde{W},\dd''}
}
\]
providing a lift of the natural morphism
\begin{equation}
\label{specCo}
 \Coha_{\tau,\tilde{Q},\tilde{W}}\otimes_{\HO_T} \HO_{T'}\rightarrow\Coha_{\tau',\tilde{Q},\tilde{W}}
\end{equation}
in cohomology to a morphism in the category of algebra objects in the category of complexes of mixed Hodge structures.  
\begin{theorem}\label{degThm}
Let 
\[
\xymatrix{
\tilde{Q}_1\ar@/^1pc/[rr]^-{\tau'}\ar[r]_-{\tau}&\mathbb{Z}^s\ar@{->>}[r]_-{\upsilon}&\mathbb{Z}^{s'}
}
\]
be as above a specialisation of a $\tilde{W}$-invariant weighting of $\tilde{Q}$.  Then there is an isomorphism of mixed Hodge structures
\begin{equation}
\label{degIso}
\Coha_{\tau,\tilde{Q},\tilde{W}}\cong\Coha_{\tau',\tilde{Q},\tilde{W}}\otimes_{\mathbb{Q}} \HO_{T^{\chi}}
\end{equation}
with $T^{\chi}$ as in \eqref{Tchidef}.  Furthermore, the morphism (\ref{specCo}) is an isomorphism, and both sides of (\ref{degIso}) are pure.
\end{theorem}
\begin{proof}
First we consider the special case $s'=0$, $\tau'=0$.  Then by Theorem \ref{purityThm} and Lemma \ref{Sgoo}, the right hand side of (\ref{E2Def}) is a pure Hodge structure, and so the spectral sequence $E^{\bullet,\bullet}_{\upsilon,d,\infty,\bullet}$ degenerates at the second sheet, and the existence of the isomorphism (\ref{degIso}) follows, along with the fact that \eqref{specCo} is an isomorphism.

As a consequence, $A_{\tau,\tilde{Q},\tilde{W}}$ is pure for all $\tau$.  So it follows that for general $\upsilon$, the right hand side of (\ref{E2Def}) is pure, and the general case follows via the same argument as the special case.
\end{proof}

\subsubsection{}
Let the torus $T=(\mathbb{C}^*)^s$ act on $\Mst(\tilde{Q})_{\dd}$ via the weight function $\tau\colon\tilde{Q}_1\rightarrow \mathbb{Z}^s$.  Then for each $\dd\in\mathbb{N}^{Q_0}$, ignoring the overall Tate twist, via Theorem \ref{dimRedThm} there is an isomorphism in Borel--Moore homology
\[
\Psi_{\tau,Q,\dd}\colon\HO^{\BM}\!\left({}^{\tau}\Mst(\tilde{Q})_{\dd},\phi_{\Tr(\tilde{W})}\mathbb{Q}\right)\cong \HO_{T\times \Gl_{\dd}}^{\BM}(\mu^{-1}_{Q,\dd}(0),\mathbb{Q})=:\Coha_{\tau,\Pi_Q,\dd}.
\]
The domain of $\bigoplus_{\dd\in\mathbb{N}^{Q_0}}\Psi_{\tau,Q,\dd}$ carries the Kontsevich--Soibelman cohomological Hall algebra product recalled above, while the target carries the Schiffmann--Vasserot product \cite[Sec.4]{ScVa12}.  By \cite[Cor.4.5]{DRS15} or \cite{YaZh16}, the modified morphism
\begin{equation}
\label{signedIso}
\Psi'_{\tau,Q}:=\bigoplus_{\dd\in\mathbb{N}^{Q_0}} (-1)^{\sum_{i\in Q_0}\binom{\dd_i}{2}}\Psi_{\tau,Q,\dd}\colon \Coha_{\tau,\tilde{Q},\tilde{W}}\rightarrow \Coha_{\tau,\Pi_Q}
\end{equation}
is an isomorphism of algebras.  Since $\Psi_{\tau,Q,\dd}$ is a morphism of $\HO_T$-modules, we deduce the following corollary of Theorem \ref{degThm}.
\begin{corollary}
\label{specCor}
Let $\mathfrak{m}$ be the maximal homogeneous ideal in $\HO_T$.  Then $\Coha_{\tau,\Pi_Q}$ is free as a $\HO_T$-module, and the natural morphism of algebras 
\[
\Xi\colon \Coha_{\tau,\Pi_Q}\otimes_{\HO_T}(\HO_T/\mathfrak{m})\rightarrow \bigoplus_{\dd\in\mathbb{N}^{Q_0}}\HO_{\Gl_{\dd}}^{\BM}(\mu^{-1}_{Q,\dd}(0),\mathbb{Q})
\]
is an isomorphism.
\end{corollary}
\begin{proof}
While the vertical dimension reduction isomorphisms in the following commutative diagram are not morphisms of algebras (due to sign issues), the horizontal morphisms are, and the top one is an isomorphism since the bottom one is by Theorem \ref{degThm}:
\[
\xymatrix{
\bigoplus_{\dd\in\mathbb{N}^{Q_0}}\HO_{\Gl_{\dd}}^{\BM}(\mu^{-1}_{Q,\dd}(0),\mathbb{Q})&\Coha_{\Pi_Q}\otimes_{\HO_T}(\HO_T/\mathfrak{m})\ar[l]_-{\Xi}
\\
\Coha_{\tau',\tilde{Q},\tilde{W}}\ar[u]^{\Psi_{\tau',Q}}&\Coha_{\tau,\tilde{Q},\tilde{W}}\otimes_{\HO_T}(\HO_T/\mathfrak{m})
\ar[l]\ar[u]_{\Psi_{\tau,Q}\otimes_{\HO_T}(\HO_T/\mathfrak{m})}.
}
\]
\end{proof}
\section{Shuffle algebras, torsion-freeness and noncommutativity}
\label{CoHASec2}
\subsection{Definition of the shuffle algebra}
Fix a weight function $\tau\colon \tilde{Q}_1\rightarrow \mathbb{Z}^s$.  As above we write $T=\Hom(\mathbb{Z}^s,\mathbb{C}^*)$.  We recall the shuffle algebra description of the cohomological Hall algebra 
\[
\Coha_{\tau,\tilde{Q}}=\bigoplus_{\dd\in\mathbb{N}^{Q_0}}\HO({}^{\tau}\Mst(\tilde{Q})_{\dd},\mathbb{Q})\otimes\LL^{(\dd,\dd)_{\tilde{Q}}/2}.
\]
Set $\Bbbk\coloneqq \HO_T$.  Since $X(\tilde{Q})_{\dd}$ is equivariantly contractible, there is an isomorphism in cohomology
\[
\HO_{T\times \Gl_{\dd}}(X(\tilde{Q})_{\dd},\mathbb{Q})\cong \Bbbk[x_{i,n} \;\lvert\; i\in Q_0,1\leq n\leq \dd_i]^{\mathfrak{S}_{\dd}}.
\]
Here $\mathfrak{S}_{\dd}=\prod_{i\in Q_0}\mathfrak{S}_{\dd_i}$ is the product of symmetric groups, with $\mathfrak{S}_{\dd_i}$ acting by permuting the variables $x_{i,1},\ldots,x_{i,\dd_i}$.  For $\dd'+\dd''=\dd$, we define $\Sh_{\dd',\dd''}\subset \mathfrak{S}_{\dd}$ to be the subset of permutations $(\sigma_i)_{i\in Q_0}$ such that for each $i\in Q_0$ we have inequalities
\begin{align*}
\sigma_i(1)<\sigma_i(2)<\ldots<\sigma_i(\dd'_i)&&\textrm{    and    }&&\sigma_i(\dd'_i+1)<\ldots <\sigma_i(\dd_i).
\end{align*}
We fix generators $t_1,\ldots, t_s$ of $\HO_T$, with $t_i$ corresponding to the generator of the equivariant cohomology of $\Hom(\mathbb{Z}_i,\mathbb{C}^*)$, where $\mathbb{Z}_i$ is the $i$th copy of $\mathbb{Z}$ inside $\mathbb{Z}^s$.
\smallbreak
For $a\in\tilde{Q}_1$ define 
\[
E_a(z)=z+\sum_{i\leq s}\tau(a)_it_i.
\]

We use $\star$ to denote the multiplication in the CoHA $\Coha_{\tau,\tilde{Q}}$.  Then it is shown as in \cite[Sec.1]{COHA}:
\begin{align}\label{SAP}
&f(x_{1,1},\ldots,x_{r,\dd'_r})\star g(x_{1,1},\ldots,g_{r,\dd''_r})=\\
\nonumber
&\sum_{\sigma\in\Sh_{\dd',\dd''}}\sigma\BIG{(}f(x_{1,1},\ldots,x_{r,\dd'_r})g(x_{1,\dd'_1+1},x_{1,\dd'_1+2},\ldots ,x_{1,\dd_1},x_{2,\dd'_2+1},\ldots,x_{r,\dd_r})\cdot \\
\nonumber
&\prod_{a\in \tilde{Q}_1} \left(\prod_{\substack{1\leq m\leq \dd'_{s(a)} \\\dd'_{t(a)}<n\leq \dd_{t(a)}}}E_a(x_{t(a),n}-x_{s(a),m})\right)\prod_{i\in Q_0}\left(\prod_{\substack{1\leq m\leq \dd'_i\\ \dd'_i<n\leq \dd_i}}(x_{i,n}-x_{i,m})^{-1}\right)\BIG{)}.
\end{align}
\subsubsection{}
Let $z\colon Z\hookrightarrow X(\tilde{Q})_{\dd}\eqqcolon X$ be the subvariety cut out by the matrix valued equation $\sum_{a\in Q_1}[a,a^*]=0$.  Then since $Z\subset \Tr(\tilde{W})^{-1}(0)$ there is a (dual) restriction map $z_*z^!\mathbb{Q}_X\rightarrow \phin{\Tr(\tilde{W})}\mathbb{Q}_X $, inducing the morphism $\alpha$ in the following diagram
\begin{equation}
\label{shuffDiag}
\xymatrix{
\HO^{T\times \Gl_{\dd}}(X(\tilde{Q})_{\dd},\phin{\Tr(\tilde{W})}\mathbb{Q})\otimes\LL^{(\dd,\dd)_{\tilde{Q}}/2}\ar@{.>}[dr]^{\Phi_{\dd}}
\\
\HO^{\BM}_{T\times \Gl_{\dd}}(Z,\mathbb{Q})\otimes\LL^{s-(\dd,\dd)_{\tilde{Q}}/2}\ar[r]^{\beta}\ar[d]\ar[u]^-{\alpha}&\HO^{T\times\Gl_{\dd}}(X(\tilde{Q})_{\dd},\mathbb{Q})\otimes \LL^{(\dd,\dd)_{\tilde{Q}}/2}\ar[d]^{\epsilon}
\\
\HO^{\BM}_{T\times \Gl_{\dd}}(\mu_{Q,\dd}^{-1}(0),\mathbb{Q})\otimes\LL^{s-(\dd,\dd)_Q}\ar[r]^{\gamma}&\HO^{\BM}_{T\times\Gl_{\dd}}(X(\overline{Q})_{\dd},\mathbb{Q})\otimes \LL^{s-(\dd,\dd)_{Q}}.
}
\end{equation}
The unlabelled vertical morphisms are isomorphisms because they are induced by affine fibrations, while $\alpha$ is an isomorphism by Theorem \ref{dimRedThm}.  Then we define $\Phi_{\dd}=\beta\alpha^{-1}$.  In particular, $\Phi_{\dd}$ is injective if and only if $\gamma$ is.  
\begin{proposition}
The morphism
\[
\Phi\colon \bigoplus_{\dd\in\mathbb{N}^{Q_0}}\HO^{T\times \Gl_{\dd}}(X(\tilde{Q})_{\dd},\phin{\Tr(\tilde{W})}\mathbb{Q})\otimes\LL^{(\dd,\dd)_{\tilde{Q}}/2}\rightarrow \bigoplus_{\dd\in\mathbb{N}^{Q_0}}\HO^{T\times \Gl_{\dd}}(X(\tilde{Q})_{\dd},\mathbb{Q})\otimes\LL^{( \dd,\dd)_{\tilde{Q}}/2}
\]
is an algebra morphism, where the domain and target are given the KS cohomological Hall algebra structure.
\end{proposition}
\begin{proof}
For $a\in \tilde{Q}_1$ we define
\[
E^{\tw}_a(z)=\begin{cases} E_a(z)&\textrm{if }a\neq \omega_i\;\textrm{ for }i\in Q_0\\
-E_a(z)&\textrm{if }a=\omega_i\;\textrm{ for }i\in Q_0.
\end{cases}
\]
The shuffle algebra $\Coha^{\tw}_{\tau,\tilde{Q}}$ is defined to have the same underlying graded vector space as $\Coha_{\tau,\tilde{Q}}$, with multiplication defined by \eqref{SAP}, but with all instances of $E_a(\bullet)$ replaced by $E'_a(\bullet)$.  Let 
\[
F\colon \bigoplus_{\dd\in\mathbb{N}^{Q_0}}\HO^{\BM}_{T\times \Gl_{\dd}}(\mu_{Q,\dd}^{-1}(0),\mathbb{Q})\otimes\LL^{s-(\dd,\dd)_Q}\rightarrow \Coha^{\tw}_{\tau,\tilde{Q}}
\]
be the morphism defined by taking the sum of $\epsilon^{-1}\gamma$ over all $\dd$.  Then by \cite{ScVa12}, $F$ is a homomorphism of algebras.  We define an isomorphism
\[
\Gamma\colon \Coha^{\tw}_{\tau,Q}\rightarrow \Coha_{\tau,\tilde{Q}}\\
\]
by setting $\Gamma_{\dd}=(-1)^{\sum_{i\in Q_0}\binom{\dd_i}{2}}\cdot \id_{\Coha_{\tau,\tilde{Q},\dd}}$.  We define $\Psi'_{\tau,Q}$ as in \eqref{signedIso}.  Then $\Phi=\Gamma\circ F\circ\Psi'_{\tau,Q}$ is a composition of morphisms of algebras.
\end{proof}

\subsection{Torsion freeness}
Due to the complicated geometry of preprojective stacks, the algebra $\Coha_{\tau,\Pi_Q}$ is still not wholly understood, despite intensive study.  On the other hand, in their work on the AGT conjectures \cite[Sec.4.3]{ScVa12} Schiffmann and Vasserot conjectured that for $\tau$ as defined in Example \ref{SVT}, the morphism $\Phi$ (or equivalently, the morphism $F$) is in fact an embedding of algebras, making the cohomological Hall algebra $\Coha_{\tau,\Pi_{Q_{\Jor}}}$ much more manageable.  We prove this conjecture in the case of a \textit{general} quiver $Q$, for any sufficiently large $T$.

\begin{theorem}
\label{TFtheorem}
Let $Q$ be a finite quiver, and let $\dd\in\mathbb{N}^{Q_0}$ be a dimension vector.  Let $\tau\colon \tilde{Q}_1\rightarrow \mathbb{Z}^s$ be a $\tilde{W}$-invariant grading, such that the grading of Example \ref{SVT} is a specialization of $\tau$.  The $\HO_{T\times \Gl_{\dd}}$-module $\HO^{\BM}_{T\times \Gl_{\dd}}(\mu^{-1}_{Q,\dd}(0),\mathbb{Q})$ is torsion free, and the natural map $F$ to the shuffle algebra $\Coha^{\tw}_{\tau,\tilde{Q}}$ is an inclusion of algebras.
\end{theorem}
\begin{remark}
Note that the $\HO_T$-action on $\HO^{\BM}_{T\times \Gl_{\dd}}(\mu_{Q,\dd}^{-1}(0),\mathbb{Q})$ is determined by the restriction of $\tau$ to $\overline{Q}_1$.  Thus, the $\HO_T$ action on all of the cohomology groups in \eqref{shuffDiag} are determined by this restriction.
\end{remark}
\renewcommand*{\proofname}{Proof of Theorem \ref{TFtheorem}}
\begin{proof}
The passage from torsion freeness to all of the other statements of the theorem is as explained in \cite{ScVa12}, so we focus on torsion-freeness.  The proof for this is a modification of \cite[Prop.4.6]{SV17}; the original statement of this result in \cite{SV17}, and its proof, require fixing, we indicate how.
\smallbreak
Firstly, by assumption, $T$ contains two one-dimensional tori $\mathbb{C}^*_1$ and $\mathbb{C}^*_2$, where $\mathbb{C}^*_1$ acts on arrows $a,a^*,\omega_i$, for $a\in Q_1$ and $i\in Q_0$ with weights $1,-1,0$ respectively, and $\mathbb{C}^*_2$ acts with weights $1,0,-1$ respectively.
\sbrk
Let $\Bbbk_i=\HO_{\mathbb{C}^*_i}$, let $I_i\subset \HO_{T\times \Gl_{\dd}}$ be the ideal of functions vanishing on $\mathrm{Lie}(\mathbb{C}^*_i)\subset \mathrm{Lie}(T\times \Gl_{\dd})$, and let $K_i$ be the fraction field of $\Bbbk_i$.  Considering $X(Q)_{\dd}$ as a subvariety of $\mu^{-1}_{Q,\dd}(0)$ via the extension by zero map, $X(Q)_{\dd}$ contains the fixed locus of the $\mathbb{C}^*_1$-action on $\mu^{-1}_{Q,\dd}(0)$, and so the pushforward
\[
(\HO_{T\times \Gl_{\dd}})_{I_1}\cong\HO^{\BM}_{T\times \Gl_{\dd}}(X(Q)_{\dd},\mathbb{Q})_{I_1}\rightarrow \HO^{\BM}_{T\times \Gl_{\dd}}(\mu_{Q,\dd}^{-1}(0),\mathbb{Q})_{I_1}
\]
is an isomorphism by \cite[Thm.6.2]{GKM97}.  It is thus enough to prove that 
\[
\HO^{\BM}_{T\times \Gl_{\dd}}(\mu^{-1}_{Q,\dd}(0),\mathbb{Q})
\]
is $S_1$-torsion-free for $S_1=\HO_{T\times\Gl_{\dd}}\setminus I_1$.  The $\Bbbk_2$-module $\HO^{\BM}_{T\times \Gl_{\dd}}(\mu^{-1}_{Q,\dd}(0),\mathbb{Q})$ is free by Theorem \ref{degThm}, and so the morphism
\begin{equation}
\label{twop}
\HO^{\BM}_{T\times \Gl_{\dd}}(\mu^{-1}_{Q,\dd}(0),\mathbb{Q})\rightarrow \HO^{\BM}_{T\times \Gl_{\dd}}(\mu^{-1}_{Q,\dd}(0),\mathbb{Q})\otimes_{\Bbbk_2} K_2
\end{equation}
is an embedding.  Therefore it is sufficient to show that the right hand side of \eqref{twop} has no $I_1$-torsion.  By two applications of dimensional reduction (see Theorem \ref{dimRedThm}, and the discussion before the proof of Lemma \ref{lem3}), we have the $\Bbbk_2$-linear isomorphisms (leaving out Tate twists/shifts in cohomological degree):
\begin{align*}
\HO^{\BM}_{T\times \Gl_{\dd}}(\mu^{-1}_{Q,\dd}(0),\mathbb{Q})\cong &\HO^{\BM}\!\left({}^{\tau}\Mst(\tilde{Q})_{\dd},\phi_{\Tr(\tilde{W})}\mathbb{Q}\right)\\
\cong &\HO^{\BM}_{T\times \Gl_{\dd}}(\mathcal{C}_{\dd},\mathbb{Q})
\end{align*}
where $\mathcal{C}_{\dd}\subset X(Q^+)_{\dd}$ is the subspace of $Q^+$-modules such that the linear maps assigned to the loops $\omega_i$ define a $\mathbb{C}Q$-module endomorphism.  We define
\[
\mathcal{N}_{\dd}=\mathcal{C}_{\dd}\cap X(Q^+)^{\omega \nnilp}_{\dd}.
\]
Since the torus $\mathbb{C}^*_2$ acts by scaling the loops $\omega_i$, the natural map
\[
\HO^{\BM}_{T\times \Gl_{\dd}}(\mathcal{N}_{\dd},\mathbb{Q})\otimes_{\Bbbk_2} K_2\rightarrow \HO^{\BM}_{T\times \Gl_{\dd}}(\mathcal{C}_{\dd},\mathbb{Q})\otimes_{\Bbbk_2} K_2
\]
is an isomorphism.  So it is sufficient to show that $\HO^{\BM}_{T\times \Gl_{\dd}}(\mathcal{N}_{\dd},\mathbb{Q})$ has no $I_1$-torsion.  
\sbrk
Consider the stratification of the space $\mathcal{N}_{\dd}$ by Jordan types, introduced in the proof of Lemma \ref{lem3}: if $\overline{\pi}=(\pi^{(i)})_{i\in Q_0}$ is a tuple of partitions, with each $\pi^{(i)}$ a partition of $\dd_i$, the stratum $\mathcal{N}_{\overline{\pi}}\subset \mathcal{N}_{\dd}$ is the space for which the Jordan normal form of the operator assigned to $\omega_i$ has blocks with sizes given by $\pi^{(i)}$.  Then it is easy to check that each $\HO^{\BM}(\mathcal{N}_{\overline{\pi}},\mathbb{Q})$ has no $I_1$-torsion, and the claim that $\HO^{\BM}(\mathcal{N}_{\dd},\mathbb{Q})$ has no $I_1$-torsion follows from the long exact sequences in compactly supported cohomology induced by the stratification of $\mathcal{N}_{\dd}$.
\end{proof}
\renewcommand*{\proofname}{Proof}
\begin{remark}
The same proof works with $\mu^{-1}_{Q,\dd}(0)$ replaced by $\mu^{-1}_{Q,\dd}(0)^{\sharp}$ for $\sharp$ any out of $\SN,\SSN,\N$.
\end{remark}
\begin{remark}
\label{cat1}
It is possible for the strata $\mathcal{N}_{\overline{\pi}}$ to have $S_2$-torsion, so we cannot substitute $I_2$ for $I_1$ in the above proof, and merely insist on the inclusion $\mathbb{C}^*_2\subset T$.  Indeed we show in \S\ref{ttyb2} that this (stronger) version of the statement of Theorem \ref{TFtheorem} with (weaker) assumptions is false.
\end{remark}
\begin{remark}
\label{cat2}
For example, let $Q$ be the Jordan quiver, and consider the partition $\pi=(2)$ of the dimension vector $2\in\mathbb{N}^{Q_0}$.  We set $T=\mathbb{C}^*_1\times\mathbb{C}^*_2$, and write $\HO_T=\mathbb{Q}[t_1,t_2]$.  The space $\mathcal{N}_{\pi}$ is smooth, and we claim that $\HO_{T\times\Gl_2(\mathbb{C})}(\mathcal{N}_{\pi},\mathbb{Q})$ has $S_2$-torsion.  There is an equivariant homotopy equivalence $\mathcal{N}_{\pi}\simeq \mathcal{N}'_{\pi}$, where $\mathcal{N}'_{\pi}\subset \mathrm{Mat}_{2\times 2}(\mathbb{C})$ is the subspace of nonzero nilpotent matrices.  We denote by $T'\subset \Gl_2(\mathbb{C})$ the usual maximal torus, and write 
\begin{align*}
\HO_{T'}=&\mathbb{Q}[z_1,z_2]\\
\HO_{\Gl_2(\mathbb{C})}=&\mathbb{Q}[z_1,z_2]^{\mathfrak{S}_2}.
\end{align*}
There is precisely one $\Gl_2(\mathbb{C})$-orbit inside $\mathcal{N}'_{\pi}$.  So there is an isomorphism
\[
\HO_{T\times \Gl_2(\mathbb{C})}(\mathcal{N}'_{\pi},\mathbb{Q})\cong \HO_{A}
\]
where $A$ is the stabiliser subgroup of the nilpotent matrix 
\[
M=\left(\begin{array}{cc}0&1\\0&0\end{array}\right).
\]
We calculate 
\[
\HO_A=\mathbb{Q}[z_1,z_2,t_1,t_2]/(z_1-z_2-t_2)
\]
with the morphism $\HO_{T\times \Gl_2(\mathbb{C})}\rightarrow \HO_A$ given by the natural surjection defined by the composition
\[
\mathbb{Q}[z_1,z_2,t_1,t_2]^{\mathfrak{S}_2}\hookrightarrow\mathbb{Q}[z_1,z_2,t_1,t_2] \twoheadrightarrow \mathbb{Q}[z_1,z_2,t_1,t_2]/(z_1-z_2-t_2),
\]
where $\mathfrak{S}_2$ fixes the variables $t_1,t_2$.  The above factorization corresponds to the composition of inclusions
\[
T''\hookrightarrow T\times T'\hookrightarrow T\times \Gl_2(\mathbb{C})
\]
where $T''\subset T\times T'$ is the 3-torus stabilising $M$, and we are using that $T''$ is homotopic to $A$.  In particular, the action of $\gamma=t_2^2-(z_1-z_2)^2$ on $\HO_A$ is not injective, although $\gamma\in S_2$.  On the other hand, clearly $H_A$ has no $S_1$-torsion.%  This seems to be where the original proof of \cite[Prop.4.6]{SV17} breaks.
\end{remark}

\subsection{Noncommutativity}
\label{ECsec}
Theorem \ref{TFtheorem} enables explicit calculations inside $\Coha_{\tau,\Pi_Q}$.  Furthermore, although (as we have seen in Remarks \ref{cat1} and \ref{cat2}, and will see further, with Proposition \ref{TTYB3}) it is important that we work equivariantly with respect to a sufficiently large torus $T$ in Theorem \ref{TFtheorem}, we will demonstrate in this section how Theorem \ref{TFtheorem} enables us to perform concrete calculations for \textit{trivial} $T$, i.e. in the undeformed preprojective CoHA $\Coha_{\Pi_Q}$.
\subsubsection{}
We use explicit calculations in the algebra $\Coha_{\tau,\widetilde{Q_{\Jor}}}$ to show that $\Coha_{\widetilde{Q_{\Jor}},\tilde{W}}\cong\Coha_{\Pi_{Q_{\Jor}}}$ is noncommutative\footnote{Equivalently, since the entire algebra lives in even cohomological degrees, we show that it is not \textit{super}commutative.}.  Recall that by Theorem \ref{PBW3d}, for an arbitrary (symmetric) quiver $Q$ with potential $W$, there is a PBW isomorphism
\[
\Sym(\DT_{Q,W}\otimes\HO(\B\mathbb{C}^*,\mathbb{Q})_{\vir})\cong \Coha_{Q,W}.
\]
By Theorem \ref{dzBPS} there is an isomorphism of cohomologically graded vector spaces
\[
\DT_{\widetilde{Q_{\Jor}},\tilde{W},d}\cong \mathbb{Q}[3]
\]
so that for each $d\geq 1$ and $e\geq 0$ there is an element $\alpha^{(e)}_d$, of cohomological degree $2e-2$, well defined up to scalar, defined to be the image of $1\otimes u^e$ under the embedding
\[
\DT_{\widetilde{Q_{\Jor}},\tilde{W},d}\otimes \HO(\B\mathbb{C}^*,\mathbb{Q})_{\vir}\hookrightarrow \Coha_{\widetilde{Q_{\Jor}},\tilde{W}}.
\]
\begin{lemma}
\label{TTYB}
The commutator $[\alpha^{(1)}_1,\alpha^{(0)}_1]$ does not vanish.
\end{lemma}
In particular, the algebra $\Coha_{\Pi_{Q_{\Jor}}}$ is noncommutative.
\begin{proof}
We set $Q=Q_{\Jor}$.  Pick $\tau$ as in Example \ref{SVT}, with associated torus $T\cong\mathbb{C}_1^*\times\mathbb{C}_2^*$ in the notation of the proof of Theorem \ref{TFtheorem}.  By Theorem \ref{TFtheorem} the morphism 
\[
\iota \colon\Coha_{\tau,\Pi_Q}\rightarrow \Coha_{\tau,\tilde{Q}}
\]
is an inclusion of algebras.  Write $\Coha'\subset \Coha_{\tau,\tilde{Q}}$ for the image of this inclusion.  Then by Corollary \ref{specCor} there is an isomorphism of algebras
\[
\Coha_{\Pi_Q}\cong \Coha'/(t_1,t_2)\cdot\Coha'.
\]
We write $\Coha_{\tau,\Pi_Q,1}\cong \Coha_{\Pi_Q,1}\otimes \HO_T$, and define $\tilde{\alpha}_1^{(e)}=\alpha_1^{(e)}\otimes 1$.  Then $\iota(\tilde{\alpha}_1^{(e)})=x_1^e\in\mathbb{Q}[x_1,t_1,t_2]$.
\sbrk
First we calculate the commutator in $\Coha_{\tau,\tilde{Q}}$:
\begin{align*}
[x_1,x^0_1]=&(x_1-x_2)(x_2-x_1+t_1)(x_2-x_1+t_2)(x_2-x_1-t_1-t_2)/(x_2-x_1)+\\&(x_2-x_1)(x_1-x_2+t_1)(x_1-x_2+t_2)(x_1-x_2-t_1-t_2)/(x_1-x_2)\\
=&-2t_1t_2(t_1+t_2).
\end{align*}
This element has cohomological degree $-2$.  We claim that the unique nonzero element of cohomological degree less than $-2$ in $\iota(\Coha')$ is $x^0_1\star x^0_1$ (up to scalar).  Firstly, $x^0_1\star x^0_1$ has cohomological degree $-4$, since $x^0_1$ has cohomological degree $-2$.  Secondly, it is indeed nonzero, as we calculate below.  Finally, it follows from e.g. Corollary \ref{go2d} that the stack $\mathcal{C}_2\cong \mathfrak{C}oh_d(\mathbb{A}^2)$ of pairs of commuting $2\times 2$ matrices has a unique irreducible component of (complex) dimension greater than 1, and that component has dimension 2.
\smallbreak
Now we calculate 
\begin{align*}
x^0_1\star x^0_1=&(x_2-x_1+t_1)(x_2-x_1+t_2)(x_2-x_1-t_1-t_2)/(x_2-x_1)+\\&(x_1-x_2+t_1)(x_1-x_2+t_2)(x_1-x_2-t_1-t_2)/(x_1-x_2)\\
=&2(x_1-x_2)^2-2(t_1^2+t_1t_2+t_2^2).
\end{align*}
Thus $[x_1,x^0_1]\notin (t_1,t_2) \cdot (x^0_1\star x^0_1)$ and so $[x_1,x^0_1]\notin (t_1,t_2)\cdot\Coha'$.  It follows that $[\alpha_1^{(0)},\alpha_1^{(1)}]\neq 0$.
\end{proof}
\subsubsection{}
We can in fact be a little more specific regarding the commutator $[\alpha_1^{(1)},\alpha_1^{(0)}]$.
\begin{corollary}
\label{msc}
There exists a nonzero scalar $\lambda\in\mathbb{Q}$ such that
\[
[\alpha_1^{(1)},\alpha_1^{(0)}]=\lambda\alpha_2^{(0)}.
\]
\end{corollary}
\begin{proof}
In \cite{DaMe15b} it is shown that for a general (symmetric) QP $(Q',W')$, the CoHA $\Coha_{\tilde{Q},\tilde{W}}$ is a filtered algebra, for the perverse filtration defined by setting
\[
P^n\!\Coha_{Q',W'}=\HO\!\left(\Msp(Q'),\tau^{\leq n} \p_{*}^{\zeta}\phin{\mathfrak{T}r(W')}\ICS_{\Mst(Q')}(\mathbb{Q})\right).
\]
and the associated graded algebra is supercommutative.  In particular, $[\alpha_1^{(1)},\alpha_1^{(0)}]\in P^3\Coha_{\tilde{Q},\tilde{W}}$, since $\alpha_d^{(i)}\in P^{2i+1}\Coha_{\tilde{Q},\tilde{W}}$; here we have used commutativity of the associated graded object, along with the calculation $3+1-1=3$.

Inspecting \eqref{starInt}, for general (symmetric) quiver $Q'$ with potential $W'$, $P^3\!\Coha_{Q',W'}$ is spanned by 
\begin{enumerate}
\item $(\DT_{Q',W'}[-1])\star (\DT_{Q',W'}[-1])$
\item $\DT_{Q',W'}[-1]$
\item $\DT_{Q',W'}[-3]$.  
\end{enumerate}
So $P^3\!\Coha_{\widetilde{Q_{\Jor}},\tilde{W},2}$ is spanned by $\alpha_1^{(0)}\star \alpha_1^{(0)},\alpha_2^{(0)}, \alpha_2^{(1)}$, which have cohomological degrees $-4,-2,0$ respectively.  The cohomological degree of $[\alpha_1^{(1)},\alpha_1^{(0)}]$ is $-2$.  By Lemma \ref{TTYB}, $[\alpha_1^{(1)},\alpha_1^{(0)}]\neq 0$ and the result follows.
\end{proof}
\subsubsection{} We finish with a positive result regarding $\Coha_{\Pi_{Q_{\Jor}}}$.  A version of this result is to be found in \cite{KaVa19}; the proposition can be proved by observing that the elements $\alpha_{i}^{(n)}$ are primitive for the cocommutative coproduct considered in \cite{KaVa19}.  Using ideas from the proof of Corollary \ref{msc} we give an alternative proof.  Note that, in view of \ref{TTYB}, we do not claim that the Lie bracket on $\hat{\mathfrak{g}}$ is trivial.
\begin{proposition}
The $\mathbb{Q}$-vector space $\hat{\mathfrak{g}}\subset \Coha_{\Pi_{Q_{\Jor}}}$ spanned by the elements $\alpha_{i}^{(n)}$ is closed under the commutator Lie bracket, and there is an isomorphism
\[
\UEA(\hat{\mathfrak{g}})\cong \Coha_{\Pi_{Q_{\Jor}}}.
\]
\end{proposition}
\begin{proof}
Set $Q=Q_{\Jor}$.  The final statement follows from the first statement and the PBW theorem (Theorem \ref{2dint}) for $\Coha_{\Pi_{Q}}$.  For the first statement, we consider the perverse filtration from the proof of Corollary \ref{msc}.  Then $P^l\!\Coha_{\tilde{Q},\tilde{W}}$ has a basis given by monomials $\overline{\alpha}=\alpha_{i_1}^{(n_1)}\cdots\alpha_{i_s}^{(n_s)}$ with 
\[
\sum_{a=1}^s (1+2n_a)\leq l.
\]
On the other hand, the cohomological degree of such an element is given by
\[
\lvert \overline{\alpha}\lvert=\sum_{a=1}^s(2n_a-2).
\]
Set
\[
\beta=[\alpha_{i}^{(e)},\alpha_j^{(f)}].
\]
As in the proof of Corollary \ref{msc} we have
\[
\beta \in P^{2e+2f+1}\Coha_{\Pi_{Q}},\quad\quad\quad \lvert\beta\lvert=2e+2f-4.
\]
So $\beta$ can be written as a linear combination of elements $\overline{\alpha}$ with $\lvert\beta\lvert=\lvert\overline{\alpha}\lvert$ and
\begin{align*}
\sum_{a=1}^s (1+2n_a)&\leq 2e+2f+1\\
3s&\leq 5 &\textrm{since }\lvert\beta\lvert=\lvert\overline{\alpha}\lvert.
\end{align*}
So we find $s= 1$ as required.
\end{proof}
\begin{remark}
The same proof(s) demonstrate that $\Coha^{\mathcal{SN}}_{\Pi_{Q_{\Jor}}}$ is a universal enveloping algebra for some $\hat{\mathfrak{g}}^{\mathcal{SN}}$ satisfying $\hat{\mathfrak{g}}^{\mathcal{SN}}\cong \hat{\mathfrak{g}}[-2]$ (as graded vector spaces).
\end{remark}

\subsection{Torsion}
\label{ttyb2}
Finally we show that torsion-freeness of $\Coha_{\tau,\Pi_Q}$ can fail if we relax the conditions on $\tau$ in Theorem \ref{TFtheorem}.   
\begin{proposition}
\label{TTYB3}
Let $T=\mathbb{C}^*$ be any one of the tori $\mathbb{C}^*_1$, $\mathbb{C}^*_2$, $\mathbb{C}^*_3$ acting with weights $(1,-1,0)$ or $(1,0,-1)$ or $(0,1,-1)$ on the three arrows $a,a^*,\omega$ of $\widetilde{Q_{\Jor}}$.  Let $\sharp$ be one of $\emptyset,\N,\SN,\SSN$, chosen so that the pushforward morphism $\Phi_1\colon \Coha^{\sharp}_{\tau,\Pi_{Q_{\Jor}},1}\rightarrow \Coha_{\tau,\Pi_{Q_{\Jor}},1}$ is injective.  Then the $\HO_{T\times \Gl_2(\mathbb{C})}$-module $\Coha^{\sharp}_{\tau,\Pi_{Q_{\Jor}},2}=\HO^{\BM}_{T\times \Gl_2(\mathbb{C})}(\mu^{-1}_{Q_{\Jor},2}(0)^{\sharp},\mathbb{Q})$ is not torsion free.  The natural map
\begin{equation}
\label{HWG}
\Phi\colon \Coha^{\sharp}_{\tau,\Pi_{Q_{\Jor}}}\rightarrow\Coha_{\tau,\widetilde{Q_{\Jor}}}
\end{equation}
from the cohomological Hall algebra to the shuffle algebra is not injective.
\end{proposition}
As in Theorem \ref{TFtheorem}, the $\HO_{T\times \Gl_d(\mathbb{C})}$-action on $\Coha^{\sharp}_{\tau,\Pi_{Q_{\Jor}},2}$ is determined by the weights assigned to the arrows $a$ and $a^*$, i.e. the relevant weights are given by $(1,-1)$, $(1,0)$ and $(0,1)$ respectively.
\smallbreak
For $\sharp=\emptyset,\N,\SN,\SSN$ we have that $\mathcal{M}(\widetilde{Q_{\Jor}})_1^{\sharp}=\mathbb{A}^i$ for $i=3,1,2,2$ respectively, and the condition on $\Phi_1$ is just that the equivariant Euler class $E$ of the normal bundle of the inclusion $\mathbb{A}^i\hookrightarrow \mathbb{A}^3=\mathcal{M}(\widetilde{Q_{\Jor}})_1$ is nonzero.
\renewcommand*{\proofname}{Proof of Proposition \ref{TTYB3}}
\begin{proof}
Set $Q=Q_{\Jor}$.  Torsion-freeness is equivalent to the statement that the morphism 
\[
\HO^{\BM}_{T\times \Gl_2(\mathbb{C})}(\mu^{-1}_{Q,2}(0)^{\sharp},\mathbb{Q})\rightarrow \HO^{\BM}_{T\times \Gl_2(\mathbb{C})}(X(\overline{Q})_2,\mathbb{Q})
\]
is injective.  We start with the $\sharp=\emptyset$ case.
\sbrk
The commutator map
\[
[\cdot,\cdot]\colon \Coha_{\tau,\Pi_{Q},1}\otimes \Coha_{\tau,\Pi_{Q},1}\rightarrow \Coha_{\tau,\Pi_{Q},2}
\]
is nonzero, since by Lemma \ref{TTYB} it is nonzero after tensoring with $(\HO_T/\HO_T^{\geq 2})$.  Explicitly, writing \begin{align*}
\Coha_{\tau,\Pi_{Q},1}\cong&\HO(\mathbb{A}^3,\mathbb{Q})\otimes \HO(\B\mathbb{C}^*,\mathbb{Q})\otimes\HO(\B T,\mathbb{Q})\\
\cong& \mathbb{Q}\otimes \mathbb{Q}[u]\otimes\mathbb{Q}[t]
\end{align*}
we have seen that $[1,u]\neq 0$ (where the commutator is taken in $\Coha_{\tau,\Pi_{Q}}$ and lands in $\Coha_{\tau,\Pi_{Q},2}$).  For $T=\mathbb{C}^*_l$ with $l=1,2,3$, the shuffle algebra $\Coha_{\tau,\widetilde{Q}}$ is commutative; e.g. for the $\mathbb{C}^*_2$ case, by inspection of \eqref{SAP} and the equations
\begin{align*}
E_{a^*}(z)=&z\\
E_{a}(z)=&-E_{\omega}(-z)
\end{align*}
for $a$ the unique arrow in $Q$ it follows that $\Coha_{\tau,\widetilde{Q}}$ is commutative.  So 
\[
\Phi_2(\textrm{Im}([\cdot,\cdot]))=0
\]
and $\Phi_2$ is not injective, proving both parts of the proposition.
\smallbreak
Now let $\sharp\neq \emptyset$, and $E\neq 0$.  Then the image of the algebra homomorphism $\Coha^{\sharp}_{\tau,\Pi_{Q_{\Jor}}}\rightarrow \Coha_{\tau,\Pi_{Q_{\Jor}}}$, restricted to $\Coha^{\sharp}_{\tau,\Pi_{Q_{\Jor}},1}$, contains the elements $E,Eu$.  Then (in $\Coha_{\tau,\Pi_{Q_{\Jor}},2}$) we have
\begin{align*}
[E,Eu]=&E^2[1,u]\\
\neq &0
\end{align*}
since $\Coha_{\tau,\Pi_{Q_{\Jor}},2}$ is free as a $\mathbb{Q}[t]$-module.  In particular,
\[
[\cdot,\cdot]\colon \Coha_{\tau,\Pi_{Q},1}^{\sharp}\otimes \Coha^{\sharp}_{\tau,\Pi_{Q},1}\rightarrow \Coha_{\tau,\Pi_{Q},2}^{\sharp}
\]
is nonzero, and the proof continues as in the $\sharp=\emptyset$ case.
\end{proof}
\renewcommand*{\proofname}{Proof}
The calculation of the commutator $[x_1,x^0_1]$ in the proof of Lemma \ref{TTYB} offers hope that $\mathbb{C}^*_1,\mathbb{C}^*_2,\mathbb{C}^*_3$ are the only three ``bad'' 1-dimensional tori, i.e. those for which Proposition \ref{TTYB3} holds.
\bibliographystyle{amsplain}
\bibliography{SPP}

\vfill

\textsc{\small B. Davison: School of Mathematics, University of Edinburgh}\\
\textit{\small E-mail address:} \texttt{\small ben.davison@ed.ac.uk}\\
\\

\end{document}